%% file: main.tex
\documentclass[11pt, a4paper, twoside]{amsart}

\input{style}
\input{commands}

\graphicspath{{figures/}}

\newcolumntype{?}{!{\vrule width 1.5pt}}
\usepackage{subcaption}
\usepackage{algpseudocode}
\usepackage{aligned-overset}
\usepackage{setspace}
\usepackage{todonotes}

\usepackage{enumerate}
\usepackage[square, numbers]{natbib}
\usepackage{bbm}
\title[Learning in Inverse Optimization]{Learning in Inverse Optimization: Incenter Cost, Augmented Suboptimality Loss, and Algorithms}

\author{Pedro Zattoni Scroccaro, Bilge Atasoy and Peyman Mohajerin Esfahani}%
\thanks{The authors are with the Delft Center for Systems and Control ({\tt\small \{P.ZattoniScroccaro, P.MohajerinEsfahani\}@tudelft.nl}) and the Department of Maritime and Transport Technology ({\tt\small B.Atasoy@tudelft.nl}), Delft University of Technology, Delft, The Netherlands. This work is partially supported by the ERC grant TRUST-949796.}

\begin{document}
\maketitle

\begin{abstract}
In Inverse Optimization (IO), an expert agent solves an optimization problem parametric in an exogenous signal. From a learning perspective, the goal is to learn the expert's cost function given a dataset of signals and corresponding optimal actions. Motivated by the geometry of the IO set of consistent cost vectors, we introduce the ``incenter'' concept, a new notion akin to circumcenter recently proposed by Besbes et al. (2023). Discussing the geometric and robustness interpretation of the incenter cost vector, we develop corresponding tractable convex reformulations, which are in contrast with the circumcenter, which we show is equivalent to an intractable optimization program. We further propose a novel loss function called \textit{Augmented Suboptimality Loss} (ASL), a relaxation of the incenter concept for problems with inconsistent data. Exploiting the structure of the ASL, we propose a novel first-order algorithm, which we name \textit{Stochastic Approximate Mirror Descent}. This algorithm combines stochastic and approximate subgradient evaluations, together with mirror descent update steps, which is provably efficient for the IO problems with discrete feasible sets with high cardinality. We implement the IO approaches developed in this paper as a Python package called InvOpt. Our numerical experiments are reproducible, and the underlying source code is available as examples in the InvOpt package.
\end{abstract}


\section{Introduction}\label{sec:introduction}

In Inverse Optimization (IO) problems, our goal is to model the behavior of an expert agent, which given an exogenous signal, returns a response action. The underlying assumption of IO is that to compute its response, the expert agent solves an optimization problem parametric in the exogenous signal. We assume to know the constraints imposed on the expert, but not its cost function. Therefore, our goal is to model the cost function being optimized by the expert, using examples of exogenous signals and corresponding expert response actions. For example, consider the problem of learning how to route vehicles using examples from experienced drivers. Given a set of customer demands (exogenous signal), the experienced driver chooses a certain vehicle route to serve the customers (expert response). Assuming the drivers solve some kind of route optimization problem to compute their routes, where their cost function depends on the drivers' evaluation of how costly it is to drive from one customer to another, one could look at this problem as an IO problem. Thus, one could use IO tools to learn the cost function being used by these drivers, i.e., we learn the preferences of the expert drivers when driving from a certain customer to another. IO tools have also been used in many other application areas, for instance, healthcare problems \citep{chan2014generalized, chan2022inverse}, modeling of consumer behavior \citep{bertsimas2015data}, transportation problems \citep{burton1992instance, chow2012inverse}, portfolio selection \citep{bertsimas2012inverse}, network design \citep{farago2003inverse}, forecast of electricity prices \citep{saez2017short} and control systems \citep{akhtar2021learning}. For more examples of applications of IO, we refer the reader to the recent review paper \citep{chan2021inverse} and references therein.

The literature on IO can be roughly divided into so-called \textit{classical} and \textit{data-driven} IO. In classical IO, the goal is usually to find a cost function that renders a \textit{single} signal-response pair optimal, that is, a cost function under which the observed response is optimal. A direct approach to modeling this scenario leads to a bi-level optimization problem, thus, much of the early IO literature focuses on reformulating this problem into a single-level tractable program \citep{ahuja2001inverse}. This idea has been extended to different classes of optimization problems, such as conic and integer programs \citep{heuberger2004inverse, iyengar2005inverse, ahmed2005inverse, wang2009cutting, schaefer2009inverse}. On the other hand, data-driven IO usually deals with problems with multiple signal-response examples, and it is not necessarily assumed that there exists a cost function consistent with all signal-response data. In this scenario, the IO problem can be viewed as a supervised learning problem with multivariate output where the IO model forms a hypothesis class. In this view, one minimizes a training loss function to find a good "fit" to the input (response), output (optimizer) data. With this in mind, much of the data-driven IO literature focuses on the choice of the training loss function, particularly because the usual supervised learning losses lead to nonconvex programs \citep{aswani2018inverse}. Examples of such losses are the \textit{KKT loss} \citep{keshavarz2011imputing}, \textit{first-order loss} \citep{bertsimas2015data}, \textit{predictability loss} \citep{aswani2018inverse}, and the \textit{suboptimality loss} \citep{esfahani2018data}. It is worth noting that in the standard IO setting, directly learning the cost function (i.e., regression with signal-pair as the input and the cost function value as the scaler output) is typically not an option since the available data contains only the signal-response pair and not the value of the unknown cost function. Recently, in a framework called ``predict, then optimize'', \cite{elmachtoub2022smart} consider this additional information, and proposes another loss function coined as SPO. Other concepts have also been investigated in the context of IO problems, such as goodness-of-fit \citep{chan2019inverse}, robustness against misspecification \citep{ghobadi2018robust}, and learning constraints instead of cost functions \citep{ghobadi2021inferring}.

Recently, \cite{besbes2023contextual} show that given a dataset of expert examples and the set $\mathbb{C}$, which is defined as the set of cost vectors consistent with the dataset, the optimal solution to the IO problem, for a certain regret performance measure, is the so-called \textit{circumcenter} of the set $\mathbb{C}$. To the best of our knowledge, this is the first minimax regret result for IO problems. To derive this result, the authors exploit the geometry of IO problems and provide strong regret guarantees. The general optimization program associated with computing the circumcenter vector, however, turns out to contain intractable problem instances. Moreover, when the IO dataset is inconsistent with a single cost function, it is not clear how one could generalize the circumcenter concept to solve the IO problem. These shortcomings motivate us to propose a new concept called ``incenter" to select an IO cost vector from the set of consistent costs. In light of this, the main contributions of this work are summarized as follows:
\begin{enumerate}[(i)]
    \item \textbf{(Geometric \& robustness interpretations)} Motivated by the geometry of the set of consistent cost vectors, we introduce the concept of an incenter (Definition \ref{def:incenter}), and provide insights into its geometry (Figure \ref{fig:centers}) and robustness (Remark \ref{remark:robustness}) in comparison with the \textit{circumcenter} concept from \citep{besbes2023contextual}.

    \item \textbf{(Convex reformulations \& tractability)} We develop tractable convex reformulations of the incenter (Theorem \ref{theo:incenter_reformulation}, Corollary \ref{coro:incenter_convex_reformulation}), along with a geometric interpretation of these characterizations (Figure~\ref{fig:alternative_incenter}). This may be of particular interest as the corresponding circumcenter concept is equivalent to an intractable optimization program (Theorem~\ref{theo:np_hard_circumcenter}). To establish this intractability result, we draw connections between the circumcenter/incenter concepts and the well-known problem of extremal volume balls (Remark~\ref{remark:extremal_volume}). Moreover, since the problem of extremal volume balls is a special case of the problem of extremal volume ellipsoids, this connection allows us to derive ``ellipsoidal'' versions of the incenter and circumcenter concepts (e.g., Eq. \eqref{eq:ellipsoidal_incenter}), which performs well in our numerical experiments (Section \ref{sec:consistent_out}).
    
    \item \textbf{(Generalization to inconsistent data \& Augmented Suboptimality loss)} Generalizing to the inconsistent data setting, we propose a novel loss function for IO problems, which we name \textit{Augmented Suboptimality Loss} (\ref{eq:ASL}) (Definition \ref{def:ASL}). This loss can be interpreted as a \textit{relaxation} of the incenter concept, to handle IO problems with inconsistent data (Eq. \eqref{eq:incenter_gen_inconsistent}). In special cases, this formulation of the IO problem can be shown to be equivalent to the so-called Structured SVM formulation of structured prediction problems, revealing a connection between IO and structure prediction (Remark \ref{remark:connections_structured}). We further propose a general convex reformulation of the ASL for IO problems with mixed-integer feasible sets (Theorem \ref{theo:MI_reformulation} and Corollary \ref{coro:MI_reformulation_linear}), which generalizes several reformulations from the literature (Remark \ref{remark:generality}).
    
    \item \textbf{(Tailored first-order algorithm: Stochastic Approximate Mirror Decent)} Motivated by the structure of IO loss functions, we propose a novel first-order algorithm, which we name \textit{Stochastic Approximate Mirror Descent} (\ref{alg:SAMD}). This algorithm exploits the finite sum structure of the IO problem to compute stochastic gradients, uses approximate evaluating IO loss functions to compute approximate subgradients, and uses mirror descent update steps for problems with favorable geometry (Sections \ref{sec:mirror}, \ref{sec:stochastic_subgradients}, and \ref{sec:approximate_subgradients}). We prove convergence rates for this algorithm (Proposition \ref{prop:samd_rate}) and show how the components of this algorithm can be tailored to exploit the structure of IO loss functions (Algorithm \ref{alg:IO_SAMD}).
    
    \item \textbf{(Inverse Optimization Python package)}
    We implement the IO approaches developed in this paper as a Python package called InvOpt \citep{zattoniscroccaro2023invopt}. This is a Python package developed to solve general IO problems. Our numerical experiments are reproducible, and the underlying source code is available as examples in the InvOpt package.
\end{enumerate}

The rest of the paper is organized as follows. In Section \ref{sec:problem} we formalize the IO problem. In Section \ref{sec:center} we present approaches based on the incenter concept. In Section \ref{sec:ASL}, we define the Augmented Suboptimality Loss and show how it can be used to tackle IO problems with inconsistent data and mixed-integer feasible sets. In Section \ref{sec:first-order}, we present first-order algorithms tailored to IO problems. In Section \ref{sec:numerical} we report our numerical results. In Section \ref{sec:conclusion}, we conclude the paper and discuss future research directions. In Appendix \ref{app:theoretical} we discuss the theoretical properties of the ASL, in Appendix \ref{app:continuous} we present concise reformulations for some special IO problems with continuous feasible sets, in Appendix \ref{app:proofs} we present proofs, and in Appendix \ref{app:further_numerical} we present further numerical results.

\textbf{Notation.} The trace, range space, and Moore-Penrose inverse of a matrix $Q \in \mathbb{R}^{m \times m}$ are denoted as $\text{Tr}(Q)$, $\mathcal{R}(Q)$ and $Q^\dagger$, respectively. For two symmetric matrices $Q,R \in \mathbb{R}^{m \times m}$, $Q \succcurlyeq R$ means $Q-R$ is positive semidefinite. The Euclidean inner product between two vectors $x,y \in \mathbb{R}^m$ is denoted by $\inner{x}{y}$. The set of vectors with nonnegative components is denoted as $\mathbb{R}^n_+ \coloneqq \{x \in \mathbb{R}^n : x \geq 0\}$. The cardinality, complement, interior, and convex hull of a set $A$ are denoted as $|A|$, $\overline{A}$, $\text{int}(A)$ and $\text{conv}(A)$, respectively. The set of integers from $1$ to $N$ is denoted as $[N]$. A set of indexed values is compactly denoted by $\{x_i\}_{i=1}^N \coloneqq \{x_1,\ldots,x_N\}$. The euclidean projection of $x$ onto a set $A$ is defined as $\Pi_A(x) \coloneqq \argmin_{y \in A}\|y - x\|_2$. The euclidean angle between two nonzero vectors $\theta$ and $\tilde{\theta}$ is defined as $a(\theta, \tilde{\theta}) \coloneqq \text{arccos} \inner{\theta}{\tilde{\theta}}/(\|\theta\|_2\|\tilde{\theta}\|_2)$. For vectors $x,y \in \mathbb{R}^m$, $x \odot y$ and $\exp(x)$ mean element-wise multiplication and element-wise exponentiation, respectively. Moreover, whenever the arguments of $\exp(\cdot)$ or $\log(\cdot)$ are matrices, they should be interpreted as the usual matrix exponentiation and matrix logarithm, respectively. The vector $\text{vec}(Q) \in \mathbb{R}^{pq}$ denotes the vectorization of the matrix $Q \in \mathbb{R}^{p\times q}$ formed by stacking the columns of $Q$ into a single column vector. The symbol $\otimes$ denotes the Kronecker product. The vector $e_j \in \mathbb{R}^n$ denotes the vector of zeros except for the $j$'th element which is equal to one. The vector of ones is represented by $\mathbbm{1}$.


\section{Problem Description}
\label{sec:problem}

Let us begin by formalizing the IO problem. Let $s \in \mathbb{S}$ be an exogenous signal, and $\mathbb{S}$ be the signal space. The expert agent is assumed to solve the parametric \textit{Forward Optimization Problem}
\begin{equation}
\label{eq:expert_problem}
    \min_{x \in \mathbb{X}(s)} \ F(s,x),
\end{equation}
where $\mathbb{X}(s)$ is the feasible set, and $F: \mathbb{S} \times \mathbb{X} \to \mathbb{R}$ is the expert's cost function, where we define $\mathbb{X} \coloneqq \bigcup_{s \in \mathbb{S}} \mathbb{X}(s)$.  Recall that the cost $F$ is unknown to us, and we only have access to a dataset of $N$ pairs of exogenous signals and respective expert optimal decisions $\widehat{\mathcal{D}} \coloneqq \{(\hat{s}_i, \hat{x}_i)\}_{i=1}^N$, that is,
\begin{equation*}
    \hat{x}_i \in \argmin_{x \in \mathbb{X}(\hat{s}_i)} \ F(\hat{s}_i,x), \quad \forall i \in [N].
\end{equation*}
We use the notation ``$\hat{\cdot}$'' to indicate signal-response data. Using this data, our goal is to learn a cost function that when minimized, reproduces the expert's decisions as well as possible. 

Naturally, we can only search for functions in a restricted function space. In this work, we consider the following parametric \textit{hypothesis space}:
\begin{equation}
\label{eq:hypothesis_class}
    \mathcal{F}_\phi \coloneqq \big\{ \inner{\theta}{\phi(\cdot,\cdot)} : \theta \in \mathbb{R}^p \big\},
\end{equation}
where $\theta \in \mathbb{R}^p$ is called the \textit{cost vector} and $\phi: \mathbb{S} \times \mathbb{X} \to \mathbb{R}^p$ is called \textit{feature mapping}, which maps a signal-response pair $(s,x)$ to a feature vector $\phi(s,x) \in \mathbb{R}^p$. In practice, choosing a suitable mapping $\phi$ is part of the modeling of the IO problem. To exemplify the generality of this hypothesis space, the next example shows how to model a quadratic function as a member of $\mathcal{F}_\phi$. Other utility functions from the literature, for instance, the CES function and the Cobb-Douglas function can also be easily fitted into our framework \cite[Appendix A]{chen2020online}.

\begin{example}[Quadratic hypothesis]
\label{ex:hypothesis_example}
Consider a quadratic function
$$
F(s,x) = \inner{x}{Q_{xx}x} + \inner{x}{Q_{xs}s} + \inner{x}{q},
$$
parametrized by $(Q_{xx},Q_{xs},q) \in \mathbb{R}^{n \times n} \times \mathbb{R}^{n \times m} \times \mathbb{R}^n$. Using the identity $\emph{vec}(AXB) = \inner{\emph{vec}(X)}{B \otimes A^\top}$, we can rewrite this quadratic function as
$$
F(s,x) = \inner{\emph{vec}(Q_{xx})}{x \otimes x} + \inner{\emph{vec}(Q_{xs})}{s \otimes x} + \inner{x}{q} = \inner{\theta}{\phi(s,x)},
$$
where
$$
\theta \coloneqq 
\begin{bmatrix}
\emph{vec}(Q_{xx}) \\
\emph{vec}(Q_{xs}) \\
q
\end{bmatrix}
\quad \text{and} \quad
\phi(s,x) \coloneqq
\begin{bmatrix}
x \otimes x \\
s \otimes x \\
x
\end{bmatrix}.
$$
\end{example}
Thus, our IO goal is to find a cost vector $\theta$ such that when solving the optimization problem
\begin{equation} \label{eq:IO_problem}
    \min_{x \in \mathbb{X}(s)} \ \inner{\theta}{\phi(s,x)},
\end{equation}
we can reproduce (or in some sense approximate) the action the expert would have taken, given the same signal $s$. An important object in IO problems is $\mathbb{C}$, the set of cost vectors consistent with the data $\widehat{\mathcal{D}}$.
\begin{definition}[Consistent cost vectors]
\label{def:consistent}
Given a dataset $\widehat{\mathcal{D}} = \{(\hat{s}_i, \hat{x}_i)\}_{i=1}^N$, feature mapping $\phi$ and feasible set $\mathbb{X}$, the \emph{set of consistent cost vectors} is defined as
\begin{equation}
    \label{eq:consistent}
    \mathbb{C} \coloneqq \big\{\theta \in \mathbb{R}^p : \inner{\theta}{\phi(\hat{s}_i,\hat{x}_i) - \phi(\hat{s}_i,x_i)} \leq 0, \ \forall x_i \in \mathbb{X}(\hat{s}_i), \ \forall i \in [N] \big\}.
\end{equation}
\end{definition}
In other words, $\mathbb{C}$ corresponds to the set of cost vectors $\theta$ that render the expert responses $\hat{x}_i$'s optimal, i.e., $\hat{x}_i \in \argmin_{x_i \in \mathbb{X}(\hat{s}_i)} \inner{\theta}{\phi(\hat{s}_i,x_i)}$, for all $i \in [N]$. Geometrically, the set $\mathbb{C}$ is a cone, defined by the intersection of half-spaces defined by hyperplanes that pass through the origin.


\section{Incenter Cost Vector}
\label{sec:center}

From Definition \eqref{def:consistent}, one can observe that a trivial element in the set $\mathbb{C}$ is $\theta = 0$. However, this trivial choice is not really of interest since it yields any response $x \in \mathbb{X}(s)$ to the signal s an optimal solution, hence no differentiation between different responses. Therefore, when $\mathbb{C}$ contains other elements than $\theta = 0$, that is, the data is consistent with the hypothesis space in the sense that the expert's true cost is $F(s,x) = \inner{\theta^\star}{\phi(s,x)}$ for some nonzero element $\theta^\star$, it is natural to restrict the search space to $\mathbb{C} \setminus \{0\}$. It is also worth noting that a similar issue concerns the feature map $\phi$. Namely, from a modeling perspective, it is important to consider a feature class where $\phi(s,x)$ is not identically zero in the second argument over the set $\mathbb{X}(s)$ for typical signal variables $s \in \mathbb{S}$. This assumption, together with excluding the trivial cost vector $\theta = 0$ from the search space, is necessary. Thus, in this section, we assume $\mathbb{C} \setminus \{0\} \neq \emptyset$ and that $\phi$ is not trivially zero, and we discuss different ways to choose a cost vector from the set $\mathbb{C}$. A straightforward way to do this is by solving a feasibility optimization problem, which we call \textit{feasibility program}:
\begin{equation}
\label{eq:feasibility}
    \begin{aligned}
    \min_{\theta} \quad & 0 \\
    \text{s.t.} \ \quad & \inner{\theta}{\phi(\hat{s}_i,\hat{x}_i) - \phi(\hat{s}_i,x_i)} \leq 0 \quad \quad \forall x_i\in\mathbb{X}(\hat{s}_i), \ \forall i \in [N] \\
    & \|\theta\| = 1.
    \end{aligned}
\end{equation}
Problem \eqref{eq:feasibility} simply translates the idea that we wish to find some nonzero $\theta \in \mathbb{C}$. Notice that the normalization constraint $\|\theta\| = 1$ excludes the trivial solution $\theta = 0$, and it is nonrestrictive in the sense that $\argmin_{x \in \mathbb{X}(s)} \inner{\theta}{\phi(s,x)} = \argmin_{x \in \mathbb{X}(s)} \inner{\alpha \theta}{\phi(s,x)}$, for any $\alpha>0$. By solving the IO problem using the feasibility program \eqref{eq:feasibility}, we implicitly suggest that any cost vector $\theta \in \mathbb{C} \setminus \{0\}$ is an equally good solution to the IO problem. However, one could argue that there are more principled ways to choose a cost vector from the set of consistent vectors $\mathbb{C}$. 

\subsection{Geometry, robustness, and tractability}

Considering two cost vectors $\theta$ and $\theta^\star$ with their corresponding responses defined as 
\begin{subequations}
\label{eq:regret}
\begin{equation}
\label{eq:regret1}
    x^\theta \in \argmin_{x \in \mathbb{X}(\hat{s})} \inner{\theta}{\phi(\hat{s},x)}  \quad \text{and} \quad x^\star \in \argmin_{x \in \mathbb{X}(\hat{s})} \inner{\theta^\star}{\phi(\hat{s},x)}, 
\end{equation}
the authors in \citep{besbes2023contextual} define the regret
\begin{equation}
    \label{eq:regret2}
    R(\theta, \theta^\star) \coloneqq \langle \theta^\star , \phi(\hat{s}, x^\theta) - \phi(\hat{s}, x^\star) \rangle.
\end{equation}
\end{subequations}
The regret~\eqref{eq:regret2} is a distance function between two cost vectors that measures the difference via the impact of their respective optimizers~$x^\theta$ and $x^\star$ evaluated through the second argument, here the ground truth~$\theta^\star$. As shown in \citep{besbes2023contextual}, the worst-case optimal cost vector $\theta$ for the regret $R(\theta, \theta^\star)$ is the so-called \textit{circumcenter} of $\mathbb{C}$, as defined in the following.
\begin{definition}[Circumcenter \citep{besbes2023contextual}]
\label{def:circumcenter}
Let $\mathbb{C}$ be a nonempty set. A \emph{circumcenter} of $\mathbb{C}$ is defined as
\begin{equation}
\label{eq:circumcenter}
    \theta^{\emph{c}} \in \argmin_{\theta \neq 0} \max_{\substack{\tilde{\theta} \in \mathbb{C} \\ \tilde{\theta} \neq 0}} \ a(\theta, \tilde{\theta}).
\end{equation}
\end{definition}
In the original definition of \cite{besbes2023contextual}, the non-zero constraints over the cost vectors are enforced via $\|\theta\|_2 = \|\tilde{\theta}\|_2 = 1$, which does not change the circumcenter since the angle in the objective is scale invariant. Geometrically, $\theta^\text{c}$ can be interpreted as a point in the axis of the revolution cone with the smallest aperture angle containing $\mathbb{C}$. Informally speaking, the circumcenter approach chooses a cost vector by looking for the vector in $\mathbb{C}$ furthest away (in terms of the angle) from the ``corners'' of $\mathbb{C}$ (more precisely, from the extreme rays of $\mathbb{C}$). However, in this study we identify the following three possible shortcomings when using the circumcenter concept to choose an IO model:
\begin{enumerate}[(i)]
    \item \textbf{Computational tractability:} Computing the circumcenter, in general, turns out to be computationally intractable.
    \item \textbf{Robustness interpretation:} While the circumcenter is, by definition, robust to an adversarial (i.e., worst-case) ground truth data-generating model, the challenge in practice is often related to noisy data.
    \item \textbf{Inconsistent data:} The circumcenter is only defined for problems with consistent data (i.e., nonempty $\mathbb{C}$) and, to the best of our knowledge, cannot be straightforwardly extended to IO problems with inconsistent data.
\end{enumerate}

In the rest of this section, we first formalize the tractability result (i), and then introduce a new notion that addresses these possible shortcomings related to circumcenter. 

\begin{theorem}[Intractability of circumcenter]
\label{theo:np_hard_circumcenter}
Assume $\mathbb{C}$ is a nonempty polyhedral cone. Then, solving 
the inner maximization of \eqref{eq:circumcenter} for any $\theta$
is equivalent to maximizing a quadratic function over a polytope, which is NP-hard.
\end{theorem}
\begin{definition}[Incenter]
\label{def:incenter}
Let $\mathbb{C}$ be a nonempty set. An \emph{incenter} of $\mathbb{C}$ is defined as
\begin{equation}
    \label{eq:incenter}
    \theta^{\emph{in}} \in \argmax_{\theta \neq 0} \min_{\substack{\tilde{\theta} \in \overline{\emph{int}(\mathbb{C})} \\ \tilde{\theta} \neq 0}} \ a(\theta, \tilde{\theta}).
\end{equation}
\end{definition}
We will show how the incenter problem \eqref{eq:incenter} can be reformulated as a tractable convex optimization problem (as opposed to the NP-hardness of the circumcenter), and by relaxing the reformulation, it can be straightforwardly extended to handle problems with inconsistent data. Let us first elaborate more on the intuition behind the incenter vector in Definition~\ref{def:incenter} compared to the circumcenter vector in Definition~\ref{def:circumcenter}. Geometrically, an incenter of $\mathbb{C}$ can be interpreted as a vector furthest away (in terms of the angle) from the \textit{exterior} of $\mathbb{C}$, or in other words, the vector in $\mathbb{C}$ furthers away from the boundary of $\mathbb{C}$. To provide insight into these different centers, we visualize them in a simple three-dimensional space in Figure~\ref{fig:centers}. Interestingly, these notions also have robustness interpretations, which are different from existing learning models in which the robustness is introduced explicitly via regularization or distributional robust approaches \citep{esfahani2018data}.

\begin{figure}
\centering
    \begin{subfigure}[t]{.4\linewidth}
    \centering
    \includegraphics[width = \linewidth]{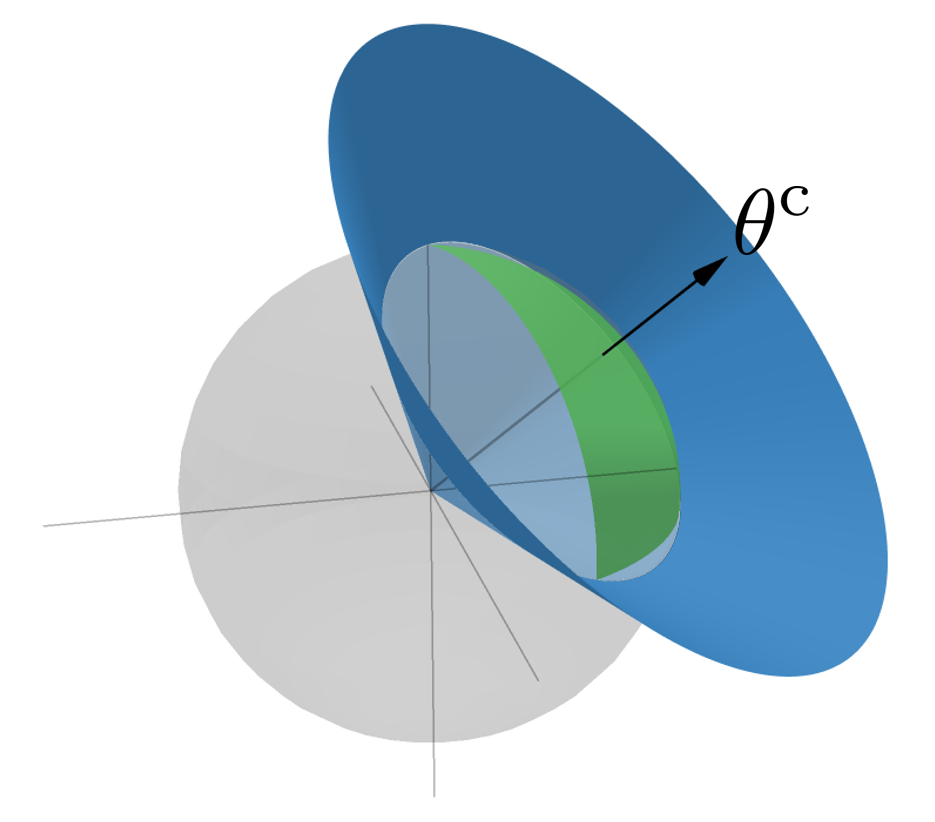}
    \caption{Circumcenter $\theta^\text{c}$ defined in \eqref{eq:circumcenter}.}
    \label{fig:circumcenter}
    \end{subfigure}
    \begin{subfigure}[t]{.4\linewidth}
    \centering
    \includegraphics[width = \linewidth]{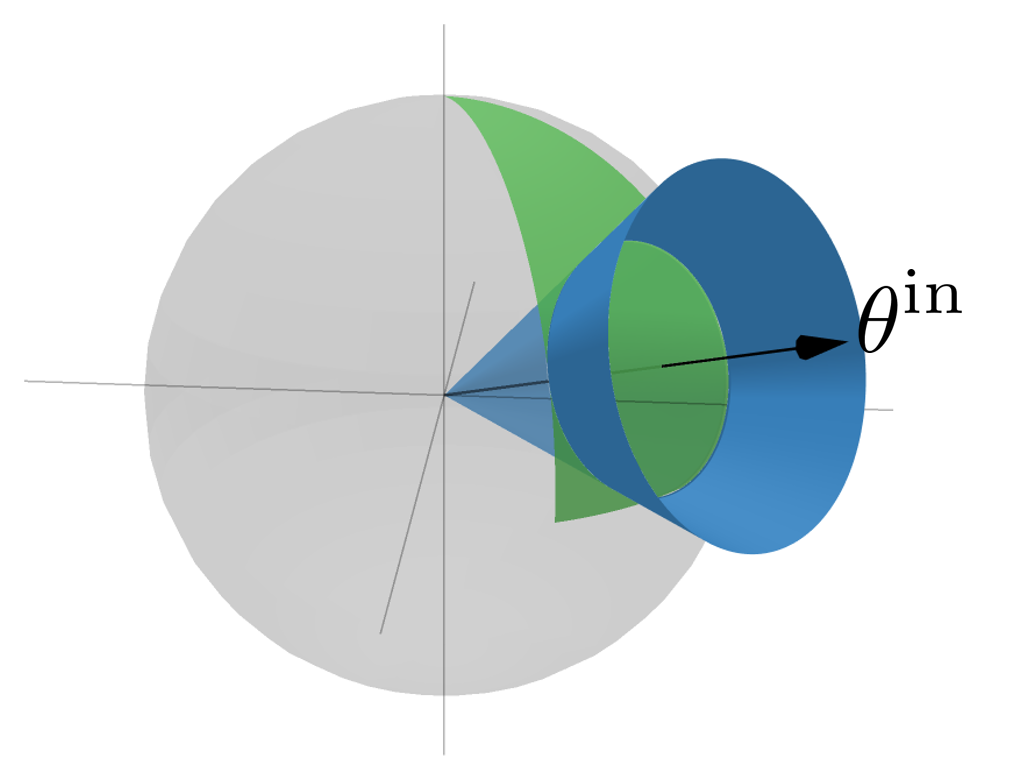}
    \caption{Incenter $\theta^\text{in}$ defined in \eqref{eq:incenter}.}
    \label{fig:incenter}
    \end{subfigure}
\caption{Geometrical visualization of the circumcenter and incenter vectors. The green regions are the intersection of $\mathbb{C}$ with the sphere and the blue cones are revolution cones with an aperture angle equal to the optimal value of \eqref{eq:circumcenter} and \eqref{eq:incenter}.}
\label{fig:centers}
\end{figure}

\begin{remark}[Robustness: circumcenter vs. incenter]
\label{remark:robustness}
A circumcenter $\theta^{\emph{c}}$ in \eqref{eq:circumcenter} can be interpreted as the vector most robust to an adversary with the authority to choose the true data-generating cost vector in $\mathbb{C}$ to be furthest away from our learned model, which is a notion captured by the regret performance measure introduced in \citep{besbes2023contextual}. On the other hand, the incenter $\theta^{\emph{in}}$ \eqref{eq:incenter} can be viewed as the vector furthest away from the boundary of $\mathbb{C}$. Since each facet of $\mathbb{C}$ is determined by a signal-response $(\hat{s}_i, \hat{x}_i)$ pair in the dataset of the IO problem, the incenter vector can be interpreted as the cost vector most robust to perturbations of the training dataset (i.e., perturbations of the facets of $\mathbb{C}$). 
\end{remark}

Let us provide some additional details supporting the robustness interpretation in Remark \ref{remark:robustness}. To formalize this robustness notion, recall the definition of the set of consistent cost vectors $\mathbb{C} = \left\{\theta \in \mathbb{R}^p : \left\langle \theta, \Delta(\hat{s}_i, \hat{x}_i, x) \right\rangle \leq 0, \ \forall x \in \mathbb{X}(\hat{s}_i), \ \forall i \in [N] \right\}$, where we define $\Delta(\hat{s}_i, \hat{x}_i, x) \coloneqq (\phi(\hat{s}_i,\hat{x}_i) - \phi(\hat{s}_i,x_i)) / \| \phi(\hat{s}_i,\hat{x}_i) - \phi(\hat{s}_i,x_i) \|_2$ if $\phi(\hat{s}_i,\hat{x}_i) \neq \phi(\hat{s}_i,x_i)$, and $\Delta(\hat{s}_i, \hat{x}_i, x) \coloneqq 0$ otherwise. Differently from the definition in Eq. \eqref{eq:consistent}, here we simply normalize the vectors $\phi(\hat{s}_i,\hat{x}_i) - \phi(\hat{s}_i,x_i)$, which does not change the set $\mathbb{C}$. Notice that the vectors $\Delta(\hat{s}_i, \hat{x}_i, x)$ are defined by the data $\{(\hat{s}_i,\hat{x}_i)\}_{i=1}^N$, and in turn define the set $\mathbb{C}$. Thus, the vector $\theta \in \mathbb{C}$ most robust to perturbations in the data, i.e., perturbation of $\Delta(\hat{s}_i, \hat{x}_i, x)$, can be found by solving the minimax problem
\begin{equation}
\label{eq:robust}
\begin{aligned}
\max_{\substack{\| \theta \|_2 = 1 \\ \theta \in \mathbb{C}}} \min_{\substack{w \in \mathbb{R}^p\\ (\hat{s}_k, \hat{x}_k) \in \{ (\hat{s}_i, \hat{x}_i) \}_{i=1}^N \\ x \in \mathbb{X}(\hat{s}_k)}} \quad & \|w\|_2^2 \\
\emph{s.t.} \ \quad & \langle \theta,  \Delta(\hat{s}_k, \hat{x}_k, x) + w \rangle = 0.
\end{aligned}  
\end{equation}
In problem \eqref{eq:robust}, the ``max player'' optimizes for the vector $\theta \in \mathbb{C}$ that requires the largest perturbation $w$ so that it lies in the hyperplane $\langle \theta,  \Delta(\hat{s}_k, \hat{x}_k, x) + w \rangle = 0$ (i.e., a perturbed facet of $\mathbb{C}$). The ``min player'' tries to find the ``most vunerable'' $\Delta_i(x)$, that is, the facet of $\mathbb{C}$ that requires the smallest perturbation vector $w$ to make $\langle \theta,  \Delta(\hat{s}_k, \hat{x}_k, x) + w \rangle = 0$. It can be shown that the optimal perturbation is $w = -\big(\langle \theta, \Delta(\hat{s}_k, \hat{x}_k, x) \rangle / \|\theta\|_2^2\big) \theta$, and substituting it in the objective function of \eqref{eq:robust}, one can show that the optimization problem \eqref{eq:robust} is equivalent to the incenter reformulation of Theorem \ref{theo:incenter_reformulation}.

As a final comment, even ignoring the computational cost of computing the circumcenter vector versus computing the incenter vector, for cases when the set of consistent vectors $\mathbb{C}$ is not empty, the incenter approach may perform better in practice (e.g., see the results in Section \ref{sec:consistent_out}). In light of the discussion in Remark \ref{remark:robustness}, this suggests that approaches robust to perturbations in the data (i.e., incenter) might be preferred to the one that is robust to an adversarial true data-generating model (i.e., circumcenter).

\subsection{Reformulations}
\label{sec:reformulations}

Next, we present a reformulation of \eqref{eq:incenter}. This reformulation will be used to derive tractable incenter-based approaches to IO problems, and we show that if $\text{int}(\mathbb{C}) \neq \emptyset$, one can compute the incenter vector by solving an optimization problem without a norm equality constraint.

\begin{theorem}[Incenter reformulation]
\label{theo:incenter_reformulation}
Problem \eqref{eq:incenter} can be reformulated as
\begin{equation}
    \label{eq:incenter_reformulation}
    \begin{aligned}
        (\theta^{\emph{in}}, r^{\emph{in}}) \in \arg\max_{\theta, r} \ & r \\
        \emph{s.t.} \ & \inner{\theta}{\phi(\hat{s}_i,\hat{x}_i) - \phi(\hat{s}_i,x_i)} + r\| \phi(\hat{s}_i,x_i) - \phi(\hat{s}_i,\hat{x}_i) \|_2 \leq 0 \ \forall x_i \in \mathbb{X}(\hat{s}_i), \forall i \in [N] \\
        &\|\theta\|_2 = 1.
    \end{aligned}
\end{equation}
\end{theorem}

\begin{corollary}[Incenter convex characterization]
\label{coro:incenter_convex_reformulation}
Assume $\emph{int}(\mathbb{C}) \neq \emptyset$. Then, problem \eqref{eq:incenter_reformulation} can be solved as
\begin{equation}
    \label{eq:incenter_convex_reformulation}
    \begin{aligned}
        \bar{\theta}^{\emph{in}} \in \arg\min_{\theta} \quad & \|\theta\|_2 \\
        \emph{s.t.} \quad & \inner{\theta}{\phi(\hat{s}_i,\hat{x}_i) - \phi(\hat{s}_i,x_i)} + \| \phi(\hat{s}_i,x_i) - \phi(\hat{s}_i,\hat{x}_i) \|_2 \leq 0 \quad \forall x_i \in \mathbb{X}(\hat{s}_i), \ \forall i \in [N],
    \end{aligned}
\end{equation}
where $\theta^{\emph{in}} = \bar{\theta}^{\emph{in}}/\|\bar{\theta}^{\emph{in}}\|_2$ is an optimal solution of problem \eqref{eq:incenter_reformulation}.
\end{corollary}

\begin{figure}
    \centering
    \includegraphics[width = 0.5\linewidth]{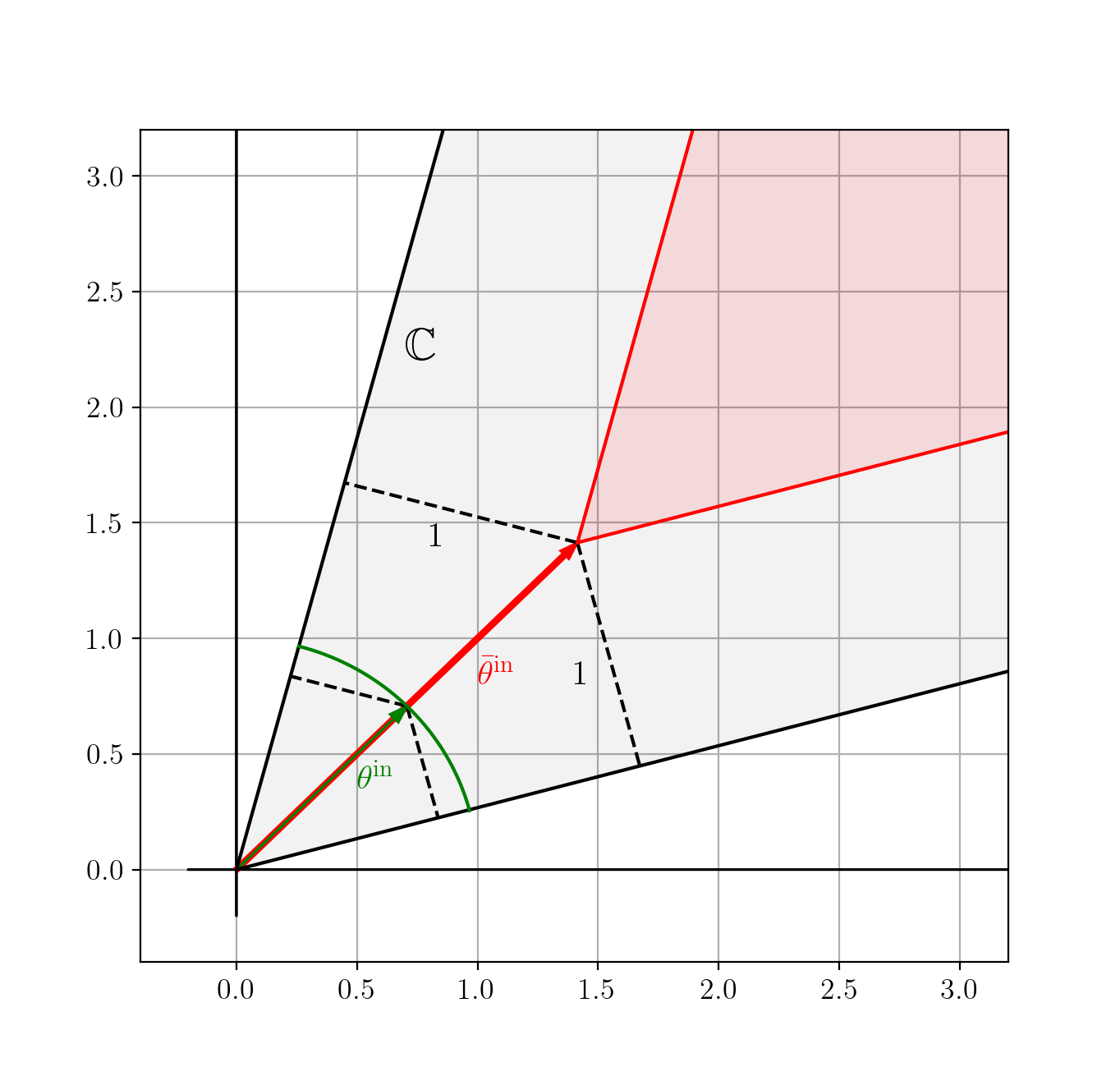}
    \caption{Geometrical illustration of Corollary \ref{coro:incenter_convex_reformulation} in a simple 2D example. The grey region represents $\mathbb{C}$, the green region represents the feasible set of \eqref{eq:incenter_reformulation}, and the red region represents the feasible set of \eqref{eq:incenter_convex_reformulation}. As can be seen, the optimal solution of \eqref{eq:incenter_convex_reformulation} can be interpreted as the smallest norm vector inside an inner cone with boundaries 1 unit away from the boundaries of $\mathbb{C}$. Also, it can be seen that by normalizing $\bar{\theta}^{\text{in}}$, we retrieve $\theta^{\text{in}}$, an optimal solution of \eqref{eq:incenter_reformulation}.}
    \label{fig:alternative_incenter}
\end{figure}

Figure \ref{fig:alternative_incenter} presents a geometrical intuition behind the alternative reformulation of Corollary \ref{coro:incenter_convex_reformulation} in a simple 2D example. Exploiting the nonempty interior assumption $\text{int}(\mathbb{C}) \neq \emptyset$, we were able to formulate problem \eqref{eq:incenter_convex_reformulation} without a norm equality constraint, making it a convex optimization problem. In particular, it is tractable whenever $\mathbb{X}(\hat{s}_i)$ has finitely many elements for all $i \in [N]$ . In sections \ref{sec:mixed} and \ref{sec:first-order}, we discuss approaches to deal with cases when $\mathbb{X}(\hat{s}_i)$ may have infinitely many elements. The tractability results in Theorem~\ref{theo:np_hard_circumcenter} and Corollary~\ref{coro:incenter_convex_reformulation} build on a connection to well-known problems in the optimization literature: computing extremal volume balls. 

\begin{remark}[Connection with extremal volume balls]
\label{remark:extremal_volume}
Interestingly, the circumcenter (resp.\,incenter) problem is closely related to the problem of finding the minimum (resp.\,maximum) volume ball that covers (resp.\,lies inside) a nonempty polyhedron. The difference is that, instead of optimizing the volume of a ball so that it covers (resp.\,lies inside) a nonempty polyhedron, we optimize the aperture angle of a circular cone, so that it covers (resp.\,lies inside) a polyhedral cone (see Figure~\ref{fig:centers}). It is known that computing the minimum volume ball that contains a polyhedron defined by linear inequalities is an intractable convex optimization problem \cite[Section 10.5.1]{nemirovsky1996interior}, and Theorem~\ref{theo:np_hard_circumcenter} indicates that the circumcenter has the same property. On the other hand, finding the maximum volume ball that lies inside a polyhedron defined by linear inequalities (a.k.a. the Chebyshev center of the polyhedron) is known to be equivalent to a convex optimization problem \cite[Section 8.5.1]{boyd2004convex}, and Corollary~\ref{coro:incenter_convex_reformulation} also confirms that the incenter has the same property.
\end{remark}

Since the problem of extremal volume balls is a special case of the more general problem of extremal volume ellipsoids \cite[Section 8.4]{boyd2004convex}, one natural generalization of the incenter and circumcenter concepts is to use \textit{ellipsoidal cones}. For instance, an ellipsoidal generalization of the circumcenter concept is used by \cite{besbes2023contextual} for online IO problems. An ellipsoidal generalization of the incenter reformulation of Theorem \ref{theo:incenter_reformulation} is
\begin{equation}
    \label{eq:ellipsoidal_incenter}
    \begin{aligned}
        \max_{\theta, A} \ & \log \left( \det \left( A  \right)\right) \\
        \emph{s.t.} \ & \inner{\theta}{\phi(\hat{s}_i,\hat{x}_i) - \phi(\hat{s}_i,x_i)} + \| A(\phi(\hat{s}_i,x_i) - \phi(\hat{s}_i,\hat{x}_i)) \|_2 \leq 0 \quad \forall x_i \in \mathbb{X}(\hat{s}_i), \forall i \in [N] \\
        &\|\theta\|_2 = 1, \quad A \succcurlyeq 0,
    \end{aligned}
\end{equation}
which is based on \cite[Section 8.4.2]{boyd2004convex}. Even though, to the best of our knowledge, there are no theoretical results for the performance of ellipsoidal incenter and circumcenter approaches for offline IO problems, their performance in our numerical experiments in Section \ref{sec:numerical} is indeed promising.

We end this section presenting a generalization of the program \eqref{eq:incenter_convex_reformulation}. This generalization will give us the flexibility to exploit structures specific to different IO problems. We generalize \eqref{eq:incenter_convex_reformulation} as
\begin{equation}
    \label{eq:incenter_gen}
    \begin{aligned}
        \min_{\theta} \quad & \mathcal{R} (\theta) \\
        \text{s.t.} \ \quad & \inner{\theta}{\phi(\hat{s}_i,\hat{x}_i) - \phi(\hat{s}_i,x_i)} + d(\hat{x}_i,x_i)\leq 0 \quad\quad \forall x_i \in \mathbb{X}(\hat{s}_i), \ \forall i \in [N] \\
        &\theta \in \Theta.
    \end{aligned}
\end{equation}
This generalization occurs on three levels:
\begin{enumerate}[(i)]
    \item \textbf{Regularizer}: We generalize the objective function to a general regularization function $\mathcal{R} : \mathbb{R}^p \to \mathbb{R}$. For example, we could have $\mathcal{R} (\theta) = \| \theta - \hat{\theta} \|_2^2$, where $\hat{\theta}$ is an a priory belief or estimate of the true cost vector, or $\mathcal{R} (\theta) = \| \theta \|$ for a general norm.
    \item \textbf{Distance in the response space}: Instead of computing the distance between vectors using the Euclidean distance, we consider a general distance function $d : \mathbb{X} \times \mathbb{X} \to \mathbb{R}_+$. For instance, $d(\hat{x},x) = \|\hat{x} - x\|$ on continuous spaces, or $d(\hat{x},x) = I(\hat{x},x)$ on discrete spaces, where $I(\hat{x},x) = 0$ if $\hat{x}=x$, otherwise $I(\hat{x},x) = 1$. We note that $d$ could also be a function of $\hat{s}$. For instance, we could use it to penalize the distance between $\hat{x}$ and $x$ in the \textit{feature space}, by choosing $d(\hat{x},x) = \|\phi(\hat{s}, \hat{x}) - \phi(, \hat{s}, x)\|$. For simplicity, we omit this dependency.
    \item \textbf{Prior information constraint}: We add the additional constraint $\theta \in \Theta$. The set $\Theta$ is used to encode any prior information or assumption we may have on the expert's true cost function, e.g., nonnegativity of the cost vector.
\end{enumerate}
Notice that as long as $\mathcal{R} (\theta)$ is convex in $\theta$ and the set $\Theta$ is convex, Problem \eqref{eq:incenter_gen} is a convex optimization problem.


\section{Augmented Suboptimality Loss}
\label{sec:ASL}

In the previous section, we have shown how one can tackle IO problems when the dataset $\widehat{\mathcal{D}}$ is consistent with some cost $F(s,x) = \inner{\theta^\star}{\phi(s,x)} \in \mathcal{F}_\phi$. Moreover, to use the proposed approaches in practice, we also need to be able to list all elements of $\mathbb{X}(\hat{s}_i)$ for all $i \in [N]$, since each of these elements induces one inequality constraint (e.g., see \eqref{eq:incenter_gen}). In this section, we present approaches to tackle an IO problem based on a \textit{loss function}, which will allow us to handle problems with possibly inconsistent data and when the cardinality of $\mathbb{X}(\hat{s}_i)$ is very large (or even infinite). We first introduce a novel loss function for IO problems and show how solving the IO problem using this loss can be interpreted as a relaxation of the optimization program \eqref{eq:incenter_gen}.

Recall that in IO problems, our goal is to find a parameter vector $\theta$, such that when solving the optimization problem $\min_{x \in \mathbb{X}(s)} \ \inner{\theta}{\phi(s,x)}$, we can reproduce (or approximate) the action the expert would have taken given the same signal $s \in \mathbb{S}$ (see Section \ref{sec:problem}). In the data-driven IO literature, this problem is usually posed as a (regularized) empirical loss minimization problem
\begin{equation}
\label{eq:reg_loss_minimization}
    \min_{\theta \in \Theta} \ \kappa \mathcal{R}(\theta) + \frac{1}{N}\sum_{i=1}^N \ell_\theta (\hat{s}_i, \hat{x}_i),
\end{equation}
where $\mathcal{R} : \mathbb{R}^p \to \mathbb{R}$ is a regularization function, $\kappa$ is a nonnegative regularization parameter, and $\ell_\theta : \mathbb{S} \times \mathbb{X} \to \mathbb{R}$ is some loss function, designed in a way that, by solving \eqref{eq:reg_loss_minimization}, we find a cost vector $\theta$ such that the function $\inner{\theta}{\phi(s,x)}$ achieves our IO goal. Similar to supervised learning methods in machine learning, the regularization parameter $\kappa$ should be chosen to improve the out-of-sample performance of the learned model, e.g., using cross-validation techniques. In this work, we propose a new loss function for IO problems, which we name \textit{Augmented Suboptimality Loss} (ASL).

\begin{definition}[Augmented Suboptimality Loss]
\label{def:ASL}
Given a signal-response pair $(\hat{s},\hat{x})$, the Augmented Suboptimality Loss of a cost vector $\theta$ is
\begin{equation}
\label{eq:ASL}
\tag{ASL}
    \ell_\theta (\hat{s},\hat{x}) \coloneqq \max_{x \in \mathbb{X}(\hat{s})} \left\{ \inner{\theta}{\phi(\hat{s},\hat{x}) - \phi(\hat{s},x)} + d(\hat{x},x) \right\},
\end{equation}
where $\phi$ is a feature mapping and $d$ is a distance function.  
\end{definition}

To simplify the notation, we omit the dependence of the \ref{eq:ASL} on $\phi$ and $d$. We note that the \ref{eq:ASL} can be the well-known \textit{Suboptimality Loss} \citep{esfahani2018data} for the special case with no distance penalization, i.e., $d(\hat{x},x) = 0$. 

\subsection{Connections with incenter}

Next, we show how to solve problem \eqref{eq:reg_loss_minimization} when the loss function is the \ref{eq:ASL}. This reformulation is useful to tackle IO problems with inconsistent data, where there is no cost vector $\theta^\star$ that describes the dataset $\widehat{\mathcal{D}} = \{(\hat{s}_i, \hat{x}_i)\}_{i=1}^N$ perfectly. Using a standard epigraph reformulation, one can see that the optimization problem \eqref{eq:reg_loss_minimization} with the \ref{eq:ASL} yields the program
\begin{equation}
    \label{eq:incenter_gen_inconsistent}
    \begin{aligned}
        \min_{\theta, \beta_1, \ldots, \beta_N} \quad & \kappa \mathcal{R}(\theta) + \frac{1}{N}\sum_{i=1}^N \beta_i \\
        \text{s.t.} \quad \quad & \inner{\theta}{\phi(\hat{s}_i,\hat{x}_i) - \phi(\hat{s}_i,x_i)} + d(\hat{x}_i,x_i) \leq \beta_i \quad\quad \forall x_i \in \mathbb{X}(\hat{s}_i), \ \forall i \in [N] \\
        &\theta \in \Theta.
    \end{aligned}
\end{equation}

Problem \eqref{eq:incenter_gen_inconsistent} is closely related to the incenter reformulation \eqref{eq:incenter_gen}. More precisely, \eqref{eq:incenter_gen_inconsistent} can be interpreted as a \textit{relaxation} of \eqref{eq:incenter_gen} to handle IO problems with inconsistent data. To see this, simply notice that we can arrive at \eqref{eq:incenter_gen_inconsistent} by adding \textit{slack variables} to the constraints of \eqref{eq:incenter_gen} and using $\kappa$ as a trade-off parameter between the sum of slack variables (i.e., the total violation of the constraints of \eqref{eq:incenter_gen}) and the regularization term $\mathcal{R}(\theta)$. Another way to see \eqref{eq:reg_loss_minimization} as a relaxation of \eqref{eq:incenter_gen}, is to write \eqref{eq:incenter_gen} in a ``loss minimization'' form. Namely, notice that \eqref{eq:incenter_gen} can be written as
\begin{equation}
\label{eq:incenter_gen_loss}
    \min_{\theta \in \Theta} \ \mathcal{R}(\theta) + \mathbb{I} \left(\frac{1}{N}\sum_{i=1}^N \ell_\theta(\hat{s}_i,\hat{x}_i) \right),
\end{equation}
where the indicator function $\mathbb{I}(a) = 0$ if $a \leq 0$, otherwise $\mathbb{I}(a) = \infty$. Thus, \eqref{eq:reg_loss_minimization} can also be interpreted as Lagrangian relaxation of \eqref{eq:incenter_gen_loss}. In Appendix \ref{app:theoretical}, we present a different way to motivate the \ref{eq:ASL}, namely, as a convex surrogate of the so-called \textit{predictability loss}. There, we also discuss several properties of the \ref{eq:ASL} which are relevant to IO problems.

So far, we have only considered the case when the data $\{\hat{x}_i\}_{i=1}^N$ is feasible, i.e., $\hat{x}_i \in \mathbb{X}(\hat{s}_i)$. However, in practice, it may be the case that $\hat{x}_i \notin \mathbb{X}(\hat{s}_i)$ for some $i \in [N]$. This may happen due to noise in the data acquisition, or simply due to a mismatch in the choice of the set mapping $\mathbb{X}$ compared with the true constraint set used by the expert agent. For example, consider the problem where the signal $\hat{s}_i$ represents a graph, and $\mathbb{X}(\hat{s}_i)$ is the set of \emph{Hamiltonian cycles} over this graph (i.e., a graph cycle that visits each node exactly once). This scenario is common when modeling vehicle routing problems, where the nodes of the graph represent a set of customers, and the Hamiltonian cycle represents the sequence of customers a driver visited (for example, to deliver packages). In this scenario, infeasible data could reflect, for example, missing data on the complete cycle driven by the driver (i.e., noise in the data acquisition), or the fact that in reality, the driver can visit the same customer multiple times to deliver missing packages (i.e., mismatch in the choice of the set mapping $\mathbb{X}(\hat{s}_i)$).

\begin{remark}[Handling infeasible data]
One way to handle infeasible data is just trying to ignore these cases, treating them as outliers in the data. Numerically, since the nonnegativity of the \ref{eq:ASL} depends on the assumption that $\hat{x}_i \in \mathbb{X}(\hat{s}_i) \ \forall i \in [N]$ (see Proposition \ref{prop:augmented_suboptimality_surrogate}), having infeasible data may lead to negative loss values.  Thus, a possible way to handle these cases is to use a modified loss function $\tilde{\ell}_\theta(\hat{s},\hat{x}) = \max\{0, \ell_\theta(\hat{s},\hat{x})\}$, which simply ignores negative loss values (i.e., infeasible data). This technique shares the same principle as the bounded rationality loss of \citep{esfahani2018data}. Importantly, this modification preserves the convexity of the loss minimization problem. For instance, this modification is equivalent to simply adding the constraints $\beta_i \geq 0, \ \forall i \in [N]$ to \eqref{eq:incenter_gen_inconsistent}.
\end{remark}

The IO problem shares some similarities with \textit{structured prediction} problems, which commonly refer to the class of supervised learning problems, where the output space consists of a finite set of structured objects (e.g., graphs), instead of a simple set of labels, like in standard multiclass classification problems \citep{taskar2005learning}. Interestingly, the IO formulation in Eq. \eqref{eq:incenter_gen_inconsistent} is closely related to some formulations proposed to solve structured prediction problems.

\begin{remark}[Connections with structured prediction]
\label{remark:connections_structured} In the incenter program \eqref{eq:incenter_gen_inconsistent}, when $\mathcal{R}$ is the Euclidean norm, $d$ is the Hamming distance, and feature function $\phi$ is linear in $x$, learning the incenter cost vector is equivalent to the Maximum Margin Planning problem presented in \citep{ratliff2006maximum}, which is a type of structured prediction problem. More generally, formulation \eqref{eq:incenter_gen_inconsistent} is closely related to the so-called Structured Support Vector Machine (SSVM) approach for structured prediction problems \citep{tsochantaridis2005large, nowozin2011structured}.
\end{remark}

The connection between IO problems with discrete feasible sets and the SSVM approach to structured prediction problems is related to the fact that when the constraint set $\mathbb{X}(\hat{s})$ has finitely many elements, each of these elements can be interpreted as a class in a multiclass classification problem. Under this interpretation, the optimization problem \eqref{eq:incenter_gen_inconsistent} can be viewed as a generalization of the soft-margin SMV problem, where the ``classes'' $x \in \mathbb{X}(\hat{s})$ can be complex structured objects, which is precisely the SSVM problem. Moreover, the concept of ``margin'' in the SVM framework is related to the distance function $d$ in \eqref{eq:incenter_gen_inconsistent}, which comes from the fact that in our incenter-based IO formulation, we want to maximize the angle between the incenter vector and the boundaries of the set $\mathbb{C}$ (see Figure \ref{fig:incenter}). These observations reveal an interesting connection between incenter-based approaches for IO problems with discrete constraint sets and SSVM approaches for structured prediction problems. 

In the last part of this section, we revisit this relation between the incenter and circumcenter concepts through the lens of the regret $R$ defined in~\eqref{eq:regret}. In fact, given two cost vectors~$(\theta, \theta^\star)$, there is an asymmetry $R(\theta, \theta^\star) \neq R(\theta^\star, \theta)$, interestingly, each of which relating to one of these concepts.  

\begin{remark}[Connections with regret]
\label{remark:regret}
Consider the cost vectors~$(\theta, \theta^\star)$ and the regret~\eqref{eq:regret2}.  
\begin{enumerate}[(i)]
    \item \textbf{SPO vs Suboptimality losses:} The regret~$R(\theta, \theta^\star)$ is indeed the Smart ``Predict, then Optimize" loss~(SPO) studied in \citep{elmachtoub2022smart}, with respect to which the circumcenter in Definition~\ref{def:circumcenter} is shown to be the worst case optimal. However, the symmetric counterpart~$R(\theta^\star, \theta)$ coincides with the Suboptimality loss from~\citep{esfahani2018data}, which, as discussed above, is a special case of the \ref{eq:ASL} connected to the incenter in Definition~\ref{def:incenter}.

    \item \textbf{Convexity and tractability:} The regret~$R$ is convex in the second argument~\citep{esfahani2018data} (note the connection to suboptimality in the previous point), whereas inherently nonconvex in the first argument; see~\citep{elmachtoub2022smart} for the connection of the latter to the 0-1 loss in classification problems. This observation is indeed aligned with the intractability results of circumcenter (Theorem~\ref{theo:np_hard_circumcenter}) and the tractability of incenter (Corollary~\ref{coro:incenter_convex_reformulation}). It is worth noting that \citep{elmachtoub2022smart} also proposes a convexified version of the SPO loss (i.e., the regret in the first argument), reminiscent of the hinge loss in classification problems. 
    
    \item \textbf{Additional required data measurments}: Given the cost vector~$\theta$ and the IO data pair $(\hat{s}, x^\star)$~(for the definition of $x^\star$ see \eqref{eq:regret1}), the regret~$R(\theta^\star, \theta)$ can be computed without the knowledge of the ground truth $\theta^\star$, whereas its symmetric counterpart~$R(\theta, \theta^\star)$ depends explicitly on $\theta^\star$ (or its projection through a feature function as in \citep{elmachtoub2022smart}). 
\end{enumerate}
\end{remark}

\subsection{General reformulation for mixed-integer feasible sets}
\label{sec:mixed}

In this section, we present a way to reformulate the IO problem using the \ref{eq:ASL} for problems when $\mathbb{X}(\hat{s})$ is mixed-integer set, which generalizes many formulations from the literature. For this purpose, we consider the mixed-integer feasible set 
\begin{equation}
    \label{eq:MI_set}
    \mathbb{X}(\hat{s}) \coloneqq \left\{ (y,z) \in \mathbb{R}^u \times \mathbb{Z}^v : \hat{A}y + \hat{B}z \leq \hat{c}, \ z \in \mathbb{Z}(\hat{w}) \right\},
\end{equation}
where $\hat{s} \coloneqq (\hat{A}, \hat{B}, \hat{c}, \hat{w})$, and $\mathbb{Z}(\hat{w})$ is a bounded set that may depend on the signal $\hat{w}$. Since $y$ is a continuous variable, $\mathbb{X}(\hat{s})$ has infinitely many elements and we cannot use \eqref{eq:incenter_gen_inconsistent} for this IO problem in practice. To solve this problem, one could use first-order iterative methods to directly optimize the IO loss minimization problem \eqref{eq:reg_loss_minimization} (we discuss such approaches in Section \ref{sec:first-order}). In this section, we leverage classical tools from convex duality to reformulate \eqref{eq:reg_loss_minimization} with the \ref{eq:ASL} and constraint set \eqref{eq:MI_set} as a tractable finite convex optimization.

For this reformulation, we use the following hypothesis function
\begin{equation}
    \label{eq:MI_hypothesis}
    \inner{\theta}{\phi(s,x)} = \langle y, Q_{yy} y \rangle + \langle y, Q \phi_1(w, z) \rangle + \langle q, \phi_2(w, z) \rangle,
\end{equation}
where $x \coloneqq (y,z)$, $\theta \coloneqq (\text{vec}(Q_{yy}), \text{vec}(Q), q) \in \mathbb{R}^{u^2 + um + r}$, $Q_{yy} \succcurlyeq 0$, and $\phi_1 : \mathbb{W} \times \mathbb{Z}^v \to \mathbb{R}^{m}$ and $\phi_2 : \mathbb{W} \times \mathbb{Z}^v \to \mathbb{R}^{r}$ are feature functions ($\phi$ can be written in terms of $\phi_1$ and $\phi_2$). The choice of a quadratic hypothesis and $\mathbb{X}(\hat{s})$ with linear inequality constraints is chosen because it generalizes the linear hypothesis case, and also for simplicity of exposition. In general, as long as the hypothesis function and constraint set are convex w.r.t. the continuous part of the decision vector (e.g., \textit{conic representable} problems \citep{esfahani2018data}), similar results could be derived. Continuing, we use $d(\hat{x}, x) = \|\hat{y} - y\|_\infty + d_z(\hat{z}, z)$, that is, we use the $\infty$-norm for the continuous part of the decision vector and a general distance function for the integer part of the decision vector. We use the $\infty$-norm for the continuous part of the decision vector because, to use reformulation techniques based on convex duality, we need the inner maximization problem of the \ref{eq:ASL} to be concave in $y$, and by introducing auxiliary integer variables, we can reformulate the $\infty$-norm as a concave function of $y$. More precisely, we exploit the identity $\|\hat{y} - y\|_\infty = \max_{h\in \{-1, 0, 1\}^u, \|h\|_1 = 1} \langle h, \hat{y} - y \rangle$.

\begin{theorem}[\ref{eq:ASL} and mixed-integer feasible set]
\label{theo:MI_reformulation}
For the mixed-integer feasible set \eqref{eq:MI_set}, hypothesis function \eqref{eq:MI_hypothesis}, and distance function $d(\hat{x}, x) = \|\hat{y} - y\|_\infty + d_z(\hat{z}, z)$, problem \eqref{eq:reg_loss_minimization} can be reformulated as
\begin{equation}
\label{eq:MI_reformulation}
    \begin{aligned}
    \min \ & \kappa\mathcal{R}(\theta) + \frac{1}{N}\sum_{i=1}^N \beta_i  \\
    \emph{s.t.} \ & \theta = (\emph{vec}(Q_{yy}), \emph{vec}(Q), q) \in \Theta, \quad \lambda_{ijk}\geq 0, \quad \alpha_{ijk}, \beta_i \in \mathbb{R} \\
    & \inner{\theta}{\phi(\hat{s}_i,\hat{x}_i)} + \alpha_{ijk} + \langle \lambda_{ijk}, \hat{c}_i-\hat{B}_iz_{ij} \rangle - \langle q, \phi_2(\hat{w}_i,z_{ij}) \rangle + \langle h_k, \hat{y}_i \rangle + d_z(\hat{z}_i, z_{ij}) \leq \beta_i \\
    & \begin{bmatrix}
    Q_{yy} & Q \phi_1(\hat{w}_i, z_{ij}) + h_k + \hat{A}_i^\top \lambda_{ijk} \\
    * & 4\alpha_{ijk}
    \end{bmatrix} \succcurlyeq 0,
    \end{aligned}
\end{equation}
where the constraints are for all $(i, j, k) \in [N] \times [M_i] \times [2u]$, $\mathbb{Z}(\hat{w}_i) \coloneqq \{z_{i1}, \ldots,z_{ij}, \ldots, z_{iM_i}\}$, $M_i \coloneqq |\mathbb{Z}(\hat{w}_i)|$, and if $k \leq u$ (resp. $k > u$),  $h_k$ is the vector of zeros except for the $k$'th (resp. $(k-u)$'th) element, which is equal to 1 (resp. -1).
\end{theorem}

In case we use a linear hypothesis function instead of a quadratic one, the matrix inequality constraints of \eqref{eq:MI_reformulation} reduce to much simpler linear equality constraints.

\begin{corollary}[LP reformulation for linear hypotheses]
\label{coro:MI_reformulation_linear}
For the mixed-integer feasible set \eqref{eq:MI_set}, linear hypothesis function (i.e., \eqref{eq:MI_hypothesis} with $Q_{yy} = 0$), and distance function $d(\hat{x}, x) = \|\hat{y} - y\|_\infty + d_z(\hat{z}, z)$, problem \eqref{eq:reg_loss_minimization} can be reformulated as
\begin{equation}
\label{eq:MI_reformulation_linear}
    \begin{aligned}
    \min \ & \kappa\mathcal{R}(\theta) + \frac{1}{N}\sum_{i=1}^N \beta_i  \\
    \emph{s.t.} \ & \theta = (\emph{vec}(Q), q) \in \Theta, \quad \lambda_{ijk}\geq 0, \quad \beta_i \in \mathbb{R} \\
    & \inner{\theta}{\phi(\hat{s}_i,\hat{x}_i)} + \langle \lambda_{ijk}, \hat{c}_i-\hat{B}_iz_{ij} \rangle - \langle q, \phi_2(\hat{w}_i,z_{ij}) \rangle + \langle h_k, \hat{y}_i \rangle + d_z(\hat{z}_i, z_{ij}) \leq \beta_i \\
    & Q \phi_1(\hat{w}_i, z_{ij}) + h_k + \hat{A}_i^\top \lambda_{ijk} = 0,
    \end{aligned}
\end{equation}
where the constraints are for all $(i, j, k) \in [N] \times [M_i] \times [2u]$, $\mathbb{Z}(\hat{w}_i) \coloneqq \{z_{i1}, \ldots, z_{ij}, \ldots, z_{iM_i}\}$, $M_i \coloneqq |\mathbb{Z}(\hat{w}_i)|$, and if $k \leq u$ (resp. $k > u$),  $h_k$ is the vector of zeros except for the $k$'th (resp. $(k-u)$'th) element, which is equal to 1 (resp. -1).
\end{corollary}

We note that one can reduce the number of constraints in these reformulations by a factor of $2u$ by using $d(\hat{x}, x) = d_z(\hat{z}, z)$, i.e., only penalizing the integer part of the decision vector, which is equivalent to setting $h_k = 0 \ \forall k \in [2u]$.

\begin{remark}[Generality of Theorem \ref{theo:MI_reformulation}]
\label{remark:generality}
Reformulation \eqref{eq:MI_reformulation} in Theorem \ref{theo:MI_reformulation}, for IO problem with mixed-integer decision sets, generalizes several reformulations from the literature. For instance, for an IO problem with purely continuous decision sets and linear hypothesis functions, reformulation \eqref{eq:MI_reformulation} with $d_z(\hat{x}, x) = 0$ and $h_k =0 \ \forall k \in [2u]$ reduces to classical Inverse Linear Optimization reformulations \citep{ahuja2001inverse, chan2019inverse}. Similarly, if the IO problem has purely continuous decision set and a quadratic hypothesis function, then the program \eqref{eq:MI_reformulation} with $d_z(\hat{x}, x) = 0$ and $h_k =0 \ \forall k \in [2u]$ reduces to the LMI reformulation in \citep{akhtar2021learning}. If the decision set is purely discrete, that is, $\mathbb{X}(\hat{s})$ has finitely many elements, then \eqref{eq:MI_reformulation} and \eqref{eq:MI_reformulation_linear} recover \eqref{eq:incenter_gen_inconsistent}.
\end{remark}

In conclusion, by dualizing the continuous part of the problem, Theorem \ref{theo:MI_reformulation} and Corollary \ref{coro:MI_reformulation_linear} present a way to deal with a case when the feasible set $\mathbb{X}(\hat{s})$ is a mixed-integer set. It is worth noting that when the feasible set $\mathbb{X}(\hat{s})$ is a mixed-integer set, the size of reformulations \eqref{eq:MI_reformulation} and \eqref{eq:MI_reformulation_linear} grow linearly in the number of integer variables, in the dimension of the continuous variables, and in the size of the dataset. That means, for Inverse Optimization problems with large datasets and/or mixed-integer sets with a large number of decision variables, these reformulations can be computationally challenging to solve. This issue is the main focus of our tailored first-order algorithm in the next section.


\section{Tailored Algorithm: Stochastic Approximate Mirror Descent}
\label{sec:first-order}

All approaches to solving IO problems presented so far depend on optimization programs with possibly a large number of constraints. Given a dataset $\{(\hat{s}_i, \hat{x}_i)\}_{i=1}^N$, these optimization problems have at least $\sum_{i=1}^N |\mathbb{X}(\hat{s}_i)|$ (or $\sum_{i=1}^N |\mathbb{Z}(\hat{w}_i)|$ for the mixed-integer case) constraints. When $N$ is too large (i.e., we have too many signal-response examples) or $|\mathbb{X}(\hat{s}_i)|$ is too large, solving these optimization programs may be intractable in practice. Thus, the focus of this section is to develop a tailored first-order algorithm to tackle such IO problems in a provably efficient way, opening doors to a wider range of real-world applications. A straightforward first-order approach to solve \eqref{eq:reg_loss_minimization} is the standard \textit{Project Subgradient method} \citep{shor1985minimization}. This is a general first-order optimization algorithm, and it converges to an optimal solution of the problem under mild convexity conditions on the objective function and constraint set. For example, this was the approach used in \citep{ratliff2006maximum} to solve the Maximum Margin Planning problem (see Remark \ref{remark:connections_structured}). Another idea is to solve \eqref{eq:reg_loss_minimization} as a minimax problem. \cite{taskar2006structured} showed that in some special cases, this problem can be formulated as a bilinear saddle point problem, and then used Nesterov's dual extragradient method \citep{nesterov2007dual} to solve it. However, other than only being applicable to a restrictive class of problems, this approach also requires $2N$ projections onto $\mathbb{X}(\hat{s}_i)$ to be computed per iteration of the algorithm, which may be prohibitive in many applications. Other algorithms for minimax problems, such as Nemirovski's Mirror-Prox algorithm \citep{nemirovski2004prox}, would suffer from similar drawbacks. In this section, we introduce an efficient optimization algorithm that can be applied to general IO problems, while also exploiting the specific structures that stem from IO problems. To this end, we define the following notion of a stochastic approximate subgradient.

\begin{definition}[Stochastic approximate subgradient]
\label{def:SAS}
Let $f : \Theta \to \mathbb{R}$ be a convex function. We say that the random vector $\tilde{g}_\varepsilon(\theta)$ is a \emph{stochastic approximate subgradient} of $f$ at $\theta$ if $\mathbb{E}\left[\tilde{g}_\varepsilon(\theta) \ | \ \theta\right] = g_\varepsilon(\theta)$ and $g_\varepsilon(\theta)$ is an $\varepsilon$-subgradient of $f$ for some $\varepsilon \geq 0$, that is,
$$
f(\theta) - f(\nu) \leq \inner{g_\varepsilon(\theta)}{\theta - \nu} + \varepsilon, \quad \forall \nu \in \Theta.
$$
\end{definition}
For different notions of approximate subgradients, see \citep{ratliff2007approximate, tacskesen2022semi}. Next, we propose to solve IO problems with a novel algorithm, the \textit{Stochastic Approximate Mirror Descent} (SAMD). We first define the SAMD as an algorithm to optimize general convex programs and prove convergence rates for it. After that, we discuss how we can use the SAMD algorithm to exploit the structure of IO problems. For a  function $f$ and set $\Theta$, the SAMD updates are
\begin{equation}
\label{alg:SAMD}
\tag{SAMD}
\theta_{t+1} = \argmin_{\theta \in \Theta} \big\{ \eta_t \inner{\tilde{g}_{\varepsilon_t}(\theta_t)}{\theta} + \mathcal{B}_\omega(\theta,\theta_t)\big\},
\end{equation}
where $\eta_t$ is the step-size, the function $\mathcal{B}_\omega$ is the Bregman divergence w.r.t. $\omega : \Theta \to \mathbb{R}$ \cite[Section 4]{bubeck2015convex}, and $\tilde{g}_{\varepsilon_t}(\theta_t)$ is a \textit{stochastic approximate subgradient} of $f$ as defined in Definition \ref{def:SAS}. The \ref{alg:SAMD} algorithm can be interpreted as a combination of a stochastic mirror descent algorithm \cite[Section 6.1]{bubeck2015convex} and an $\varepsilon$-subgradient method \cite[Section 3.3]{bertsekas2015convex}. Next, we prove a convergence rate for the \ref{alg:SAMD} algorithm, where we use the concept of \textit{relative strong convexity}, which is a generalization of the standard strong convexity property \citep{lu2018relatively}. Namely, if $\mathcal{R}$ is $\alpha$-strongly convex relative to $\mathcal{B}_\omega$, then $\mathcal{R}(x) - \mathcal{R}(y) \leq \langle g(x), x-y \rangle - \alpha \mathcal{B}_\omega(y,x)$, where $g(x) \in \partial R(x)$.

\begin{proposition}[\ref{alg:SAMD} convergence rate]
\label{prop:samd_rate}
Let $f : \Theta \to \mathbb{R}$ be a convex function, and $\Theta$ be a convex set. Assume $\mathcal{B}_\omega(\theta,\nu) \leq R^2$, for some $R>0$, and $\mathbb{E}\left[\|\tilde{g}_{\varepsilon_t}(\theta)\|_*^2 \ | \ \theta\right] \leq G^2$, $\forall \theta,\nu \in \Theta$. Using $\eta_t = c/\sqrt{t}$ for some constant $c > 0$, the \ref{alg:SAMD} algorithm guarantees
\begin{equation*}
    \mathbb{E}\left[f\left(\frac{1}{T}\sum_{t=1}^T \theta_t\right) \right] - \min_{\theta \in \Theta} f(\theta) \leq \left( \frac{R^2}{c} + cG^2\right)\frac{1}{\sqrt{T}} + \frac{1}{T}\sum_{t=1}^T\varepsilon_t.
\end{equation*}
Moreover, if we assume $f = h + \mathcal{R}$, where $h$ is convex and $\mathcal{R}$ is $\alpha$-strongly convex relative to $\mathcal{B}_\omega$, then using $\eta_t = 2/\alpha(t+1)$, the \ref{alg:SAMD} algorithm guarantees
\begin{equation*}
    \mathbb{E}\left[f\left(\frac{2}{T(T+1)}\sum_{t=1}^T t\theta_t\right) \right] - \min_{\theta \in \Theta} f(\theta) \leq \frac{2G^2}{\alpha(T+1)} + \frac{2}{T(T+1)}\sum_{t=1}^T t\varepsilon_t.
\end{equation*}
\end{proposition}
The convergence rate of the \ref{alg:SAMD} algorithm is a combination of the $O(1/\sqrt{T})$ (or $O(1/T)$) rate of the stochastic mirror descent algorithm plus a term that depends on $\{ \varepsilon_t \}_{t=1}^T$ due to the use of approximate subgradients. Notice that, if we use the SAMD algorithm with errors $\varepsilon_t$ diminishing at a rate $O(1/t)$ (i.e., as $t$ increases, we use increasingly more precise approximate subgradients), then the convergence rates of Proposition \ref{prop:samd_rate} reduce to $O(1/\sqrt{T})$ and $O(1/T)$. The \ref{alg:SAMD} algorithm differs from a simple projected subgradient method in three major ways: it uses (i) mirror descent updates, (ii) stochastic subgradients, and (iii) approximate subgradients. Next, we discuss how each of these properties can be used to exploit the structure of the IO problem \eqref{eq:reg_loss_minimization}.

\subsection{Mirror descent updates}
\label{sec:mirror}

For instance, consider \eqref{eq:reg_loss_minimization} with $\mathcal{R}(\theta) = \|\theta\|_1$ and $\Theta = \mathbb{R}^p$. This problem can be equivalently written as
\begin{equation}
    \label{eq:l1_reformulation}
    \begin{aligned}
    \min_{\theta \in \mathbb{R}^p} \quad & \frac{1}{N}\sum_{i=1}^N \ell_\theta (\hat{s}_i, \hat{x}_i) \\
    \text{s.t.} \ \quad & \tilde{\kappa}\|\theta\|_1 \leq 1,
\end{aligned}
\end{equation}
for some $\tilde{\kappa}$ that depends on $\kappa$ and the data $\{(\hat{s}_i, \hat{x}_i)\}_{i=1}^N$. Then, by introducing the nonnegative variables $\theta^+$ and $\theta^-$, defining $\tilde{\theta} \coloneqq (\theta^+, \theta^-)$ (i.e., the concatenation of $\theta^+$ and $\theta-$), and setting $\theta = \theta^+ - \theta^- = [I \ -I]\tilde{\theta}$ (here, $I$ is the identity matrix and $[I \ -I]$ is a block matrix), \eqref{eq:l1_reformulation} can be written as
\begin{equation}
    \label{eq:l1_reformulation_simplex}
    \min_{\tilde{\theta} \in \Delta_{\tilde{\kappa}}} \frac{1}{N}\sum_{i=1}^N \ell_{[I \ -I]\tilde{\theta}} (\hat{s}_i, \hat{x}_i),
\end{equation}
where $\Delta_{\tilde{\kappa}} \coloneqq \{\tilde{\theta} \in \mathbb{R}^{2p} : \tilde{\kappa}\|\tilde{\theta}\|_1 \leq 1, \tilde{\theta} \geq 0\}$ \citep{tibshirani1996regression}. In other words, \eqref{eq:l1_reformulation} can be recast as an optimization problem over a simplex. Next, choosing $\omega (\theta) = \sum_{i=1}^p \theta_i\log(\theta_i)$, the \ref{alg:SAMD} updates applied to \eqref{eq:l1_reformulation_simplex} can be written as ``exponentiated updates''
$$
\tilde{\theta}_{t+1} = 
\begin{cases}
    \tilde{\theta}_t \odot \exp(-\eta_t \tilde{g}_{\varepsilon_t}(\tilde{\theta}_t)) & \text{if} \quad \tilde{\kappa}\|\tilde{\theta}_t \odot \exp(-\eta_t \tilde{g}_{\varepsilon_t}(\tilde{\theta}_t))\|_1 \leq 1\\
    \frac{\tilde{\theta}_t \odot \exp(-\eta_t \tilde{g}_{\varepsilon_t}(\tilde{\theta}_t))}{\tilde{\kappa}\|\tilde{\theta}_t \odot \exp(-\eta_t \tilde{g}_{\varepsilon_t}(\tilde{\theta}_t))\|_1} & \text{otherwise.}
\end{cases} 
$$
Exponentiated updates are known to have better convergence properties compared to standard subgradient descent updates for ``simplex constrained'' problems, as well as not requiring solving optimization problems to project onto the simplex. For a detailed discussion on the advantages of mirror descent updates for different settings, see \citep[Section 5.7]{juditsky2011first_i}. In general, by choosing $\omega (\theta) = \frac{1}{2}\|\theta\|_2^2$,  the \ref{alg:SAMD} updates reduce to standard project subgradient updates
$$
\theta_{t+1} = \Pi_\Theta(\theta_t - \eta_t \tilde{g}_{\varepsilon_t}(\theta_t)),
$$
that is, the subgradient method can be seen as a special case of the mirror descent algorithm.

\subsection{Stochastic subgradients}
\label{sec:stochastic_subgradients}

Using stochastic subgradients is advantageous when computing a stochastic subgradient is computationally cheaper than computing a deterministic subgradient. This observation is supported by the fact that, in expectation, the stochastic subgradient method guarantees the same convergence rate as its deterministic version \cite[Theorem 6.1]{bubeck2015convex}. For the IO loss minimization problem \eqref{eq:reg_loss_minimization}, computing a stochastic subgradient can be significantly cheaper than computing the full subgradient, due to the finite sum structure of \eqref{eq:reg_loss_minimization}. Namely, by Danskin's theorem \cite[Section B.5]{bertsekas2008nonlinear}, we have that the subdifferential of the \ref{eq:ASL} w.r.t. $\theta$ is
\begin{equation}
\label{eq:subdifferential}
    \partial \ell_\theta(\hat{s},\hat{x}) = \text{conv}\left\{ \phi(\hat{s},\hat{x}) - \phi(\hat{s},x^\star(\hat{s})) \ \bigg| \ x^\star(\hat{s}) \in \argmax_{x \in \mathbb{X}(\hat{s})} \big\{d(\hat{x},x) - \inner{\theta}{\phi(\hat{s},x)} \big\}\right\}.
\end{equation}
Thus, to compute a subgradient of $\frac{1}{N}\sum_{i=1}^N \ell_\theta (\hat{s}_i, \hat{x}_i)$, we need to solve $N$ maximization problems, one for each signal-response pair $\{(\hat{s}_i,\hat{x}_i)\}_{i=1}^N$. On the other hand, by sampling an index $j$ uniformly from $[N]$, we have that $g_j(\theta) \in \partial \ell_\theta(\hat{s}_j,\hat{x}_j)$ is a stochastic subgradient of $\frac{1}{N}\sum_{i=1}^N \ell_\theta (\hat{s}_i, \hat{x}_i)$, that is, $\mathbb{E}_{j \sim [N]}[g_j(\theta) \ | \ \theta] = (1/N)\sum_{i=1}^N g_i(\theta)$. This is a standard method for computing stochastic (sub)gradients for empirical risk minimization-type problems (e.g., \cite[Chapter 6]{bubeck2015convex}). In summary, to compute an unbiased stochastic subgradient of \eqref{eq:reg_loss_minimization}, instead of $N$, we need to solve only \textit{one} maximization problem.

\subsection{Approximate subgradients}
\label{sec:approximate_subgradients}

As shown in \eqref{eq:subdifferential}, we need to solve the optimization problem $\max_{x \in \mathbb{X}(\hat{s})} \left\{ d(\hat{x},x) - \inner{\theta}{\phi(\hat{s},x)} \right\}$ in order to compute a subgradient of the \ref{eq:ASL}. However, in practice, it may be too costly to solve this optimization problem to optimality at each iteration of the algorithm. Thus, it would be useful if we could use an approximate solution to this problem, instead of an optimal one. Turns out that, indeed, given an \textit{$\varepsilon$-approximate solution} to the maximization problem, that is, a feasible point $x_\varepsilon$ such that
$$
d(\hat{x},x_\varepsilon) - \inner{\theta}{\phi(\hat{s},x_\varepsilon)} \geq \max_{x \in \mathbb{X}(\hat{s})} \left\{ d(\hat{x},x) - \inner{\theta}{\phi(\hat{s},x)} \right\} - \varepsilon,
$$
we can construct an $\varepsilon$-subgradient of the \ref{eq:ASL}. 

\begin{lemma}[Approximate solutions and $\varepsilon$-subgradients]
\label{lemma:e-subgradient}
Let $x_\varepsilon$ be a feasible, $\varepsilon$-suboptimal solution of $\max_{x \in \mathbb{X}(\hat{s})} \left\{ d(\hat{x},x) - \inner{\theta}{\phi(\hat{s},x)} \right\}$. Thus, the vector $g_\varepsilon(\theta) = \phi(\hat{s},\hat{x}) - \phi(\hat{s},x_\varepsilon)$ is an $\varepsilon$-subgradient of $\ell_{\theta}(\hat{s},\hat{x})$ with respect to $\theta$, i.e.,
$$
\ell_{\theta}(\hat{s},\hat{x}) - \ell_{\nu}(\hat{s},\hat{x}) \leq \inner{\phi(\hat{s},\hat{x}) - \phi(\hat{s},x_\varepsilon)}{\theta - \nu} + \varepsilon.
$$
\end{lemma}

In summary, we can use an approximate solution to the inner maximization problem of the \ref{eq:ASL} to construct $\varepsilon$-subgradients, and as proved in Proposition \ref{prop:samd_rate}, these can be effectively used to compute an approximate solution of convex optimization problems. Finally, it is not difficult to show that by combining stochastic subgradients (i.e., by using only one random signal-response pair) and approximate subgradients (i.e., solving the respective maximization problem approximately), one gets a stochastic approximate subgradient (Definition \ref{def:SAS}). Putting these ideas together, Algorithm~\ref{alg:IO_SAMD} shows the pseudocode of the \ref{alg:SAMD} algorithm applied to problem \eqref{eq:reg_loss_minimization}, where we denote the gradient (or a subgradient) of $\mathcal{R}$ as $\nabla \mathcal{R}$. In line 3 of Algorithm~\ref{alg:IO_SAMD}, one example is sampled from the dataset for each iteration of the algorithm. However, in practice, it may be advantageous to instead sample a batch of $1 \leq B \leq N$ examples, and use this batch of data to compute the approximate stochastic subgradient. By changing the size $B$ of the batch, we can control the trade-off between having more precise subgradients (large $B$) versus faster subgradient computations (small $B$). This idea is explored in the numerical experiments of Section \ref{sec:iter_algs}.

\begin{algorithm}
\setstretch{1.2}
\caption{\ref{alg:SAMD} algorithm for \eqref{eq:reg_loss_minimization}}
\label{alg:IO_SAMD}
\begin{algorithmic}[1]
\State \textbf{Input:} Step-size sequence $\{ \eta_t \}_{t=1}^T$, $\kappa$, $\theta_1$, $d$, $\phi$, $\omega$, $\nabla\mathcal{R}$ and dataset $\{(\hat{s}_i, \hat{x}_i)\}_{i=1}^N$.
\For{$t=1, \ldots, T$}
\State Sample $j$ uniformly from $\{1, \ldots, N\}$
\State Compute $x_t$, a (possibly approximate) solution of $\max_{x \in \mathbb{X}(\hat{s}_j)} \big\{d(\hat{x}_j,x) - \inner{\theta_t}{\phi(\hat{s}_j,x)} \big\}$
\State Approximate stochastic subgradient: $\tilde{g}_t(\theta_t) = \kappa \nabla\mathcal{R}(\theta_t) + \phi(\hat{s}_j,\hat{x}_j) - \phi(\hat{s}_j,x_t)$
\State Mirror descent step: $\theta_{t+1} = \argmin_{\theta \in \Theta} \big\{ \eta_t \inner{\tilde{g}_t(\theta_t)}{\theta} + \mathcal{B}_\omega(\theta,\theta_t)\big\}$
\EndFor
\State \textbf{Output:} $\{\theta_t\}_{t=1}^T$
\end{algorithmic}
\end{algorithm}

We end this section by briefly discussing the case of \textit{online IO}. In this scenario, instead of having a dataset of signal-response data upfront, we receive one signal-response pair at a time. After receiving each signal-response data pair, we must choose a cost vector $\theta_t$ using all the information gathered so far, and evaluate the loss of this cost vector. The final performance of the algorithm is measured using \textit{regret} performance metrics.
\begin{remark}[Regret bounds and online IO]
\cite{barmann2017emulating} use an online Multiplicative Weights Updates (MWU) algorithm to prove an $O(\sqrt{T})$ regret bound, and more recently, \cite{besbes2023contextual} use an online adaptation of the circumcenter concept to prove an $O(\log(T))$ regret bound. In Algorithm \ref{alg:IO_SAMD}, line 3, if instead of sampling a new signal-response example we use the online data received at that iteration, Algorithm \ref{alg:IO_SAMD} can be readily adapted to online IO problems. Moreover, the convergence bound of Proposition \ref{prop:samd_rate} may also be straightforwardly converted to an $O(\sqrt{T})$ regret bound (or $O(\log(T))$ for the strongly convex case) \cite[Chapter 3]{orabona2019modern}. In particular, Algorithm \ref{alg:IO_SAMD} can be interpreted as a generalization of the MWU algorithm \cite[Appendix A.2]{allen2014linear}.
\end{remark}


\section{Numerical Experiments}
\label{sec:numerical}

In this section, we numerically evaluate the approaches to inverse optimization proposed in this paper. All linear and quadratic programs were solved using Gurobi. All semidefinite programs were solved using CVXPY with MOSEK as the solver. For all results presented in this section, we learn the expert's cost function using a training dataset of expert signal-response data and evaluate its performance using a test dataset (i.e., out-of-sample performance). Moreover, in every plot, we report the average performance value for 10 randomly generated true cost vectors, as well as the 5th and 95th percentile bounds. In Appendix \ref{app:in_sample_results}, the in-sample results are reported, complementing the ones from this section. All of our experiments are reproducible and are part of the examples in the InvOpt Python package \citep{zattoniscroccaro2023invopt}.

\subsection{Consistent data}
\label{sec:consistent_out}

In this section, we numerically evaluate the approaches discussed in Section \ref{sec:center} for IO problems with consistent data.

\textbf{Decision problem of the expert.} To generate its decisions, the expert solves the binary linear program
\begin{equation}
\label{eq:BLP}
\begin{aligned}
\min_{x} \quad & \langle \theta, x \rangle \\
\text{s.t.} \quad & Ax \leq b \\
& x \in \{0,1\}^n,
\end{aligned}
\end{equation}
where $\theta \in \mathbb{R}^n_+$ is the cost vector, $A \in \mathbb{R}^{t \times n}$, $b \in \mathbb{R}^t$, and $x$ is the decision vector. Although the decision problem \eqref{eq:BLP} is presented as a general binary linear program, many real-world problems can be modeled within this class of optimization problems. As a motivating example, we briefly discuss the problem of modeling consumer behavior, and how it fits \eqref{eq:BLP}. In this problem, the expert is a consumer, who given some contextual information, decides which products to buy from a set of $n$ products. That is, each component of the decision vector $x \in \{0, 1\}^n$ corresponds to one of the $n$ products, and equals $1$ if the consumer buys it, and $0$ otherwise. The signal with contextual information could be, for instance, $\hat{s} = (A, b)$, where each component of $A \in \mathbb{R}^n$ corresponds to the price of each product, and $b \in \mathbb{R}$ corresponds to the total budget of the consumer. In this case, the constraint $Ax \leq b$ simply represents the budget constraint of the consumer. Finally, each component of the cost vector $\theta \in \mathbb{R}^n$ represents (the negative of) the utility the consumer assigns to each product, and by minimizing $\langle \theta, x \rangle$, the goal of the consumer is to buy the set of products that maximizes its utility, while respecting its budget constraint. Thus, in this context, solving the IO problem translates to learning the utility function underlying the actions of a consumer agent.

\textbf{Data generation.} To generate training and test data, we sample cost vectors uniformly from $\{\theta \in \mathbb{R}^n : 0 \leq \theta \leq \mathbbm{1}\}$, we sample $\hat{A}$ uniformly from $\{A \in \mathbb{R}^{t \times n} : -1 \leq A \leq 0\}$ and $\hat{b}$ uniformly from $\{b \in \mathbb{R}^t : -\mathbbm{1} \leq b \leq 0\}$. To make sure the problem instances are feasible, we check if the sum of each row of $\hat{A}$ is larger than the respective component of $\hat{b}$, which ensures that $x = (1, \ldots, 1)$ is a feasible solution of \eqref{eq:BLP}. Given a signal $\hat{s} = (\hat{A}, \hat{b})$, we generate a response $\hat{x}$ by solving \eqref{eq:BLP}, i.e., $\hat{x} \in \argmin_{x \in \mathbb{X}(\hat{s})} \langle \theta, x \rangle$. To evaluate the IO approaches, we generate 10 random cost vectors $\theta_{\text{true}}$. For each of these cost vectors, we generate a training dataset $\widehat{\mathcal{D}}_{\text{train}} = \{(\hat{s}_i, \hat{x}_i)\}_{i=1}^N$ and a test dataset $\widehat{\mathcal{D}}_{\text{test}} = \{(\hat{s}_i, \hat{x}_i)\}_{i=1}^N$, with $N=100$, $n=6$ and $t=4$.

\textbf{IO approaches.} We compare seven approaches to choose a vector from the set of consistent cost vectors \eqref{eq:consistent}:
\begin{itemize}
    \item \textbf{Feasibility}: we use \eqref{eq:feasibility} with $\|\theta\|_1 = 1$, $s = (A, b)$, $\mathbb{X}(s) = \{x \in \{0,1\}^n : Ax \leq b\}$, $\phi(x,s) = x$, augmented with the constraint $\theta \in \Theta = \{\theta \in \mathbb{R}^n : \theta \geq 0\}$;
    \item \textbf{Incenter}: we use \eqref{eq:incenter_gen} with $\mathcal{R}(\theta) = \frac{1}{2}\|\theta\|_2^2$, $s = (A, b)$, $\mathbb{X}(s) = \{x \in \{0,1\}^n : Ax \leq b\}$, $\phi(x,s) = x$, $d(\hat{x}_i,x_i) = \|\hat{x}_i - x_i\|_2$, and $\Theta = \{\theta \in \mathbb{R}^n : \theta \geq 0\}$;
    \item \textbf{Circumcenter}: we solve the following optimization problem:
    \begin{equation}
        \label{eq:circumcenter3}
        \begin{aligned}
        \min_{\|\theta\|_2 = 1} & \quad r \\
        \text{s.t.} \ & \ \max_{\substack{\tilde{\theta} \in \mathbb{C} \\ \|\tilde{\theta}\|_2 = 1}} \|\theta - \tilde{\theta}\|_2^2 \leq r,
        \end{aligned}
    \end{equation}
    which is an epigraph reformulation of the circumcenter problem (see the proof of Theorem \ref{theo:np_hard_circumcenter}). To solve this optimization problem, we substitute the max constraint by the constraints $\| \theta_E - \tilde{\theta} \|_2^2 \leq r, \ \forall \theta_E \in E$, where $E$ is the set of normalized extreme points of $\mathbb{C}$;
    \item \textbf{Ellip. incenter}: we solve \eqref{eq:ellipsoidal_incenter}, the ellipsoidal version of the incenter concept. In our implementation, we use the convex constraint $\|\theta\| \leq 1$ instead of $\|\theta\| = 1$, since one can show that the constraint $\|\theta\| \leq 1$ will always be tight at an optimum when $\text{int}(\mathbb{C}) \neq \emptyset$;
    \item \textbf{Ellip. circumcenter}: we solve
    \begin{equation*}
        \begin{aligned}
        \max_{A \succcurlyeq 0, \theta} & \quad \log\det (A) \\
        \text{s.t.} \ & \ \max_{\substack{\tilde{\theta} \in \mathbb{C} \\ \|\tilde{\theta}\|_2 = 1}} \|A(\tilde{\theta} - \theta)\|_2^2 \leq 1,
        \end{aligned}
    \end{equation*}
    which is an ellipsoidal generalization of the circumcenter reformulation \eqref{eq:circumcenter3};
    \item \textbf{Cutting plane}: we use the cutting plane algorithm of \citep{wang2009cutting};
    \item \textbf{Predictability loss}: we use the predictability loss of \citep{aswani2018inverse} (Appendix \ref{app:theoretical}).
\end{itemize}

\textbf{Results.} Figure \ref{fig:consistent_out} shows the results for this scenario. To make the comparisons consistent, before evaluating the results, all cost vectors are normalized. For all the plots, the x-axis refers to the number of training examples used to compute the results. The idea is to evaluate how efficient these approaches are with respect to the amount of data fed to them. Figure \ref{fig:consistent_theta_diff} shows the difference between the cost vector returned by the IO approaches (which we name $\theta_{\text{IO}}$) and the cost vector used to generate the data ($\theta_{\text{true}}$). As expected, the more data we feed to the approaches, the closer $\theta_{\text{IO}}$ tends to get from $\theta_{\text{true}}$. Comparing them, the incenter, ellipsoidal incenter, and ellipsoidal circumcenter approaches have similar performance and clearly outperform the other approaches. Figure \ref{fig:consistent_x_diff_out} shows the average difference between the optimal decision of \eqref{eq:BLP} using $\theta_{\text{IO}}$ (which we name $x_{\text{IO}}$) and the decision in the test dataset (which we name $x_{\text{true}}$). Again, we see that the incenter, ellipsoidal incenter, and ellipsoidal circumcenter approaches present the best performance. Figure \ref{fig:consistent_obj_diff_out} shows the normalized difference between the cost of the expert decisions and the cost of the decisions using $\theta_{\text{IO}}$. More precisely, we define $\text{Cost}_{\text{IO}} \coloneqq \sum_{i=1}^N \langle \theta_{\text{true}}, x_{\text{IO},i} \rangle$ and $\text{Cost}_{\text{true}} \coloneqq \sum_{i=1}^N \langle \theta_{\text{true}}, \hat{x}_i \rangle$ and compare the relative difference between them. Notice that this difference will always be nonnegative by the optimality of $\hat{x}_i$. Once again, the incenter, ellipsoidal incenter, and ellipsoidal circumcenter approaches outperform the other approaches.

Moreover, Table \ref{table:time_consistent} shows the time it took to generate the results of this section for each approach. As can be seen, even though the incenter, ellipsoidal incenter, and ellipsoidal circumcenter approaches show similar performance, the incenter approach is at least one order of magnitude faster to compute. Although these numbers depend on the actual implementation of each IO approach, such a difference in solving time is expected, since the incenter problem can be formulated as a quadratic program, whereas the ellipsoidal approaches involve semidefinite constraints. Finally, regarding the main driver for improved out-of-sample performance of the Incenter and the ellipsoidal approaches, one possible explanation is that these approaches optimize for a vector in the interior of $\mathbb{C}$, that is, away from the boundaries of this set. This conclusion is based on the following observations: recall the intuition for the incenter vector provided in Remark \ref{remark:robustness}, also visualized in Figure \ref{fig:incenter}. That is, the idea behind the incenter is to find the vector furthest away from the boundaries of the set $\mathbb{C}$. Also, recall that \cite{besbes2023contextual} showed that circumcenter concept fails for online IO problems precisely because, in the worst case, the circumcenter vector can lie exactly in the boundary of the set $\mathbb{C}$ \cite[Figure 2]{besbes2023contextual}, and the ellipsoidal circumcenter addresses this problem since its circumcenter lies inside the set $\mathbb{C}$ \cite[Figure 3]{besbes2023contextual}. Therefore, the incenter, ellipsoidal incenter, and ellipsoidal circumcenter in some sense optimize for vector in the interior of $\mathbb{C}$, and as shown in Figure \ref{fig:consistent_out}, achieve similar out-of-sample performance in this offline IO experiment (although the incenter vector is computationally cheaper to compute).

\begin{table}
\centering
\begin{tabular}{c | c c c c c c c c c c c c c c c c} 
\textbf{Approach} & Feasibility & Incenter & Circumcenter & Ellip. incenter \\
\hline
\textbf{Time (seconds)} & 69 & 98 & 1323 & 915 \\
\textbf{Approach} & Ellip. circumcenter & Cutting plane & Predictability loss \\
\hline
\textbf{Time (seconds)} & 1345 & 206 & 236 \\
\end{tabular}
\caption{Computational time to generate the results of Figure \ref{fig:consistent_out}.}
\label{table:time_consistent}
\end{table}

\begin{figure}
\centering
\captionsetup[subfigure]{width=0.96\linewidth}%
    \begin{subfigure}[t]{0.32\linewidth}
        \includegraphics[width = \linewidth]{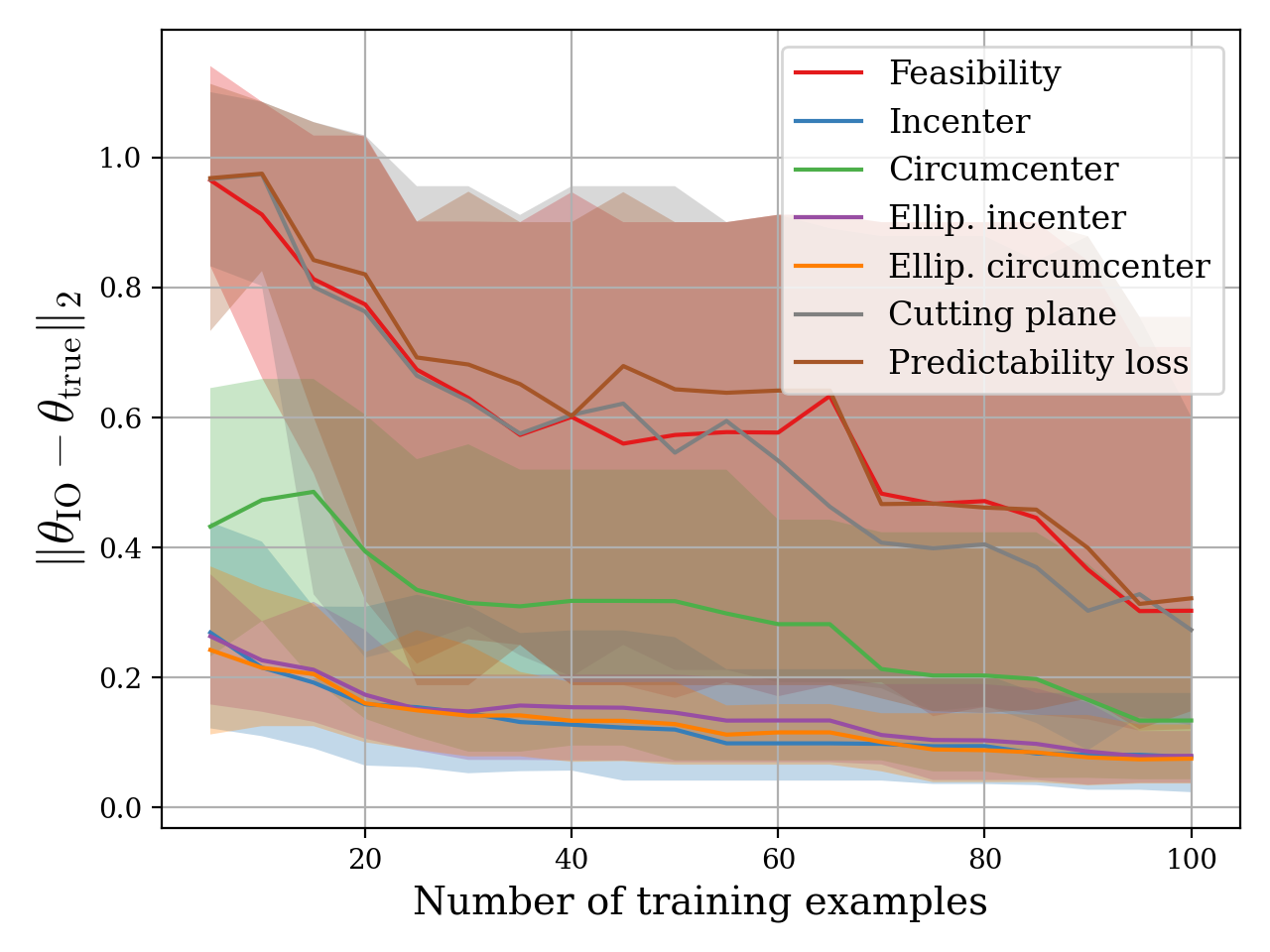}
        \caption{Difference between the true cost vector ($\theta_{\text{true}}$) and the one learned using IO ($\theta_{\text{IO}}$).}
        \label{fig:consistent_theta_diff}
    \end{subfigure}
    \begin{subfigure}[t]{0.32\linewidth}
        \includegraphics[width = \linewidth]{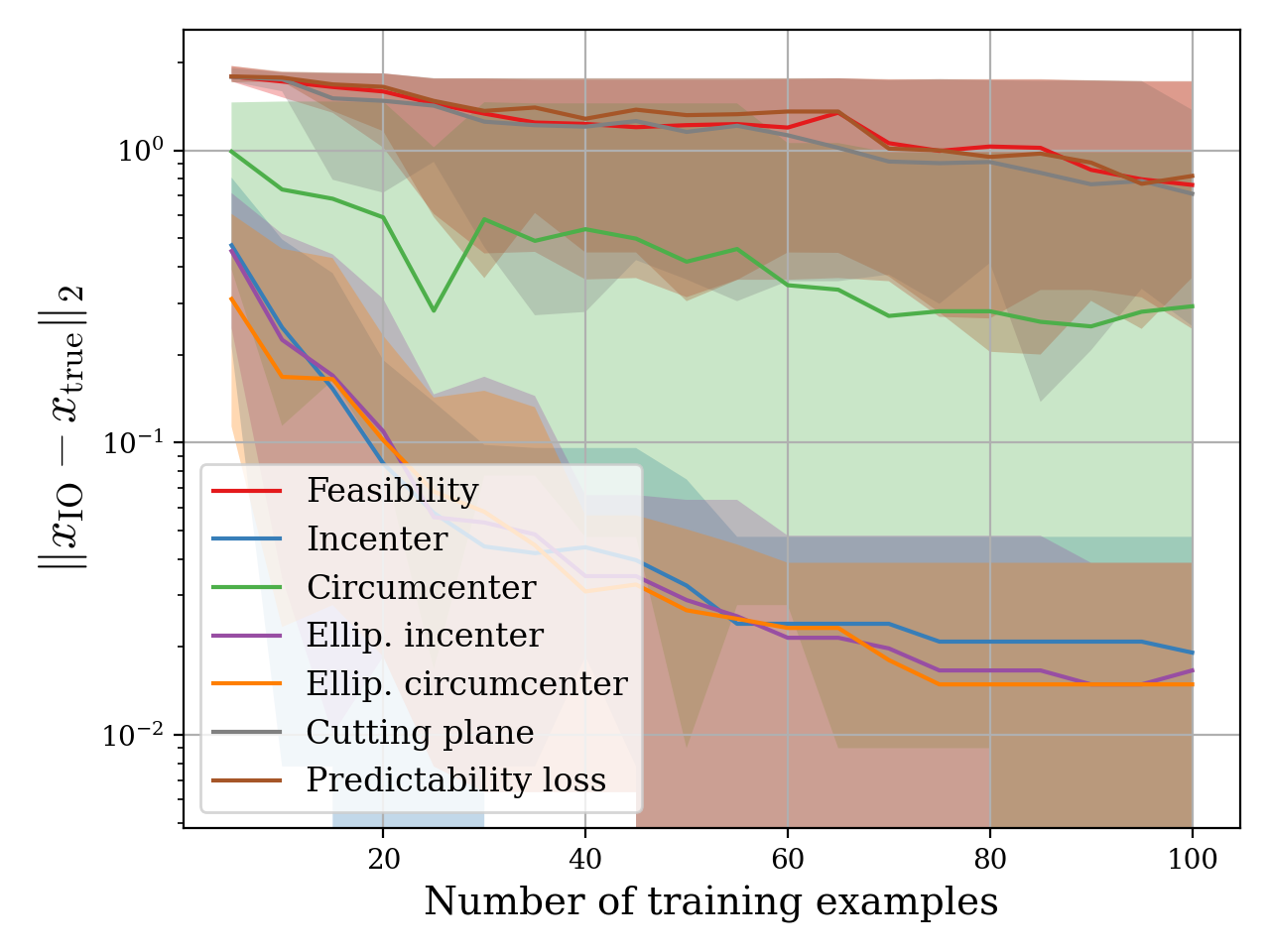}
        \caption{Average error between the decision generated by $\theta_{\text{true}}$ and $\theta_{\text{IO}}$.}
        \label{fig:consistent_x_diff_out}
    \end{subfigure}
    \begin{subfigure}[t]{0.32\linewidth}
        \includegraphics[width = \linewidth]{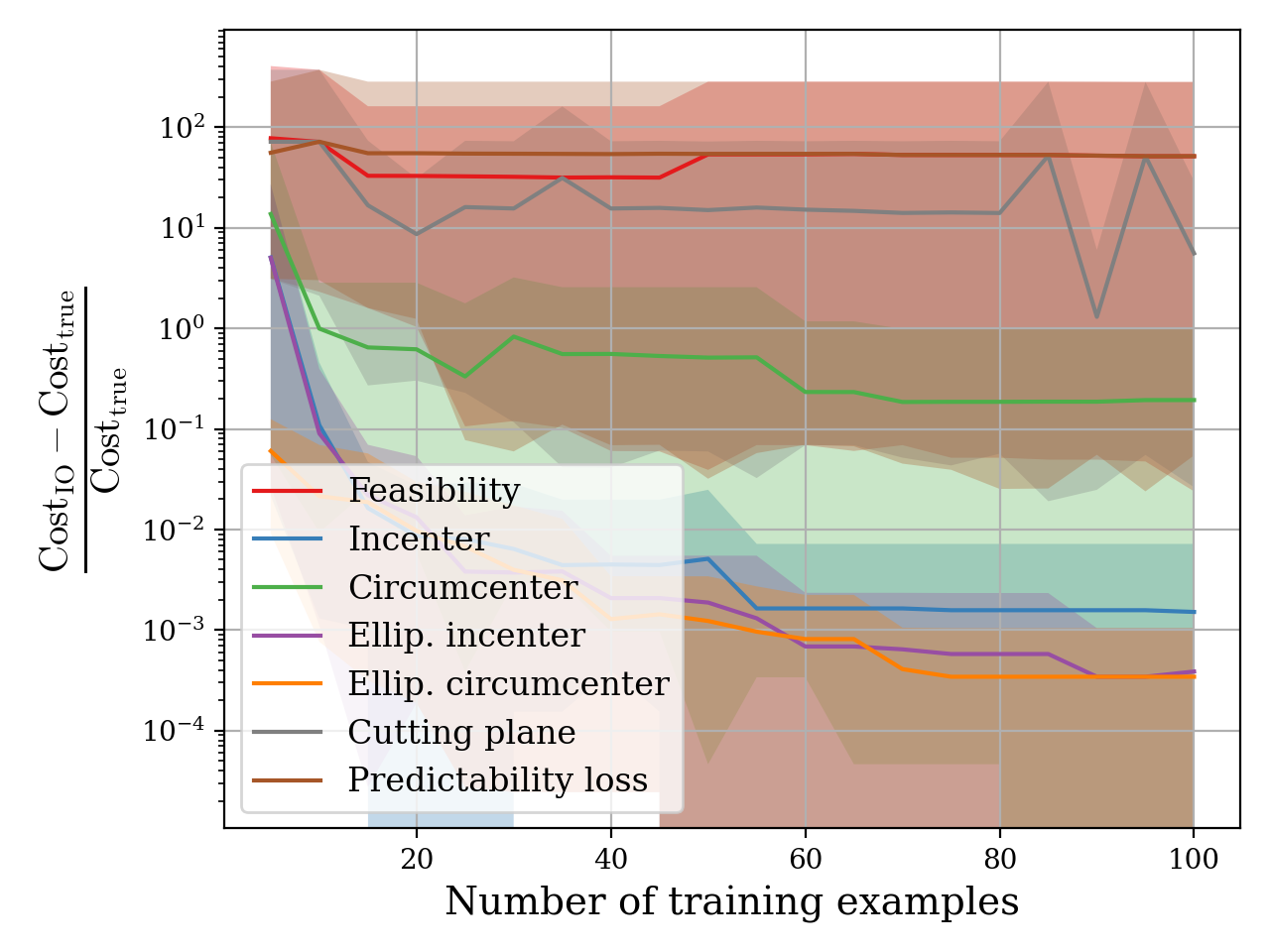}
        \caption{Relative difference between the cost of the decisions generated using $\theta_{\text{true}}$ and $\theta_{\text{IO}}$.}
        \label{fig:consistent_obj_diff_out}
    \end{subfigure}
\caption{Out-of-sample results for consistent data scenario.}
\label{fig:consistent_out}
\end{figure}

\subsection{Inconsistent data}
\label{sec:inconsistent_data}

In this section, we numerically evaluate the approaches discussed in Section \ref{sec:ASL} for IO problems with inconsistent data.

\textbf{Decision problem of the expert.} To generate its decisions, the expert solves the binary linear program \eqref{eq:BLP}, where $\theta \in \mathbb{R}^n$ is the cost vector, $A \in \mathbb{R}^{t \times n}$ and $b \in \mathbb{R}^t$, and $x$ is the decision vector. Notice that different from the previous section, we do not assume the cost vector $\theta$ is nonnegative.

\textbf{Data generation.} To generate training and test data, we sample cost vectors uniformly from $\{\theta \in \mathbb{R}^n : -\mathbbm{1} \leq \theta \leq \mathbbm{1}\}$, we sample $\hat{A}$ uniformly from $\{A \in \mathbb{R}^{t \times n} : -1 \leq A \leq 1\}$ and $\hat{b}$ uniformly from $\{b \in \mathbb{R}^t : -\mathbbm{1} \leq b \leq 0\}$. After generating a signal $\hat{s} = (\hat{A}, \hat{b})$, we check if $\mathbb{X}(\hat{s})$ is nonempty to ensure the problem instance is feasible. To generate the response vectors $\hat{x}$ for the training datasets, we solve problem \eqref{eq:BLP} with noise added to the cost vector, i.e., $\hat{x} \in \argmin_{x \in \mathbb{X}(\hat{s})} \langle \theta + w, x \rangle$, where $w \in \mathbb{R}^n$ is a random vector with components sampled from a normal distribution with zero mean and standard deviation equal to $0.05$. This means that different from the consistent data case, (most probably) there will be no single cost vector consistent with the entire training dataset. To evaluate the IO approaches, we generate 10 random cost vectors $\theta_{\text{true}}$. For each of these cost vectors, we generate a training dataset $\widehat{\mathcal{D}}_{\text{train}} = \{(\hat{s}_i, \hat{x}_i)\}_{i=1}^N$ (by solving the noisy version of \eqref{eq:BLP}, with a different noise vector $w$ for each signal-response pair) and a test dataset $\widehat{\mathcal{D}}_{\text{test}} = \{(\hat{s}_i, \hat{x}_i)\}_{i=1}^N$ (noiseless), with $N=100$, $n=10$ and $t=8$.

\textbf{IO approaches.} We compare five IO approaches for this problem:
\begin{itemize}
    \item \textbf{Suboptimality Loss (SL)}: we use \eqref{eq:reg_loss_minimization} with $\kappa=0$, $\Theta = \mathbb{R}^n$, $s = (A, b)$, $\mathbb{X}(s) = \{x \in \{0,1\}^n : Ax \leq b\}$, $\phi(x,s) = x$, and the suboptimality loss $\ell_\theta (\hat{s},\hat{x}) = \max_{x \in \mathbb{X}(\hat{s})} \left\{ \inner{\theta}{\phi(\hat{s},\hat{x}) - \phi(\hat{s},x)} \right\}$. To prevent the trivial solution $\theta=0$, we add a norm equality constraint $\|\theta\|_\infty = 1$, and solve $2n$ linear programs, one for each facet of the $\infty$-norm unit sphere \citep{esfahani2018data};
    \item \textbf{Augmented Suboptimality Loss (ASL)}: we use \eqref{eq:reg_loss_minimization} with the \ref{eq:ASL}, $\kappa=0.001$, $\mathcal{R}(\theta) = \frac{1}{2}\|\theta\|_2^2$, $s = (A, b)$, $\mathbb{X}(s) = \{x \in \{0,1\}^n : Ax \leq b\}$, $\phi(x,s) = x$, $d(\hat{x}_i,x_i) = \|\hat{x}_i - x_i\|_2$, and $\Theta = \mathbb{R}^n$;
    \item \textbf{Ellipsoidal ASL}: we solve
    \begin{equation}
    \label{eq:ellipsoidal_ASL}
        \begin{aligned}
        \min_{A, \theta, \beta_1, \ldots, \beta_N} \quad & -\kappa \log \left( \det \left( A  \right)\right) + \frac{1}{N}\sum_{i=1}^N \beta_i \\
        \text{s.t.} \ \quad \quad & \inner{\theta}{\phi(\hat{s}_i,\hat{x}_i) - \phi(\hat{s}_i,x_i)} + \| A(\phi(\hat{s}_i,x_i) - \phi(\hat{s}_i,\hat{x}_i)) \|_2 \leq \beta_i \quad \forall x_i \in \mathbb{X}(\hat{s}_i), \ \forall i \in [N] \\
        &\|\theta\|_2 \leq 1, \quad A \succcurlyeq 0,
    \end{aligned}
    \end{equation}
    which is an ellipsoidal generalization of \eqref{eq:incenter_gen_inconsistent} (see Remark \ref{remark:extremal_volume}).
    \item \textbf{Cutting plane}: we use the cutting plane method of \citep{bodur2022inverse}, which is a extension of the cutting-plane algorithm of \citep{wang2009cutting} for IO problems with inconsistent data;
    \item \textbf{Predictability loss}: we use the predictability loss of \citep{aswani2018inverse} (Appendix \ref{app:theoretical}).
\end{itemize}
Notice that we do not test circumcenter-based approaches since they are not defined for the inconsistent data case, that is, when $\mathbb{C} \setminus \{0\} = \emptyset$.

\textbf{Results.} Figure \ref{fig:inconsistent_out} shows the results this scenario. The discussion on the interpretation of each performance metric presented in Figure \ref{fig:inconsistent_out} mirrors the one of Figure \ref{fig:consistent_out}. Importantly, the approach based on the \ref{eq:ASL} and its ellipsoidal version outperforms the other approaches. Moreover, as expected, due to the noise in the training data, one can see that $\theta_{\text{IO}}$ was not able to reproduce the behavior of the expert as well as in the consistent (i.e., noiseless) case. In Table \ref{table:time_consistent} we show the time it took to generate the results of this section for each approach. As can be seen, even though the \ref{eq:ASL} and Ellipsoidal ASL show similar performance, the \ref{eq:ASL} approach is more computationally efficient. Again, this difference is because the Ellipsoidal ASL involves semidefinite constraints (Eq. \eqref{eq:ellipsoidal_ASL}), which is not the case for the standard \ref{eq:ASL} (Eq. \eqref{eq:incenter_gen_inconsistent}).

\begin{table}
\centering
\begin{tabular}{c | c c c c c} 
\textbf{Approach} & SL & ASL & Ellipsoidal ASL & Cutting plane & Predictability loss \\
\hline
\textbf{Time (seconds)} & 230 & 241 & 351 & 440 & 6192
\end{tabular}
\caption{Computational time to generate the results of Figure \ref{fig:inconsistent_out}.}
\label{table:time_inconsistent}
\end{table}

\begin{figure}
\centering
\captionsetup[subfigure]{width=0.96\linewidth}%
    \begin{subfigure}[t]{0.32\linewidth}
        \includegraphics[width = \linewidth]{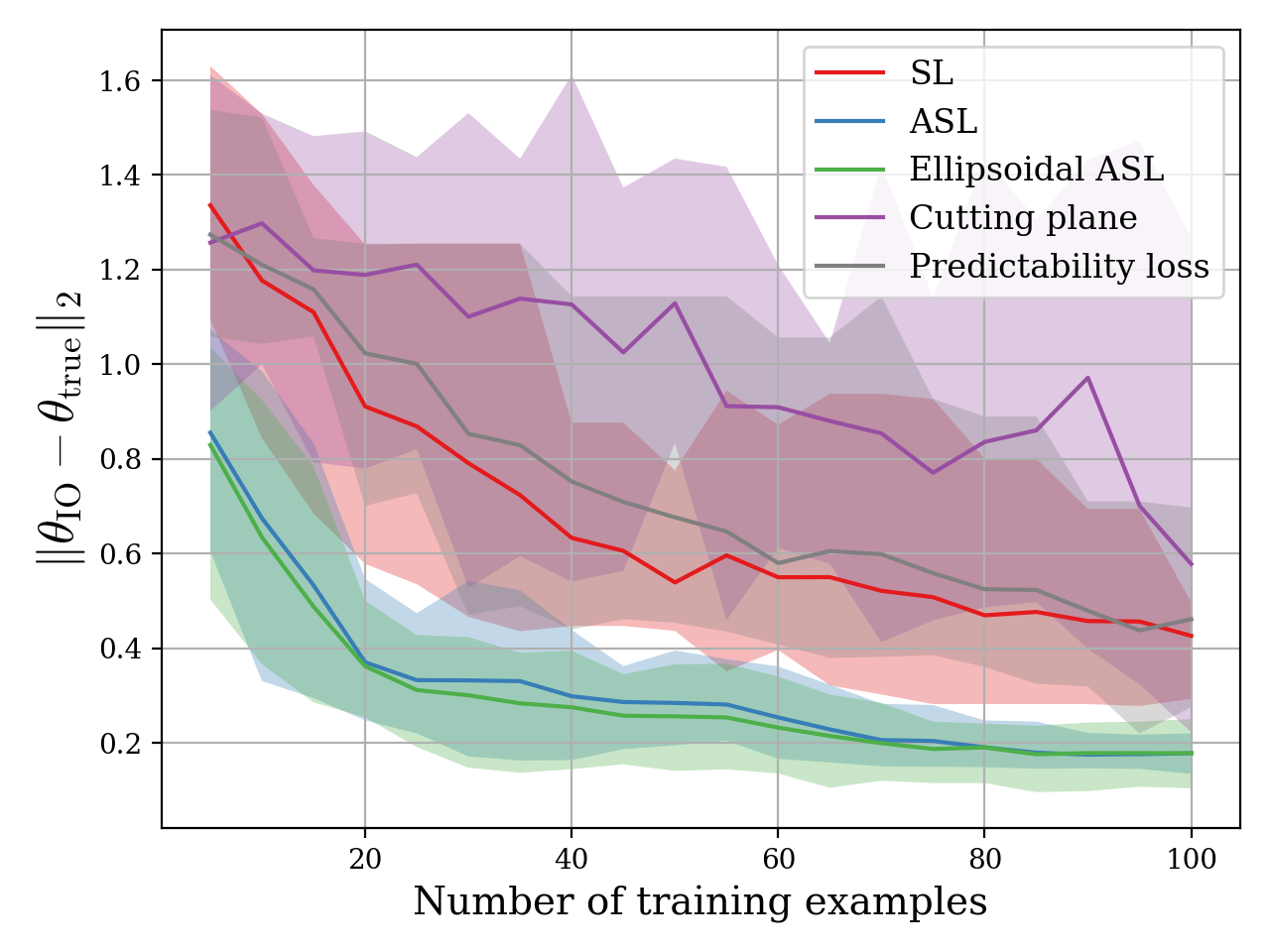}
        \caption{Difference between the true cost vector ($\theta_{\text{true}}$) and the one learned using IO ($\theta_{\text{IO}}$).}
        \label{fig:inconsistent_theta_diff}
    \end{subfigure}
    \begin{subfigure}[t]{0.32\linewidth}
        \includegraphics[width = \linewidth]{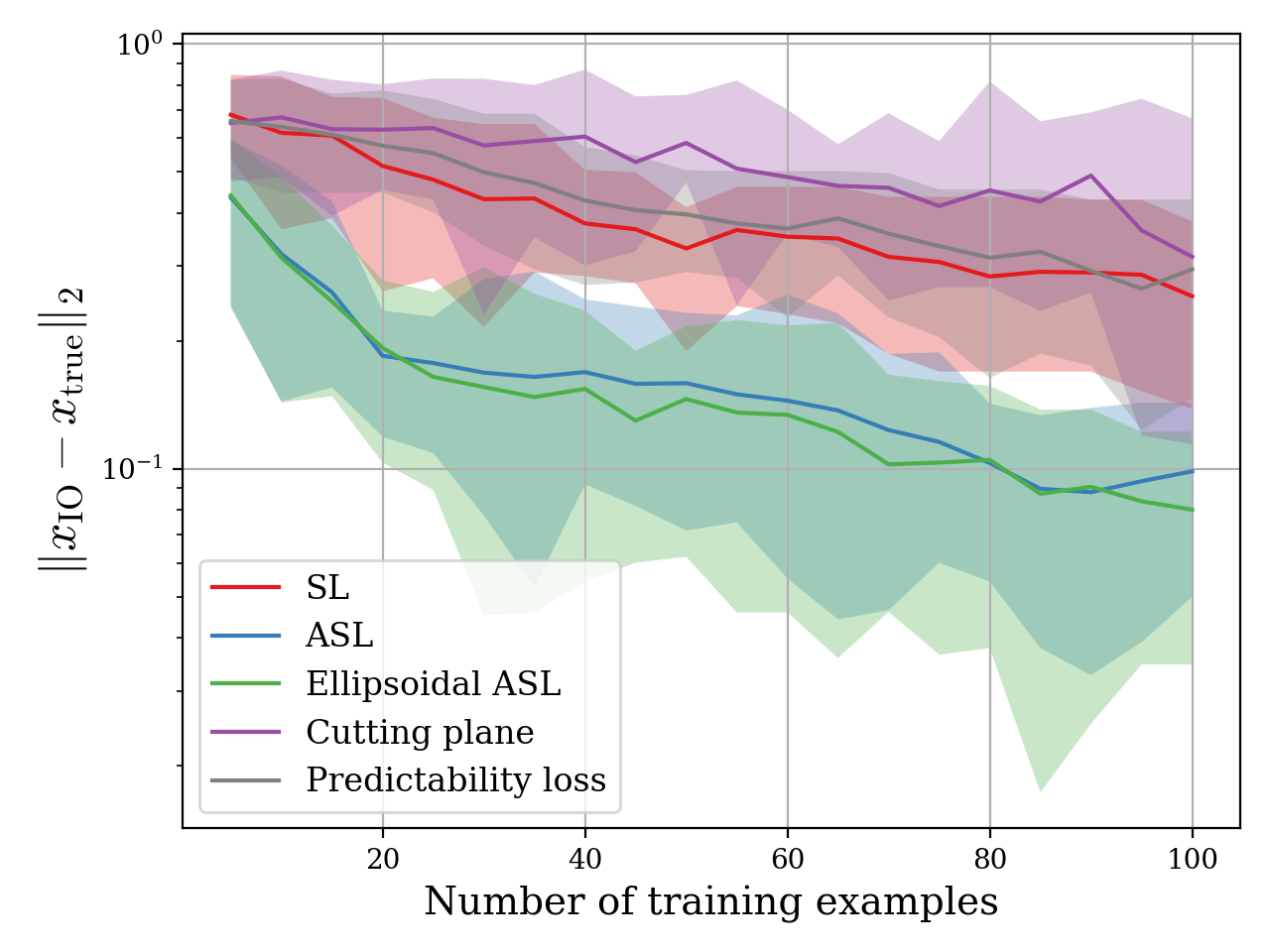}
        \caption{Average error between the decision generated by $\theta_{\text{true}}$ and $\theta_{\text{IO}}$.}
        \label{fig:inconsistent_x_diff_out}
    \end{subfigure}
    \begin{subfigure}[t]{0.32\linewidth}
        \includegraphics[width = \linewidth]{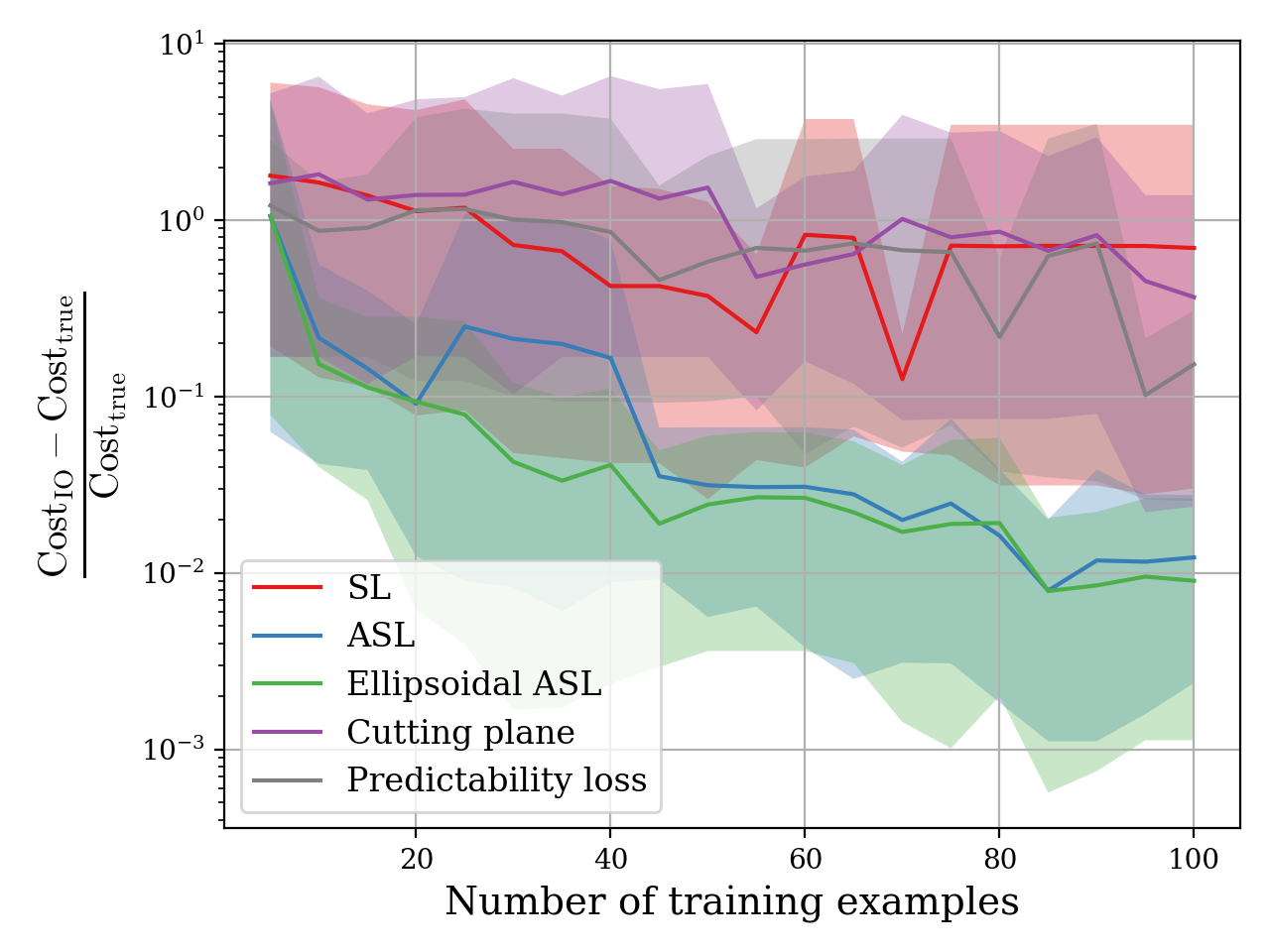}
        \caption{Relative difference between the cost of the decisions generated using $\theta_{\text{true}}$ and $\theta_{\text{IO}}$.}
        \label{fig:inconsistent_obj_diff_out}
    \end{subfigure}
\caption{Out-of-sample results for inconsistent data scenario.}
\label{fig:inconsistent_out}
\end{figure}

\subsection{Mixed-integer feasible set}
\label{sec:MI}

In this section, we numerically evaluate the approach from Section \ref{sec:mixed} for IO problems with mixed-integer feasible sets.

\textbf{Decision problem of the expert.} To generate its decisions, the expert solves a mixed-integer linear program of the form
\begin{equation}
\label{eq:MILP}
\begin{aligned}
\min_{y,z} \quad & \langle q_y, y \rangle + \langle q_z, z \rangle \\
\text{s.t.} \quad & Ay + Bz \leq c \\
& 0 \leq y \leq \mathbbm{1}, \ z \in \{0,1\}^v,
\end{aligned}
\end{equation}
where $\theta \coloneqq (q_y, q_z) \in \mathbb{R}^{u+v}_+$ is the cost vector $A \in \mathbb{R}^{t \times u}$, $B \in \mathbb{R}^{t \times v}$ and $b \in \mathbb{R}^t$.

\textbf{Data generation.} To generate training and test data, we sample the cost vector uniformly from $\{\theta \in \mathbb{R}^{u+v} : 0 \leq \theta \leq \mathbbm{1}\}$, we sample $\hat{A}$ uniformly from $\{A \in \mathbb{R}^{t \times u} : -1 \leq A \leq 0\}$, $\hat{B}$ uniformly from $\{B \in \mathbb{R}^{t \times v} : -1 \leq B \leq 0\}$ and $\hat{c}$ uniformly from $\{c \in \mathbb{R}^t : -2 \leq c \leq 0\}$. To make sure the problem instances are feasible, we checked if the sum of each row of $[\hat{A} \ \hat{B}]$ is smaller than the respective component of $\hat{c}$. Given a tuple $(\hat{A}, \hat{B}, \hat{c})$, we generate a response $\hat{x} = (\hat{y}, \hat{z})$ by solving problem \eqref{eq:MILP}. To evaluate the IO approaches, we generate 10 random cost vectors $\theta_{\text{true}}$. For each of these cost vectors, we generate a training dataset $\widehat{\mathcal{D}}_{\text{train}} = \{(\hat{s}_i, \hat{x}_i)\}_{i=1}^N$ and a test dataset $\widehat{\mathcal{D}}_{\text{test}} = \{(\hat{s}_i, \hat{x}_i)\}_{i=1}^N$, with $N=100$, $u=v=6$ and $t=4$.

\textbf{IO approaches.} We compare six IO approaches for this problem:
\begin{itemize}
    \item \textbf{Suboptimality Loss (SL)}: we use \eqref{eq:MI_reformulation_linear} with no regularization $\mathcal{R}$ and no distance function (i.e., $\kappa = 0$ and $d(\hat{x}_i,x_i) = 0$). To avoid the trivial solution $\theta = 0$, we add the constraint $\|\theta\|_1 = 1$. We call this approach SL since it is the result of performing the reformulation steps of Corollary \ref{coro:MI_reformulation_linear} using the Suboptimality loss instead of the \ref{eq:ASL};
    \item \textbf{ASL-yz}: we use \eqref{eq:MI_reformulation_linear} with $\theta = (q_y, q_z)$, $s = (A, B, c, 0)$ ($A$ and $c$ have to be augmented to account for the constraints $0 \leq y \leq \mathbbm{1}$), $\mathbb{Z}(w) = \{0,1\}^v$, $Q_{yy} = 0$, $\phi_1(w,z) = 1$, $\phi_2(w,z) = z$, $\kappa=0$, $\Theta = \{\theta \in \mathbb{R}^{u+v} : \theta \geq 0 \}$, and $d_z(\hat{x}_i,x_i) = \|\hat{z}_i - z_i\|_2$.
    \item \textbf{ASL-z}: same as \textbf{ASL-yz}, but with $h_k = 0 \ \forall k \in [2u]$ (see discussion in the paragraph after Corollary \ref{coro:MI_reformulation_linear});
    \item \textbf{Circumcenter}: we solve \eqref{eq:circumcenter3};
    \item \textbf{Cutting plane}: we use the cutting plane algorithm of \citep{wang2009cutting};
    \item \textbf{Predictability loss}: we use the predictability loss of \citep{aswani2018inverse} (Appendix \ref{app:theoretical}).
\end{itemize}

\textbf{Results.} Figure \ref{fig:MILP_out} shows the results for this scenario. The discussions on the interpretations of the results of this section mirror the ones of Figure \ref{fig:consistent_out} from Section \ref{sec:consistent_out}. Once again, the approach based on the \ref{eq:ASL} outperforms the other approaches on all three performance metrics. In particular, the ASL-z and ASL-yz show similar performance, with the ASL-yz being slightly better overall and when there are very few training examples. However, because the ASL-yz optimization problem has $2N$ times more constraints than the ASL-z optimization problem, it is much more costly to solve, as can be seen from the computational times in Table \ref{table:time_MILP}.

\begin{table}
\centering
\begin{tabular}{c | c c c c c c} 
\textbf{Approach} & SL & ASL-z & ASL-yz & Circumcenter & Cutting plane & Predictability loss \\
\hline
\textbf{Time (seconds)} & 607 & 694 & 7060 & 5954 & 981 & 1336
\end{tabular}
\caption{Computational time to generate the results of Figure \ref{fig:MILP_out}.}
\label{table:time_MILP}
\end{table}

\begin{figure}
\centering
\captionsetup[subfigure]{width=0.96\linewidth}%
    \begin{subfigure}[t]{0.32\linewidth}
        \includegraphics[width = \linewidth]{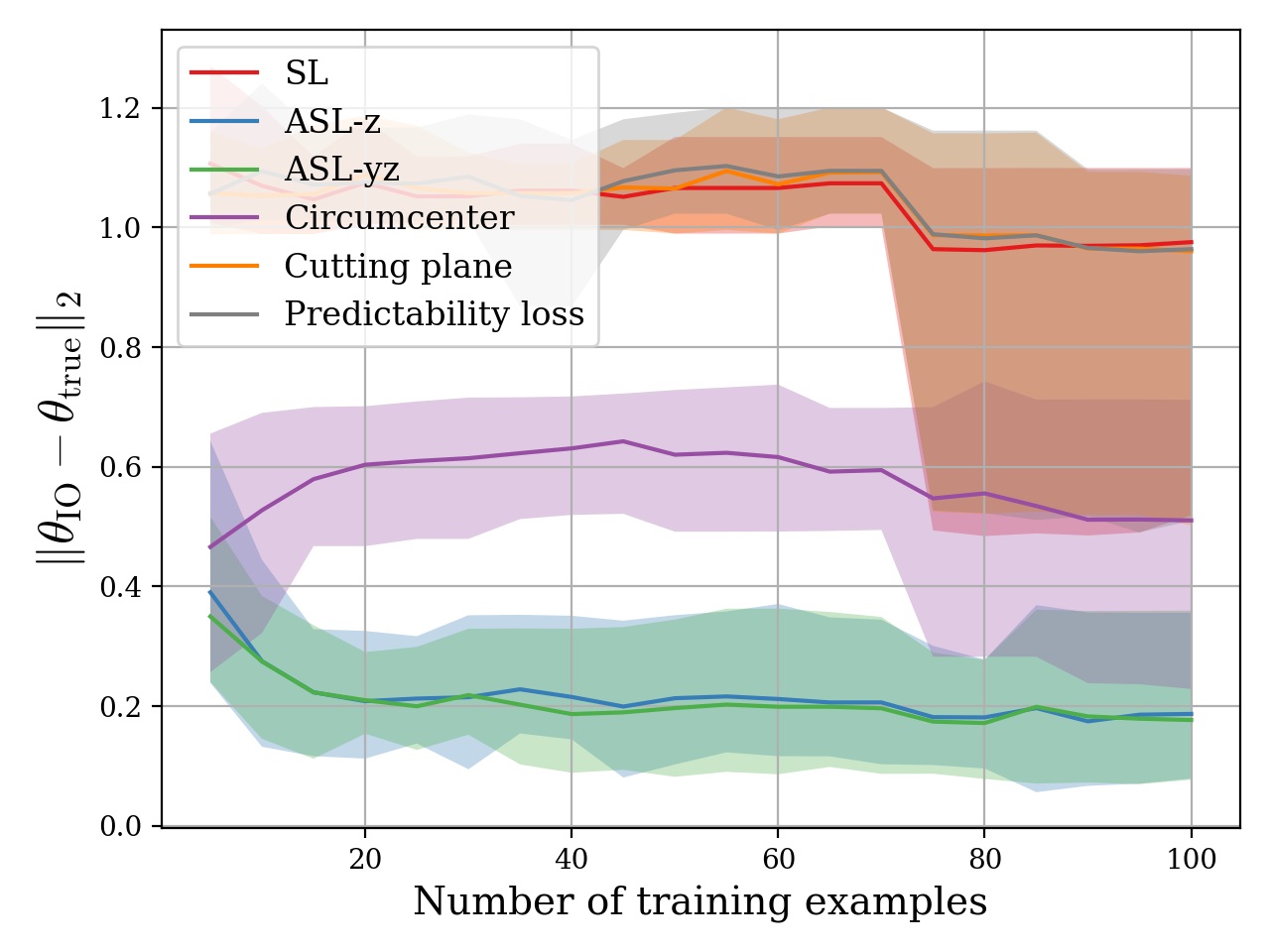}
        \caption{Difference between the true cost vector ($\theta_{\text{true}}$) and the one learned using IO ($\theta_{\text{IO}}$).}
        \label{fig:MILP_theta_diff}
    \end{subfigure}
    \begin{subfigure}[t]{0.32\linewidth}
        \includegraphics[width = \linewidth]{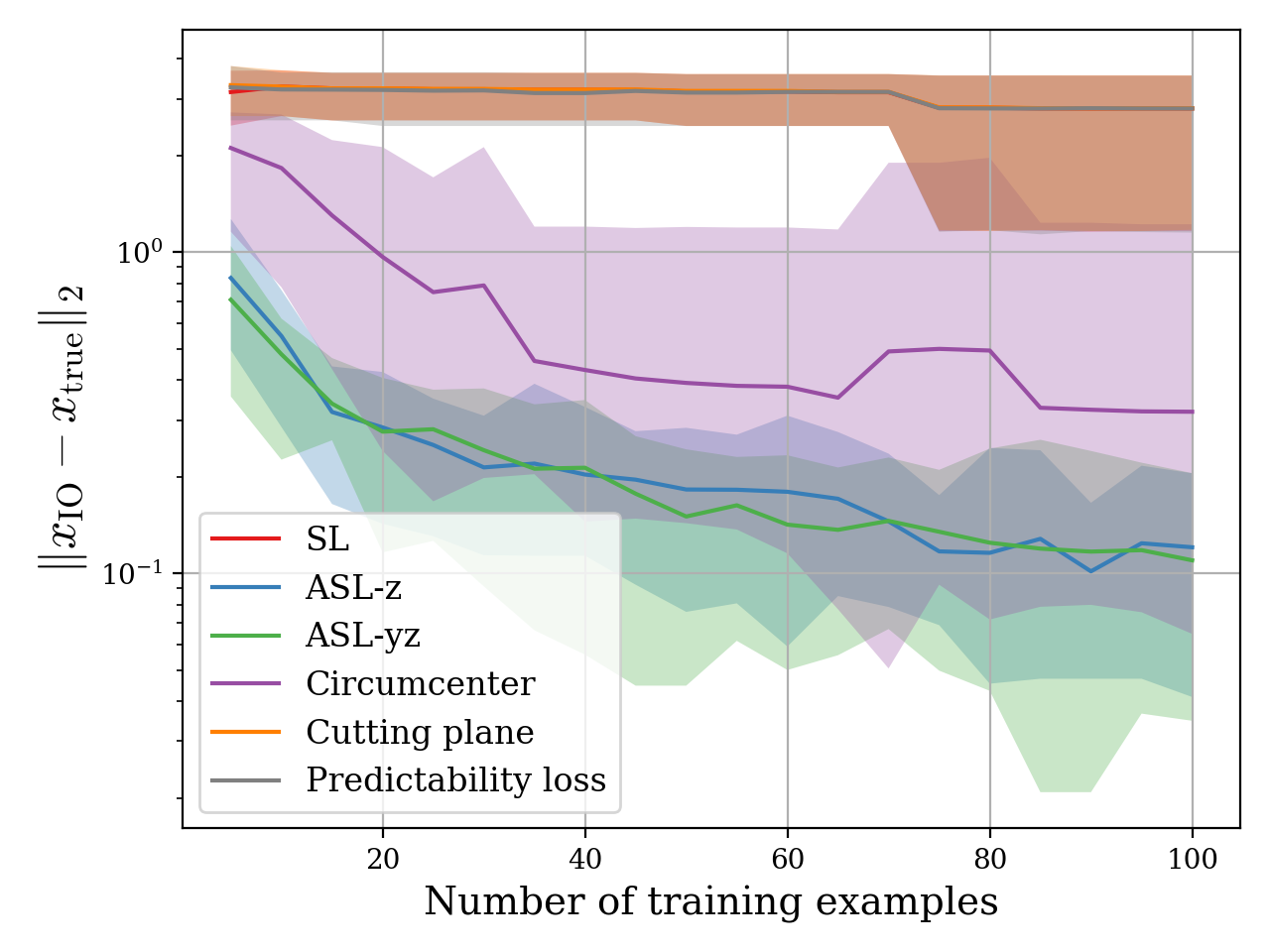}
        \caption{Average error between the decision generated by $\theta_{\text{true}}$ and $\theta_{\text{IO}}$.}
        \label{fig:MILP_x_diff_out}
    \end{subfigure}
    \begin{subfigure}[t]{0.32\linewidth}
        \includegraphics[width = \linewidth]{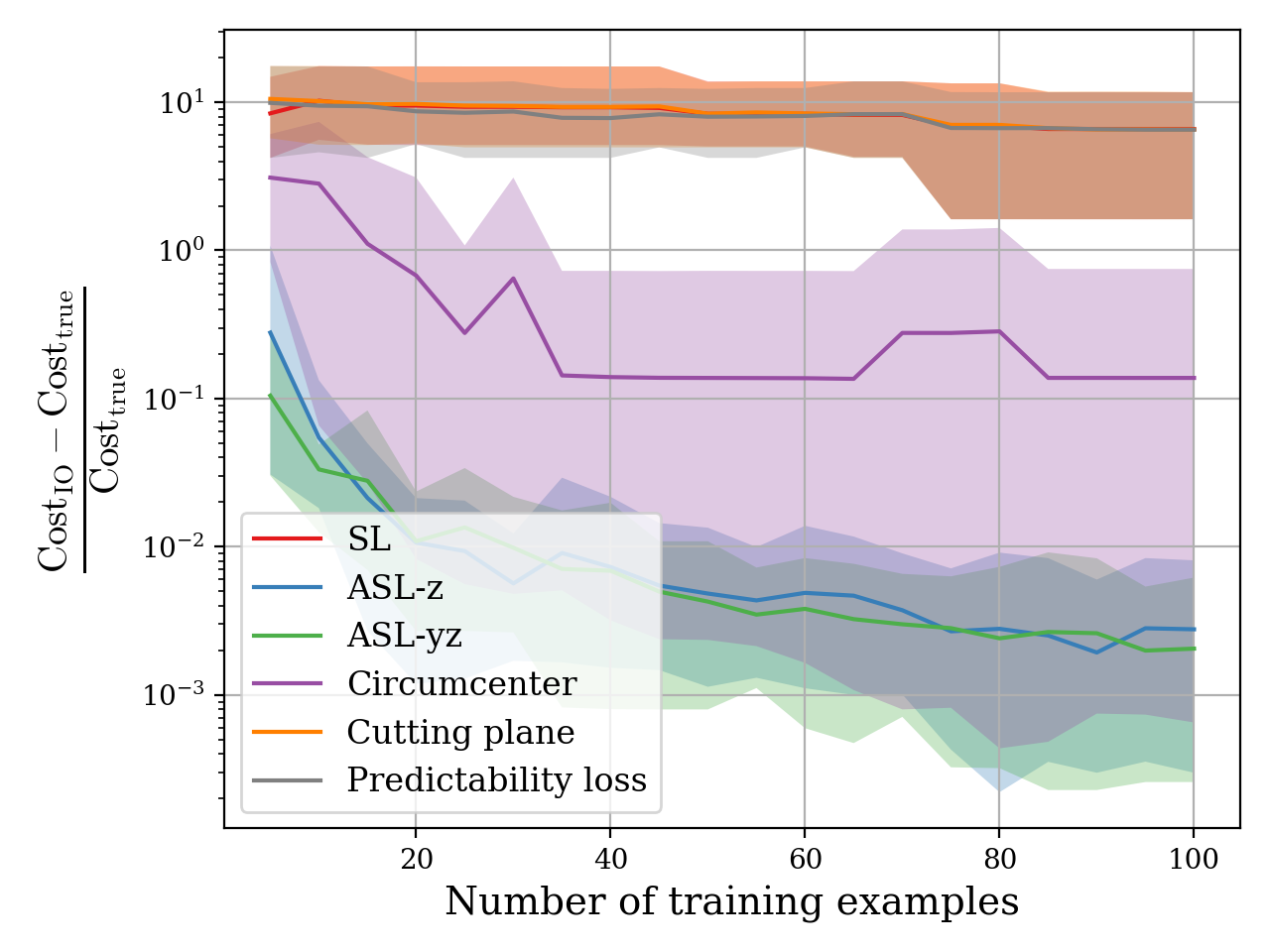}
        \caption{Relative difference between the cost of the decisions generated using $\theta_{\text{true}}$ and $\theta_{\text{IO}}$.}
        \label{fig:MILP_obj_diff_out}
    \end{subfigure}
\caption{Out-of-sample results for the mixed-integer feasible set scenario.}
\label{fig:MILP_out}
\end{figure}

\subsection{Stochastic approximate mirror descent}
\label{sec:iter_algs}

In this section, we numerically evaluate the \ref{alg:SAMD} algorithm proposed in Section~\ref{sec:first-order}.

\textbf{Decision problem of the expert.} To generate its decisions, the expert solves the binary linear program \eqref{eq:BLP}, where $\theta \in \mathbb{R}^n_+$ is the cost vector, $A \in \mathbb{R}^{t \times n}$, $b \in \mathbb{R}^t$, and $x$ is the decision vector. We can write this optimization problem as $\min_{x \in \mathbb{X}(s)} F(s,x)$ by defining $s \coloneqq (A, b)$, $F(s,x) \coloneqq \langle \theta, x \rangle$ and $\mathbb{X}(s) \coloneqq \{x \in \{0,1\}^n : Ax \geq b\}$.

\textbf{Data generation.} To generate training and test data, we sample cost vectors uniformly from $\{\theta \in \mathbb{R}^n : 0 \leq \theta \leq \mathbbm{1}\}$, we sample $\hat{A}$ uniformly from $\{A \in \mathbb{R}^{t \times n} : -1 \leq A \leq 0\}$ and $\hat{b}$ uniformly from $\{b \in \mathbb{R}^t : -\frac{n}{3} \leq b \leq 0\}$. To make sure the problem instances are feasible, we checked if the sum of each row of $\hat{A}$ is larger than the respective component of $\hat{b}$. Given a signal $\hat{s}$, we generate a response $\hat{x}$ by solving \eqref{eq:BLP}, i.e., $\hat{x} \in \argmin_{x \in \mathbb{X}(\hat{s})} \langle \theta, x \rangle$. To evaluate the IO approaches, we generate 10 random cost vectors $\theta_{\text{true}}$. For each of these cost vectors, we generate a training dataset $\widehat{\mathcal{D}}_{\text{train}} = \{(\hat{s}_i, \hat{x}_i)\}_{i=1}^N$ and a test dataset $\widehat{\mathcal{D}}_{\text{test}} = \{(\hat{s}_i, \hat{x}_i)\}_{i=1}^N$, with $N=50$, $n=20$ and $t=15$.

\textbf{IO approaches.} We tackle this IO problem using \eqref{eq:reg_loss_minimization} with $\kappa = 0.01$, $\mathcal{R}(\theta) = \|\theta\|_1$, $\Theta = \{\theta \in \mathbb{R}^n : \theta \geq 0\}$ and the \ref{eq:ASL} with $s = (A, b)$, $\mathbb{X}(s) = \{x \in \{0,1\}^n : Ax \leq b\}$, $\phi(s,x) = x$ and $d(\hat{x},x) = \|x - \hat{x}\|_1$. When testing algorithms with exponentiated updates (i.e., mirror descent updates with $\omega (\theta) = \sum_{i=1}^n \theta_i\log(\theta_i)$), we solve the reformulation \eqref{eq:l1_reformulation}, with $\tilde{\kappa}$ chosen such that \eqref{eq:reg_loss_minimization} and \eqref{eq:l1_reformulation} have the same optimal solution. We compare eight algorithms, which result from all possible combinations of standard or mirror descent updates, deterministic or stochastic subgradients, and exact or approximate subgradients.
\begin{enumerate}[(i)]
    \item \textbf{Subgradient method (SM)}: Algorithm \ref{alg:IO_SAMD} with $\omega(\theta) = \frac{1}{2}\|\theta\|_2^2$ and exact subgradients computed using the entire dataset;
    \item \textbf{Mirror descent (MD)} with exponentiated updates: Algorithm \ref{alg:IO_SAMD} with $\omega (\theta) = \sum_{i=1}^n \theta_i\log(\theta_i)$ and exact subgradients computed using the entire dataset;
    \item \textbf{Stochastic subgradient method (SSM)}: Algorithm \ref{alg:IO_SAMD} with $\omega(\theta) = \frac{1}{2}\|\theta\|_2^2$ and exact stochastic subgradients;
    \item \textbf{Approximated subgradient method (ASM)}: Algorithm \ref{alg:IO_SAMD} with $\omega(\theta) = \frac{1}{2}\|\theta\|_2^2$ and approximate subgradients computed using the entire dataset;
    \item \textbf{Stochastic mirror descent (SMD)} with exponentiated updates: Algorithm \ref{alg:IO_SAMD} with $\omega (\theta) = \sum_{i=1}^n \theta_i\log(\theta_i)$ and exact stochastic subgradients;
    \item \textbf{Approximate mirror descent (AMD)} with exponentiated updates: Algorithm \ref{alg:IO_SAMD} with $\omega (\theta) = \sum_{i=1}^n \theta_i\log(\theta_i)$ and and approximate subgradients computed using the entire dataset;
    \item \textbf{Stochastic approximate subgradient method (SASM)}: Algorithm \ref{alg:IO_SAMD} with $\omega(\theta) = \frac{1}{2}\|\theta\|_2^2$ and approximate stochastic subgradients.
    \item \textbf{Stochastic Approximate Mirror Descent (SAMD)} with exponentiated updates, i.e., Algorithm \ref{alg:IO_SAMD} with $\omega (\theta) = \sum_{i=1}^n \theta_i\log(\theta_i)$ and approximate stochastic subgradients.
\end{enumerate}
For all algorithms, we use $\eta_t = 1/(\|\tilde{g}_{\varepsilon_t}(\theta_t)\|_*\sqrt{t})$. To compute approximate subgradients for the ASM and SAMD algorithms, we give the solver (in our case, Gurobi) a time limit of $0.03$ seconds to solve the optimization problem in line 4 of Algorithm \ref{alg:IO_SAMD}. If the solver is not able to find an optimal solution within this time limit, it returns the best feasible solution found.

\begin{figure}
\centering
\captionsetup[subfigure]{width=0.96\linewidth}%
    \begin{subfigure}[t]{0.4\linewidth}
        \includegraphics[width = \linewidth]{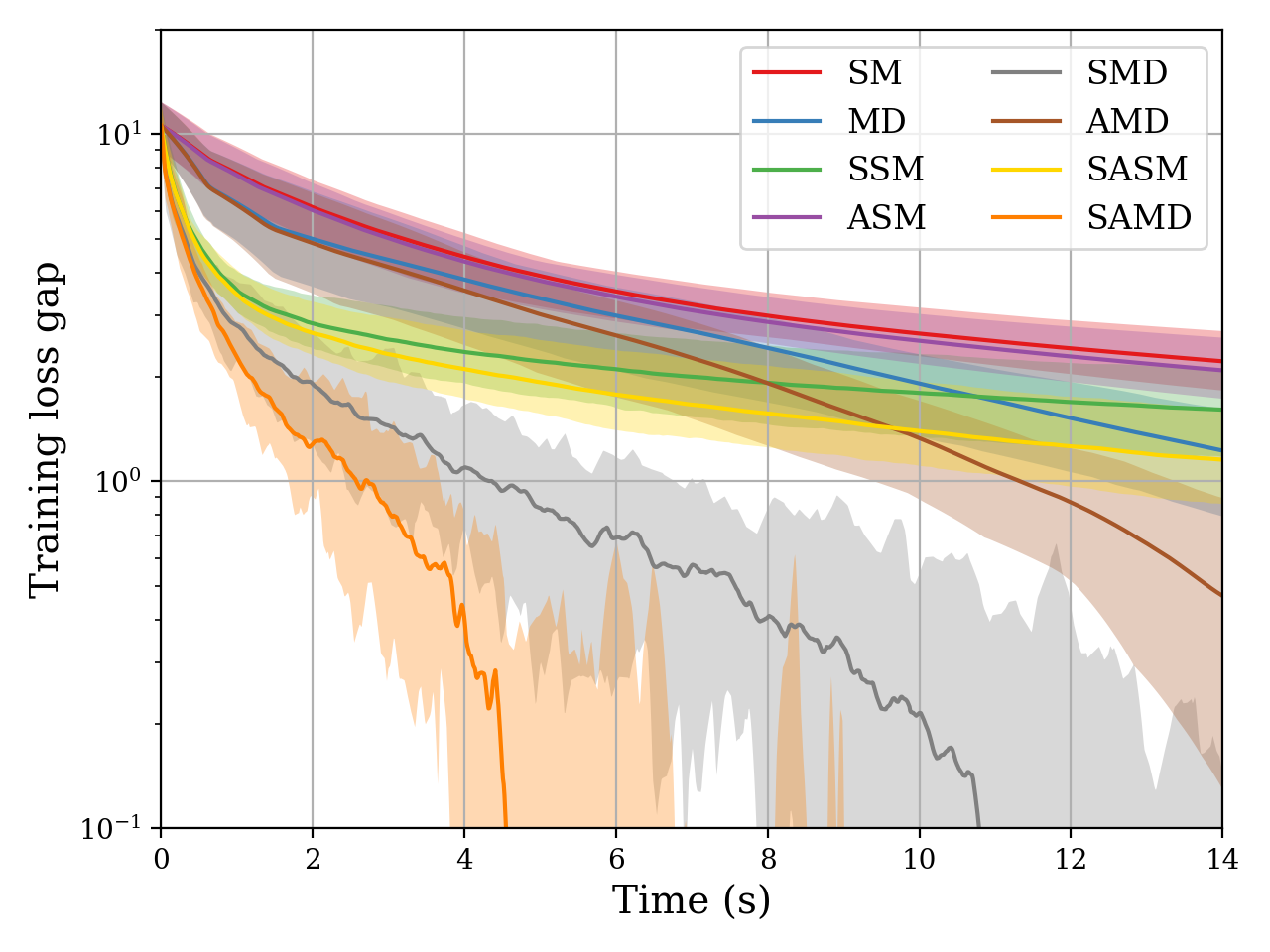}
        \caption{Convergence of the first-order optimization algorithms.}
        \label{fig:FOM_all_loss_diff}
    \end{subfigure}
    \begin{subfigure}[t]{0.4\linewidth}
        \includegraphics[width = \linewidth]{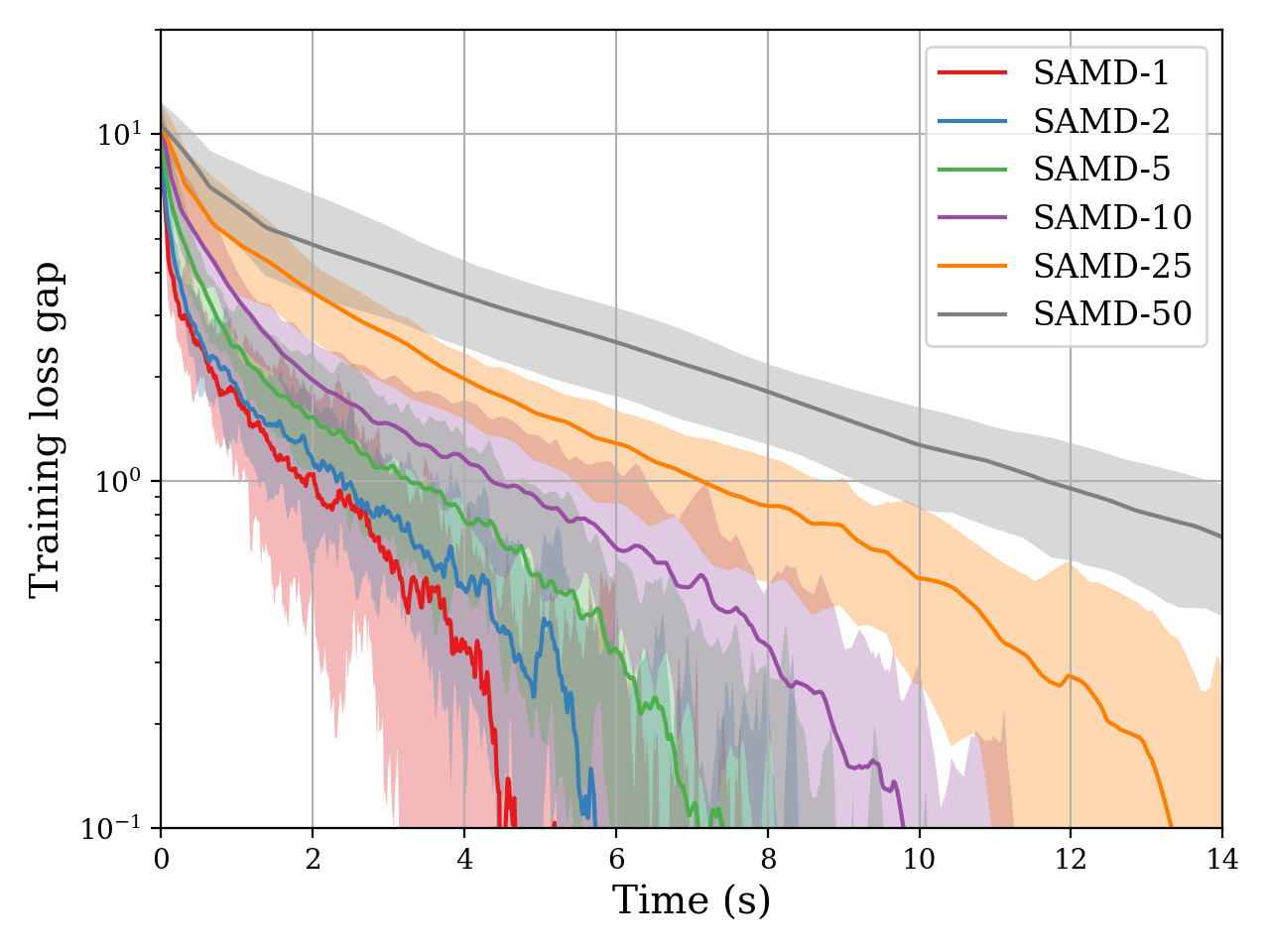}
        \caption{Convergence of the SAMD algorithms for different sizes of batch sizes.}
        \label{fig:batch_loss_diff}
    \end{subfigure}
\caption{Convergence time of different first-order optimization algorithms and for different batch sizes.}
\label{fig:loss_diff}
\end{figure}

\textbf{Results.} To compare the performance of the algorithms, we report their results in terms of running time instead of the number of iterations. Figure \ref{fig:FOM_all_loss_diff} shows the convergence of the proposed algorithms in terms of the training loss gap. More precisely, defining $f(\theta) \coloneqq \kappa \mathcal{R}(\theta) + \frac{1}{N}\sum_{i=1}^N \ell_\theta (\hat{s}_i, \hat{x}_i)$, the training loss gap of some $\theta_t$ is defined as $f(\theta_t) - \min_{\theta \in \Theta} f(\theta)$.  From this plot, it is clear that each modification of the standard subgradient method (i.e., mirror descent updates, stochastic subgradients, and approximate subgradients) contributes to improving the convergence speed of the algorithms, with the fastest convergence achieved by the combination of all of these improvements, i.e., the SAMD algorithm. In Figure \ref{fig:batch_loss_diff}, we show the convergence of the SAMD using different batch sizes to compute stochastic subgradients. That is, instead of using only one sampled example at each iteration (line 3 of Algorithm~\ref{alg:IO_SAMD}), we sample a batch of $B$ examples and use this batch of data to compute the stochastic subgradient, which we call $\text{SAMD-}B$. From Figure \ref{fig:batch_loss_diff} we can see that even though smaller batches lead to more variance in the converge of the algorithm (as expected) it also leads to a faster empirical convergence for this experiment.


\section{Further Discussions and Potential Applications}
\label{sec:conclusion}

In this work, we present new approaches to tackle inverse optimization problems. Based on the geometry of the set of consistent cost vectors, we propose the use of an incenter vector of this set as the solution to the IO problem. Moreover, we propose a new loss function for IO problems: the Augmented Suboptimality Loss. This loss can be interpreted as a relaxation of the incenter approach to tackle problems with inconsistent data. Moreover, this loss can be used to derive IO approaches tailored for problems with mixed-integer feasible sets and can be optimized directly using first-order methods. For the latter case, we propose new algorithms that combine stochastic, approximate, and mirror descent updates, all of which can be used to exploit the structure of IO losses. Finally, we numerically evaluate the approaches proposed in this paper and show that they can outperform state-of-the-art methods in a variety of scenarios.

The work presented in this paper opens several avenues for future research. For instance, one could try different ways to choose a vector from the set of consistent costs, for example, using ellipsoidal cones instead of circular cones. This idea was used in online IO algorithms (see \cite[Section 4.3]{besbes2023contextual}), but perhaps it can also be leveraged for better empirical performance in the offline case. Also, one could try to adapt and further develop the incenter approaches presented in this paper to the online IO scenario, perhaps leveraging the geometry of the IO problem to prove tighter regret bounds in some specific scenarios, akin to \citep{besbes2023contextual}. Another idea is to come up with different algorithms to tackle the optimization problems discussed in this paper, for instance, using cutting-plane/bundle methods, extending the related literature of structured prediction problems \citep{wang2009cutting, joachims2009cutting, teo2010bundle}. Moreover, given the wealth of different first-order methods in the literature, we feel like designing specialized first-order (or even proximal or second-order) algorithms tailored for inverse optimization problems is an under-explored research direction. Given the newly established bridge between IO and structured prediction problems (see Remark \ref{remark:connections_structured}), we forsee that many interesting results can be achieved by combining and extending the literature of these two, mostly disjoint, communities. Finally, it would be interesting to extend the numerical experiments of this paper using real-world data, for instance, data from transportation/routing problems.


\appendix

\section{Theoretical Properties of the Augmented Suboptimality Loss}
\label{app:theoretical}

Define $F_{\theta,\phi}(s,x) \coloneqq \inner{\theta}{\phi(s,x)}$. In this section, we discuss several theoretical properties of the Augmented Suboptimality Loss (\ref{eq:ASL}), and show how it can be interpreted as a convex surrogate of the \textit{Generalized Predictability Loss} (GPL). The GPL is a natural loss function for IO problems and is defined as follows.

\begin{definition}[Generalized Predictability Loss]
Given a signal-response pair $(\hat{s},\hat{x})$, the Generalized Predictability Loss of a cost vector $\theta$ is given by
\begin{equation}
    \label{eq:predictability}
    \begin{aligned}    \ell^{\emph{pred}}_\theta(\hat{s},\hat{x}) \coloneqq & \min_{x} \ d(\hat{x},x) \\
        & \ \emph{s.t.} \quad x \in \argmin_{y \in \mathbb{X}(\hat{s})} F_{\theta,\phi}(\hat{s},y),
    \end{aligned}
\end{equation}
where $d : \mathbb{X} \times \mathbb{X} \to \mathbb{R}_+$ is a distance function.  
\end{definition}
The GPL has an intuitive interpretation for IO problems: given a signal $s$, it computes how close (in terms of $d$) are $\hat{x}$ (the expert's response) and $x$ (an optimal response according to $F_{\theta,\phi}$). Consequently, by solving \eqref{eq:reg_loss_minimization} using the GPL, we optimize for a cost $F_{\theta,\phi}$ that best reproduces the responses taken by the expert for each signal $s$. For the special case when $d(x,y) = \|x - y\|^2_2$, the GPL reduces to the so-called \textit{predictability loss}. Unfortunately, when using the GPL, problem \eqref{eq:reg_loss_minimization} is an NP-hard bi-level optimization problem in general \citep{aswani2018inverse}. 

To come up with losses that are convex w.r.t. $\theta$, and still meaningful for IO problems, we can use the concept of \textit{convex surrogate functions}. Informally, given a nonconvex function $g$, a convex surrogate function $h$ is a convex upper bound of $g$. The idea is then to minimize $h$ instead of $g$. The hope is that by minimizing $h$, we also minimize (at least to some extent) the original nonconvex function $g$. Interestingly, we can show that the \ref{eq:ASL} is a convex surrogate for the GPL and that it possesses several properties attractive for IO loss functions. In particular, the properties of the \ref{eq:ASL} shown in Proposition \ref{prop:augmented_suboptimality_surrogate} are similar to (and in some sense generalizations of) the properties of the Suboptimality Loss presented in \citep{esfahani2018data}.

\begin{assumption}[Optimizer condition]
\label{ass:optimizer_condition}
Let the distance function $d: \mathbb{X} \times \mathbb{X} \to \mathbb{R}_+$ be such that $d(x,y) = 0$ if and only if $x=y$. Then, $\forall x \in \mathbb{X}(s)$ and $\forall s \in \mathbb{S}$, there exists a $\theta \in \Theta$ such that
$$
y^\star  \in \argmin_{y \in \mathbb{X}(s)} F_{\theta,\phi}(s,y) \implies F_{\theta,\phi}(s,x) - F_{\theta,\phi}(s,y^\star ) \geq d(x,y^\star ).
$$
\end{assumption}

Assumption \ref{ass:optimizer_condition} can be interpreted as a ``unique minimizer'' condition and it is necessary to prove the ``consistency'' property of Proposition \ref{prop:augmented_suboptimality_surrogate}. This is a strong assumption and does not hold for general IO problems. However, this assumption is not critical for the algorithm design and analysis presented in this paper. Moreover, empirically, IO algorithms based on the \ref{eq:ASL} perform well even for IO problems that do not possess the ``unique minimizer'' property (for example, see the numerical results of Section \ref{sec:numerical}). Finally, we note that Assumption \ref{ass:optimizer_condition} can be interpreted as a generalization of other assumptions from the literature, which are also used to prove consistency-like properties of loss functions. We list three such examples, where we use the notation $y^\star  \in \argmin_{y \in \mathbb{X}(s)} F_{\theta,\phi}(s,y)$.
\begin{itemize}
    \item Let $d(x,y) = I(x,y)$, the 0-1 distance. In this case, Assumption \ref{ass:optimizer_condition} reads as
    $$
    F_{\theta,\phi}(s,x) -  F_{\theta,\phi}(s,y^\star ) \geq 1 \quad \forall (s, x) \in \mathbb{S} \times \mathbb{X}(s)\setminus \{y^\star \}.
    $$
    The inequality above holds, for instance, in the case $F_{\theta,\phi}(s,x) = \inner{\theta}{x}$ with $\theta \in \mathbb{Z}^n_+$ (a vector of positive integers) and $\mathbb{X}(x) = \{0,1\}^n$ \citep[Corollary 3.6]{barmann2017emulating}.
    \item Let $d(x,y) = \|x - y\|^2_2$. In this case, Assumption \ref{ass:optimizer_condition} reads as
    $$
    F_{\theta,\phi}(s,x) - F_{\theta,\phi}(s,y^\star ) \geq \|x - y^\star \|^2_2,  \quad \forall (s, x) \in \mathbb{S} \times \mathbb{X}(s).
    $$
    Let the set $\mathbb{X}(s)$ be $\mu$-strongly convex with $\mu = 2/\|\nabla_xF_{\theta,\phi}(s,y^\star )\|_2$. Then, one can show that Assumption \ref{ass:optimizer_condition} holds \cite[Proposition 1]{el2019generalization}. Particularly, notice that the inequality holds even when $F_{\theta,\phi}(s,x) = \inner{\theta}{x}$.
    \item Let $d(x,y) = \|x - y\|^2_2$. In this case, Assumption \ref{ass:optimizer_condition} reads as
    $$
    F_{\theta,\phi}(s,x) - F_{\theta,\phi}(s,y^\star ) \geq \|x - y^\star \|^2_2,  \quad \forall (s, x) \in \mathbb{S} \times \mathbb{X}(s).
    $$
    One can show that this inequality holds when $F_{\theta,\phi}(s,x)$ is $2$-strongly convex w.r.t. $x$. Thus, one can interpret Assumption \ref{ass:optimizer_condition} as a generalization of the strong-convexity assumption used in \cite[Proposition 2.5]{esfahani2018data}.
\end{itemize}

\begin{proposition}[Properties of the \ref{eq:ASL}]
\label{prop:augmented_suboptimality_surrogate}
The \ref{eq:ASL} $\ell_\theta$ has the following properties:
\begin{itemize}
    \item (\textbf{Convex surrogate}) $\ell^{\emph{pred}}_\theta(s,x) \leq \ell_\theta(s,x) \quad \forall (s, x) \in \mathbb{S} \times \mathbb{X}(s)$;
    \item (\textbf{Convexity}) $\ell_\theta$ is convex w.r.t. $\theta$;
    \item (\textbf{Nonnegativity}) $\ell_\theta(s,x) \geq 0, \quad \forall (s, x) \in \mathbb{S} \times \mathbb{X}(s)$;
    \item (\textbf{Consistency}) $\ell_\theta(s,x)=0 \implies x \in \argmin_{y \in \mathbb{X}(s)} F_{\theta,\phi}(s,y), \quad \forall s \in \mathbb{S}$. If Assumption \ref{ass:optimizer_condition} holds, then $\ell_\theta(s,x)=0 \impliedby x \in \argmin_{y \in \mathbb{X}(s)} F_{\theta,\phi}(s,y), \quad \forall s \in \mathbb{S}$.
\end{itemize}
\end{proposition}
\begin{proof}
Let $y^\star  \in \argmin_{y \in \mathbb{X}(s)} F_{\theta,\phi}(s,y)$.

\textbf{(Convex surrogate)} By the definition of the generalized predictability loss and $y^\star $, we have
$\ell^\text{pred}_\theta(s,x) \leq d(x,y^\star )$ and
\begin{equation}
    \label{eq:optimality_y_star}
    F_{\theta,\phi}(s,x) - F_{\theta,\phi}(s,y^\star ) \geq 0,
\end{equation}
for all for all $x \in \mathbb{X}(s)$, for all $s \in \mathbb{S}$. Thus,
\begin{align*}
    \ell^\text{pred}_\theta(s,x) 
    &\leq d(x,y^\star ) \\
    &\leq d(x,y^\star ) + F_{\theta,\phi}(s,x) - F_{\theta,\phi}(s,y^\star ) \tag{Eq. \eqref{eq:optimality_y_star}}\\
    &\leq \max_{y \in \mathbb{X}(s)} \left\{d(x,y) + F_{\theta,\phi}(s,x) - F_{\theta,\phi}(s,y) \right\} = \ell_\theta(s,x).
\end{align*}

\textbf{(Convexity)} Using the fact that $F_{\theta,\phi}(s,x) = \inner{\theta}{\phi(s,x)}$, we conclude that $\ell_\theta(s,x)$ is the pointwise maximum of linear functions of $\theta$, which is convex w.r.t. $\theta$.

\textbf{(Nonnegative)} Since $\ell^\text{pred}_\theta(s,x)$ is nonnegative, and we have shown that the \ref{eq:ASL} upper bounds the GPL, this implies $\ell_\theta(s,x) \geq 0$.

\textbf{(Consistency $\implies$)}
\begin{align}
    \ell_\theta(s,x)=0 
    &\implies \max_{y \in \mathbb{X}(s)} \left\{ F_{\theta,\phi}(s,x) - F_{\theta,\phi}(s,y) + d(x,y) \right\} = 0 \nonumber\\
    &\implies F_{\theta,\phi}(s,x) = \min_{y \in \mathbb{X}(s)} \left\{F_{\theta,\phi}(s,y) - d(x,y) \right\} \nonumber\\
    &\implies F_{\theta,\phi}(s,x) \leq F_{\theta,\phi}(s,y) - d(x,y), \quad \text{for any } y \in \mathbb{X}(s) \nonumber\\
    &\implies F_{\theta,\phi}(s,x) \leq F_{\theta,\phi}(s,y^\star ) - d(x,y^\star ) \nonumber\\
    &\implies F_{\theta,\phi}(s,x) - F_{\theta,\phi}(s,y^\star ) \leq 0 \label{eq:consistency_1},
\end{align}
where the last inequality follows from the facts that $d(x,y^\star )$ is nonnegative. From \eqref{eq:optimality_y_star} and \eqref{eq:consistency_1} and the definition of $y^\star $, we conclude $x \in \argmin_{y \in \mathbb{X}(s)} F_{\theta,\phi}(s,y)$.

\textbf{(Consistency $\impliedby$)} By Assumption \ref{ass:optimizer_condition}, we have that
\begin{align*}
    x \in \argmin_{y \in \mathbb{X}(s)} F_{\theta,\phi}(s,y)
    &\implies F_{\theta,\phi}(s,y) - F_{\theta,\phi}(s,x) \geq d(y,x), \quad \text{for any } y \in \mathbb{X}(s) \\
    &\implies F_{\theta,\phi}(s,x) - F_{\theta,\phi}(s,y) + d(y,x) \leq 0, \quad \text{for any } y \in \mathbb{X}(s) \\
    &\implies \max_{y \in \mathbb{X}(s)} \left\{ F_{\theta,\phi}(s,x) - F_{\theta,\phi}(s,y) + d(y,x) \right\} \leq 0 \implies \ell_\theta (s,x) \leq 0.
\end{align*}
Since we have already shown that $\ell_\theta (s,x) \geq 0$, this implies $\ell_\theta (s,x) = 0$.
\end{proof}


\section{Continuous problems: special cases}
\label{app:continuous}

A key step in the derivation of the reformulation presented in Theorem \ref{theo:MI_reformulation} is dualizing the maximization problem in the \ref{eq:ASL} with respect to the continuous part of the decision vector. For the reformulation to be exact, strong duality needs to hold, i.e., the maximization problem needs to be concave w.r.t. the continuous variable. As mentioned in Section \ref{sec:mixed}, using a general $\infty$-norm penalization comes at the expense of increasing the number of constraints in the reformulation by a factor of $2u$. However, for some special cases, one can dualize the distance-augmented problem without increasing the number of constraints of the final problem. Here we briefly discuss two such examples: \textit{linear programs with totally unimodular constraint matrix} and \textit{quadratically constrained quadratic programs}.

\textbf{LPs with totally unimodular constraint matrix.}  Let us consider IO problems with linear hypothesis $\inner{\theta}{\phi(s,x)} = \langle \theta, x \rangle$ and linear constraints $Ax \leq b, 0 \leq x \leq \mathbbm{1}$. For this problem, the inner maximization problem of the \ref{eq:ASL} with $d(\hat{x}, x) = \|\hat{x} - x\|_1$ is of the form
\begin{equation}
\label{eq:linear_ASL_max}
    \begin{aligned}
        \max_{x \in \mathbb{R}^n} \quad & \langle \theta, \hat{x} - x \rangle + \|\hat{x} - x\|_1 \\
        \text{s.t.} \quad \ & Ax \leq b \\
        &0 \leq x \leq \mathbbm{1}.
    \end{aligned}
\end{equation}
Due to the norm in the objective function, this is a nonconcave optimization problem in general. However, consider the case when the constraint matrix $A$ in \eqref{eq:linear_ASL_max} is \textit{totally unimodular} and $b$ is an integer vector. Totally unimodular constraint matrices appear in many classical optimization problems, for instance, shortest path problems, bipartite graph matching, and maximum flow problems \citep{conforti2014integer}. In this case, it is a well-known result that every extreme point of the polyhedron $\{x : Ax \leq b, \ 0 \leq x \leq \mathbbm{1}\}$ is integral \citep{hoffman2010integral}, thus, \eqref{eq:linear_ASL_max} has integral (in this particular case, binary) optima. Combined with the fact that, if $\hat{x}$ and $x$ are binary vectors, we have that $\|\hat{x} - x\|_1 = \inner{\mathbbm{1} - 2\hat{x}}{x} + \inner{\mathbbm{1}}{\hat{x}}$, the optimization problem \eqref{eq:linear_ASL_max} is equivalent to
\begin{equation*}
    \begin{aligned}
        \max_{x \in \mathbb{R}^n} \quad & \langle \theta, \hat{x} - x \rangle + \inner{\mathbbm{1} - 2\hat{x}}{x} + \inner{\mathbbm{1}}{\hat{x}} \\
        \text{s.t.} \quad \ & Ax \leq b \\
        &0 \leq x \leq \mathbbm{1},
    \end{aligned}
\end{equation*}
that is, problem \eqref{eq:linear_ASL_max} is equivalent to an LP and, thus, can be dualized without increasing the number of constraints in the final IO reformulation.

\textbf{Quadratically constrained quadratic programs.} Next, consider IO problems with quadratic hypothesis $\inner{\theta}{\phi(s,x)} = \langle x, Qx \rangle + 2\langle x, q \rangle$ and one quadratic constraint $\langle x, Ax \rangle + 2\langle x, b \rangle + c \leq 0$. For this problem, the inner maximization problem of the \ref{eq:ASL} with $d(\hat{x}, x) = \|\hat{x} - x\|_2^2$ is of the form
\begin{equation}
\label{eq:quadratic_ASL_max}
    \begin{aligned}
        \max_{x \in \mathbb{R}^n} \quad & \langle \hat{x}, Q\hat{x} \rangle + 2\langle \hat{x}, q \rangle - \langle x, Qx \rangle - 2\langle x, q \rangle + \|\hat{x} - x\|_2^2 \\
        \text{s.t.} \quad \ & \langle x, Ax \rangle + 2\langle x, b \rangle + c \leq 0.
    \end{aligned}
\end{equation}
Problem \eqref{eq:quadratic_ASL_max} is sometimes called a \textit{Generalized Trust Region Problem}, and has numerous applications to, for example, robust optimization, signal processing, and compressed sensing \citep{more1993generalizations, pong2014generalized, wang2022generalized}. Moreover, \eqref{eq:quadratic_ASL_max} is not a concave maximization problem unless $Q - I \succcurlyeq 0$. However, it is a well-known result that strong duality holds for \eqref{eq:quadratic_ASL_max} and its dual program is
\begin{equation*}
    \begin{aligned}
        \min_{t, \lambda} \quad & \langle \hat{x}, Q\hat{x} \rangle + 2\langle \hat{x}, q \rangle + \langle \hat{x}, \hat{x} \rangle - t \\
        \text{s.t.} \quad \ & \lambda \geq 0 \\
        & \begin{bmatrix}
        Q - I + \lambda A & q + \hat{x} + \lambda b \\
        * & \lambda c - t
        \end{bmatrix} \succcurlyeq 0,
    \end{aligned}
\end{equation*}
provided Slater’s constraint qualification is satisfied, i.e., there exists an $x$ with $\langle x, Ax \rangle + \langle x, b \rangle + c \leq 0$ \cite[Appendix B]{boyd2004convex}.


\section{Proofs}
\label{app:proofs}

\subsection{Proof of Theorem \ref{theo:np_hard_circumcenter}}

Using the facts that $\|\tilde{\theta}\|_2 = \|\theta\|_2 = 1$, the range of the $\arccos$ is $[0,\pi]$, $-\cos(\gamma)$ is monotone increasing for $\gamma \in [0,\pi]$, and $-\langle \tilde{\theta}, \theta \rangle = \frac{1}{2}\| \tilde{\theta} - \theta \|_2^2 - 1$, evaluating the cost function of \eqref{eq:circumcenter} is equivalent to
\begin{equation}
\label{eq:cost_function1}
    \begin{aligned}
    \max_{\tilde{\theta}} \quad & \| \theta - \tilde{\theta} \|_2^2 \\
    \text{s.t.} \ \quad & \tilde{\theta} \in \mathbb{C}, \quad \|\tilde{\theta}\|_2 = 1.
    \end{aligned}
\end{equation}
Next, applying the change of variables $\tilde{\theta} \mapsto R\tilde{\theta}$, where $R \in \mathbb{R}^{n \times n}$ is an orthonormal matrix with its first column equal to $\theta$, \eqref{eq:cost_function1} is equivalent to
\begin{equation}
\label{eq:cost_function2}
    \begin{aligned}
    \max_{\tilde{\theta}} \quad & \| e_1 - \tilde{\theta} \|_2^2 \\
    \text{s.t.} \ \quad & R\tilde{\theta} \in \mathbb{C}, \quad \|\tilde{\theta}\|_2 = 1.
    \end{aligned}
\end{equation}
Then, defining $\mathbb{C} \coloneqq \big\{R\tilde{\theta} = (x, y) \in \mathbb{R} \times \mathbb{R}^{n-1} : a_ix + \langle b_i, y \rangle \leq 0 \ \forall i \in [N], \ -x + \langle e_j, y \rangle \leq 0  \ \text{and} \ -x - \langle e_j, y \rangle \leq 0 \ \forall j \in [n-1] \big\}$, we can rewrite \eqref{eq:cost_function2} as

\begin{equation}
\label{eq:cost_function3}
    \begin{aligned}
    \max_{x, y} \quad & (x-1)^2 + \|y\|_2^2 \\
    \text{s.t.} \ \quad & a_ix + \langle b_i, y \rangle \leq 0 & \forall i \in [N] \\
    & -x + \langle e_j, y \rangle \leq 0 & \forall j \in [n-1] \\
    & -x - \langle e_j, y \rangle \leq 0 & \forall j \in [n-1] \\
    & x^2 + \|y\|_2^2 = 1.
    \end{aligned}
\end{equation}
Next, we change the objective function of \eqref{eq:cost_function3} from $f(x,y) = (x-1)^2 + \|y\|_2^2$ to $g(x,y) = \|(1/x)y\|_2^2$. The resulting optimization problem is equivalent to \eqref{eq:cost_function3} because $g$ is a strictly positive monotonic transformation of $f$, that is, for any feasible $(x_1, y_1)$ and $(x_2, y_2)$, we have that
\begin{figure}[h]
\centering
\begin{subfigure}[c]{0.5\textwidth}
    \centering
    \begin{align*}
        f(x_1,y_1) &< f(x_2,y_2) \\
        &\iff (x_1-1)^2 + \|y_1\|_2^2 < (x_2-1)^2 + \|y_2\|_2^2 \\ 
        &\iff -2x_1 < -2x_2 \\
        &\iff \frac{1}{x_1^2} < \frac{1}{x_2^2} \\
        &\iff \frac{1 - x_1^2}{x_1^2} < \frac{1 - x_2^2}{x_2^2} \\
        &\iff \left\|\frac{1}{x_1}y_1\right\|^2_2 < \left\|\frac{1}{x_2}y_2\right\|^2_2 \\
        &\iff g(x_1,y_1) < g(x_2,y_2),
    \end{align*}
\end{subfigure}
\hfill
\begin{subfigure}[c]{0.35\textwidth}
    \centering
    \includegraphics[width=\textwidth]{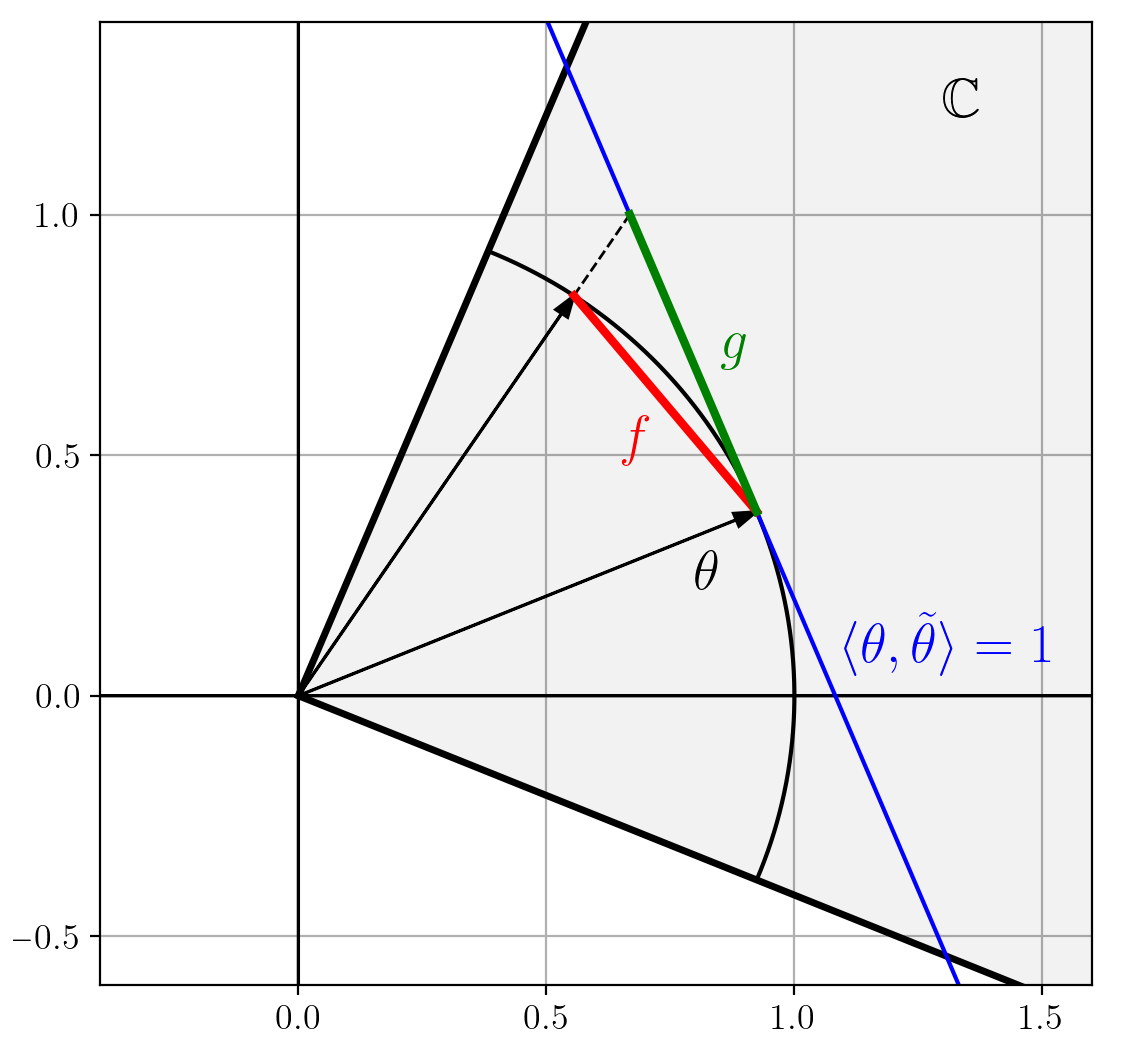}
    \captionsetup{labelformat=empty}
    \caption{Illustration of the monotonicity of the transformation from $f$ to $g$.}
\end{subfigure}
\end{figure}

\noindent
where we used the fact that from the inequalities $-x \pm \langle e_j, y \rangle \leq 0$ and the equality $x^2 + \|y\|_2^2 = 1$, any feasible $x$ is strictly larger than $0$. Thus, \eqref{eq:cost_function3} is equivalent to
\begin{equation}
\label{eq:cost_function4}
    \begin{aligned}
    \max_{x, y, z} \quad & \|z\|_2^2 \\
    \text{s.t.} \ \quad & a_ix + \langle b_i, y \rangle \leq 0 & \forall i \in [N] \\
    & -x + \langle e_j, y \rangle \leq 0 & \forall j \in [n-1] \\
    & -x - \langle e_j, y \rangle \leq 0 & \forall j \in [n-1] \\
    & x^2 + \|y\|_2^2 = 1, \quad z = \frac{1}{x}y.
    \end{aligned}
\end{equation}
Next, substituting $y = xz$ in the inequality constraints, and using the fact that $x>0$ to rearrange the constraints, we arrive at
\begin{equation}
\label{eq:cost_function6}
    \begin{aligned}
    \max_{z} \quad & \|z\|_2^2 \\
    \text{s.t.} \ \quad & \langle b_i, z \rangle \leq -a_i & \forall i \in [N] \\
    & -1 \leq z \leq 1,
    \end{aligned}
\end{equation}
where we dropped the equality constraints $x = 1/\sqrt{1 + \|z\|^2_2}$ and $y = (1/\sqrt{1 + \|z\|^2_2})z$, since they are the only ones that depend on $x$ and $y$. Therefore, we have shown that evaluating the cost function of the circumcenter problem is equivalent to a quadratic maximization problem over a polytope, which is NP-hard \citep{sahni1974computationally}.

\subsection{Proof of Theorem \ref{theo:incenter_reformulation}}

First, the non-zero constraints can be enforced via $\|\theta\|_2 = \|\tilde{\theta}\|_2 = 1$, and similar to the proof of Theorem \ref{theo:np_hard_circumcenter}, one can show that
$$
\argmax_{\|\theta\|_2 = 1} \min_{\substack{\tilde{\theta} \in \overline{\text{int}(\mathbb{C})} \\ \|\tilde{\theta}\|_2 = 1}} \ a(\theta, \tilde{\theta}) = \argmax_{\|\theta\|_2 = 1} \min_{\substack{\tilde{\theta} \in \overline{\text{int}(\mathbb{C})} \\ \|\tilde{\theta}\|_2 = 1}} \ \|\theta - \tilde{\theta}\|^2_2.
$$
Since for any $\theta \in \overline{\text{int}(\mathbb{C})}$, we can always choose $\tilde{\theta} = \theta$ and have $\|\theta - \tilde{\theta}\|^2_2 = 0$, the $\theta$ that maximizes this minimax problem will always be in $\mathbb{C}$. Consequently, to minimize the distance to any $\theta \in \mathbb{C}$, we have that the optimal $\tilde{\theta}$ will always be in the boundary of $\mathbb{C}$. Using these observations, one can show the following equivalence for the inner minimization problem:
\begin{equation}
    \label{eq:bd_to_hyperplane}
    \min_{\substack{\tilde{\theta} \in \overline{\text{int}(\mathbb{C})} \\ \|\tilde{\theta}\|_2 = 1}} \ \|\theta - \tilde{\theta}\|_2^2 = \min_{\substack{x_i \in \mathbb{X}(\hat{s}_i) \\ i \in [N]}} \ \left\{\min_{\substack{\inner{\tilde{\theta}}{\phi(\hat{s}_i,x_i) - \phi(\hat{s}_i,\hat{x}_i)} = 0 \\ \|\tilde{\theta}\|_2 = 1}} \ \|\theta - \tilde{\theta}\|_2^2 \right\},
\end{equation}
which follows from the fact that the optimal $\tilde{\theta}$ being in boundary of $\mathbb{C}$ implies $\inner{\tilde{\theta}}{\phi(\hat{s}_i,x_i) - \phi(\hat{s}_i,\hat{x}_i)} = 0$ for some $i \in [N]$ and $x_i\in \mathbb{X}(\hat{s}_i)$ (see Definition \ref{def:consistent}).

Next, we derive a closed-form solution to the minimization problems inside curly brackets in \eqref{eq:bd_to_hyperplane}. Namely, we show that
\begin{equation}
    \label{eq:norm_proj_hyper}
    \argmin_{\substack{\inner{\tilde{\theta}}{w} = 0 \\ \|\tilde{\theta}\|_2 = 1}} \|\theta - \tilde{\theta}\|_2^2 = \frac{p(\theta,w)}{\|p(\theta,w)\|_2},
\end{equation}
where
\begin{equation}
    \label{eq:proj_hyper}
    p(\theta,w) \coloneqq \argmin_{\inner{\tilde{\theta}}{w} = 0} \|\theta - \tilde{\theta}\|_2^2 = \theta - \frac{\inner{\theta}{w}}{\|w\|_2^2}w
\end{equation}
is the euclidean projection of $\theta$ onto the hyperplane defined by $w$. We prove \eqref{eq:norm_proj_hyper} by contradiction. Namely, we show that if there exists a vector $\alpha$ with lower cost value than $p(\theta,w) / \|p(\theta,w)\|_2$ for \eqref{eq:norm_proj_hyper}, this contradicts the optimality of $p(\theta,w)$ for problem \eqref{eq:proj_hyper}. Thus, let $\alpha \in \mathbb{R}^p$ be such that $\inner{\alpha}{w} = 0$ and $\|\alpha\|_2 = 1$, and assume
$$
    \left\|\theta - \frac{p(\theta,w)}{\|p(\theta,w)\|_2}\right\|_2^2 > \|\theta - \alpha\|_2^2.
$$
Consequently, we have that
\begin{align*}
    \left\|\theta - \frac{p(\theta,w)}{\|p(\theta,w)\|_2}\right\|_2^2 > \|\theta - \alpha\|_2^2
    &\iff 1 + 1 - \frac{2}{\|p(\theta,w)\|_2}\inner{\theta}{p(\theta,w)} > 1 + 1 - 2\inner{\theta}{\alpha} \\
    &\iff - 2\inner{\theta}{p(\theta,w)} > - 2\|p(\theta,w)\|_2\inner{\theta}{\alpha},
\end{align*}
which follows by expanding the norm squared and the facts that $\|\theta\|_2 = \|\alpha\|_2 = 1$. Adding $\|\theta\|_2^2 + \|p(\theta,w)\|_2^2$ to both sides and using the identity $\|\theta - p(\theta,w)\|_2^2 = \|\theta\|_2^2 + \|p(\theta,w)\|_2^2 - 2\inner{\theta}{p(\theta,w)}$, we get that
$$
\left\|\theta - \frac{p(\theta,w)}{\|p(\theta,w)\|_2}\right\|_2^2 > \|\theta - \alpha\|_2^2 \iff \|\theta - p(\theta,w)\|_2^2 > \|\theta\|_2^2 + \|p(\theta,w)\|_2^2 - 2\|p(\theta,w)\|_2\inner{\theta}{\alpha}.
$$
Next, define $\beta \coloneqq \|p(\theta,w)\|_2\alpha$, and notice that $\inner{\beta}{w} = 0$ and $\|\beta\|_2^2 = \|p(\theta,w)\|_2^2$. Thus, we arrive at
$$
\left\|\theta - \frac{p(\theta,w)}{\|p(\theta,w)\|_2}\right\|_2^2 > \|\theta - \alpha\|_2^2 \iff \|\theta - p(\theta,w)\|_2^2 > \|\theta\|_2^2 + \|\beta\|_2^2 - 2\inner{\theta}{\beta} = \|\theta - \beta\|_2^2.
$$
Therefore, since $\|\theta - p(\theta,w)\|_2^2 > \|\theta - \beta\|_2^2$ contradicts the optimality of $p(\theta,w)$ in \eqref{eq:proj_hyper}, \eqref{eq:norm_proj_hyper} follows by contradiction.

Combining \eqref{eq:bd_to_hyperplane} and \eqref{eq:norm_proj_hyper}, we have that
$$
\min_{\substack{\tilde{\theta} \in \overline{\text{int}(\mathbb{C})} \\ \|\tilde{\theta}\|_2 = 1}} \ \|\theta - \tilde{\theta}\|_2^2 = \min_{\substack{x_i \in \mathbb{X}(\hat{s}_i) \\ i \in [N]}} \ \left\|\theta - \frac{p(\theta,\phi(\hat{s}_i,x_i) - \phi(\hat{s}_i,\hat{x}_i))}{\|p(\theta,\phi(\hat{s}_i,x_i) - \phi(\hat{s}_i,\hat{x}_i))\|_2}\right\|_2^2.
$$
To conclude the proof, notice that
\begin{subequations}
\begin{align}
    \argmin_{w} \ \left\| \theta - \frac{p(\theta,w)}{\|p(\theta,w)\|_2} \right\|^2_2 
    &= \argmin_{w} \ a\left( \theta, \frac{p(\theta,w)}{\|p(\theta,w)\|_2} \right) \label{eq:a} \\
    &= \argmin_{w} \ a\left( \theta, p(\theta,w) \right) \nonumber\\
    &= \argmin_{w} \ \sin \left(a(\theta, p(\theta,w)) \right) \label{eq:b} \\
    &= \argmin_{w} \ \left\| \theta - p(\theta,w) \right\|^2_2 \label{eq:c} \\
    &= \argmin_{w} \ \frac{|\inner{\theta}{w}|}{\|w\|_2} \label{eq:d},
\end{align}
\end{subequations}
where \eqref{eq:a} follows from the definition of angle, \eqref{eq:b} follows from $a(\theta, p(\theta,w)) \leq \pi/2$, \eqref{eq:c} follows from $p(\theta,w) \perp \theta - p(\theta,w)$ and \eqref{eq:d} follows from Eq. \eqref{eq:proj_hyper}. Putting it together, we conclude that
$$
\theta^\text{in} \in \argmax_{\|\theta\|_2 = 1} \min_{\substack{\tilde{\theta} \in \overline{\text{int}(\mathbb{C})} \\ \|\tilde{\theta}\|_2 = 1}} \ a(\theta, \tilde{\theta}) = \argmax_{\|\theta\|_2 = 1} \min_{\substack{x_i \in \mathbb{X}(\hat{s}_i), i \in [N]}} \ \frac{|\inner{\theta}{\phi(\hat{s}_i,x_i) - \phi(\hat{s}_i,\hat{x}_i)}|}{\|\phi(\hat{s}_i,x_i) - \phi(\hat{s}_i,\hat{x}_i)\|_2}.
$$
Applying an epigraph reformulation and rearranging the resulting inequality constraints ends the proof.

\subsection{Proof of Corollary \ref{coro:incenter_convex_reformulation}}

First, we show that the assumption that $\text{int}(\mathbb{C}) \neq \emptyset$ implies that the optimal $r$ of problem \eqref{eq:incenter_reformulation} is positive. To see this, notice that $\text{int}(\mathbb{C}) \neq \emptyset$ implies that there exists some $\theta \in \mathbb{R}^p$ such that $\inner{\theta}{\phi(\hat{s}_i,x_i) - \phi(\hat{s}_i,\hat{x}_i)} > 0, \ \forall x_i \in \mathbb{X}(\hat{s}_i), \ \forall i \in [N]$, which follows Definition \ref{def:consistent}. Next, simply notice that 
$$
r = \frac{\min_{x_i \in \mathbb{X}(\hat{s}_i), i \in [N]} \ \inner{\theta}{\phi(\hat{s}_i,x_i) - \phi(\hat{s}_i,\hat{x}_i)}}{\max_{x_i \in \mathbb{X}(\hat{s}_i), i \in [N]} \ \| \phi(\hat{s}_i,x_i) - \phi(\hat{s}_i,\hat{x}_i) \|_2} > 0
$$
is a feasible $r$ for problem \eqref{eq:incenter_reformulation}, which shows that the optimal $r>0$. Thus, problem \eqref{eq:incenter_reformulation} can be written as
\begin{align*}
    \min_{\theta, r} \quad & \|\theta\|_2/r \\
    \text{s.t.} \ \quad & \inner{\theta}{\phi(\hat{s}_i,\hat{x}_i) - \phi(\hat{s}_i,x_i)} + r\| \phi(\hat{s}_i,x_i) - \phi(\hat{s}_i,\hat{x}_i) \|_2 \leq 0 \quad\quad \forall x_i \in \mathbb{X}(\hat{s}_i), \ \forall i \in [N] \\
    & \|\theta\|_2 = 1,
\end{align*}
where we used the facts that $\|\theta\|_2 = 1$ and the optimal $r > 0$. Next, applying the change of variables $\theta/r \to \bar{\theta}$, we get

\begin{align*}
    \min_{\bar{\theta}, r} \quad & \|\bar{\theta}\|_2 \\
    \text{s.t.} \ \quad & \inner{\bar{\theta}}{\phi(\hat{s}_i,\hat{x}_i) - \phi(\hat{s}_i,x_i)} + \| \phi(\hat{s}_i,x_i) - \phi(\hat{s}_i,\hat{x}_i) \|_2 \leq 0 \quad\quad \forall x_i \in \mathbb{X}(\hat{s}_i), \ \forall i \in [N] \\
    & \|\bar{\theta}\|_2 = 1/r.
\end{align*}
Finally, notice that since the optimization variable $r$ only appears in the constraint $\|\bar{\theta} \|_2 = 1/r$, and for any $\bar{\theta}$ we can set $r = 1/\|\bar{\theta} \|_2$, the optimization variable $r$ does not affect the optimal value of $\bar{\theta}$, and we can simply drop the constraint $\|\bar{\theta}\|_2 = 1/r$.

\subsection{Proof of Theorem \ref{theo:MI_reformulation}}

For the feasible set \eqref{eq:MI_set} and hypothesis function \eqref{eq:MI_hypothesis}, the \ref{eq:ASL} with $d(\hat{x}, x) = \|\hat{y} - y\|_\infty + d_z(\hat{z}, z)$ can be written as
\begin{align*}
\max_{\substack{z \in \mathbb{Z}(\hat{w}) \\ h \in \mathbb{H}}} \max_{y} \quad & \inner{\theta}{\phi(\hat{s},\hat{x})} - \langle y, Q_{yy}y \rangle - \langle y, Q \phi_1(\hat{w}, z) \rangle - \langle q, \phi_2(\hat{w}, z) \rangle + \langle h, \hat{y} - y \rangle + d_z(\hat{z}, z)   \\
\text{s.t.} \ \quad & \hat{A}y + \hat{B}z \leq \hat{c},
\end{align*}
where we use the identity $\|\hat{y} - y\|_\infty = \max_{h \in \mathbb{H}} \langle h, \hat{y} - y \rangle$ with $\mathbb{H} \coloneqq \{h \in \{-1, 0, 1\}^u : \|h\|_1 = 1\}$. Since we assume $Q_{yy} \succcurlyeq 0$, the inner maximization problem is convex, and by strong duality, we can write its dual as
\begin{align*}
\max_{\substack{z \in \mathbb{Z}(\hat{w}) \\ h \in \mathbb{H}}} \min_{\alpha, \lambda} \quad & \inner{\theta}{\phi(\hat{s},\hat{x})} + \alpha + \langle \lambda, \hat{c}-\hat{B}z \rangle - \langle q, \phi_2(\hat{w},z) \rangle + \langle h, \hat{y} \rangle + d_z(\hat{z}, z)  \\
\text{s.t.} \ \quad & \| Q \phi_1(\hat{w}, z) + h + \hat{A}^\top \lambda \|^2_{Q_{yy}^\dagger} \leq 4\alpha \\
& Q \phi_1(\hat{w}, z) + h + \hat{A}^\top \lambda \in \mathcal{R}(Q_{yy}) \\
& \lambda \geq 0.
\end{align*}
Next, applying a Schur complement transformation (e.g., \cite[Section A.5.5]{boyd2004convex}) to the constraints of the optimization problem above, followed by an epigraph transformation, we have
\begin{equation}
    \label{eq:maxmin1}
    \begin{aligned}
    \max_{\substack{z \in \mathbb{Z}(\hat{w}) \\ h \in \mathbb{H}}} \min_{\alpha, \beta, \lambda} \quad & \beta  \\
    \text{s.t.} \ \quad & \inner{\theta}{\phi(\hat{s},\hat{x})} + \alpha + \langle \lambda, \hat{c} - \hat{B}z \rangle - \langle q, \phi_2(\hat{w},z) \rangle + \langle h, \hat{y} \rangle + d_z(\hat{z}, z) \leq \beta \\
    & \begin{bmatrix}
    Q_{yy} & Q\phi_1(\hat{w},z) + h + \hat{A}^\top \lambda \\
    * & 4\alpha
    \end{bmatrix} \succcurlyeq 0 \\
    & \lambda \geq 0.
    \end{aligned}
\end{equation}
To reformulate \eqref{eq:maxmin1} into a single minimization problem, we use the fact that
\begin{equation}
\label{eq:opt_equivalence}
    \left\{ \begin{aligned}
    \max_{x \in \{x_1,\ldots,x_M\}} \min_{t, \delta} \quad & t \\
    \text{s.t.} \quad & g(\delta,x) \leq t
    \end{aligned} \right\}
    =
    \left\{ \begin{aligned}
    \min_{t, \delta_1,\ldots,\delta_M} \quad & t  \\
    \text{s.t.} \quad \quad & g(\delta_j,x_j) \leq t, \quad \forall j \in [M]
    \end{aligned} \right\},
\end{equation}
for a general function $g$, which follows from the observation that
\begin{align*}
    \max_{x \in \{x_1,\ldots,x_M\}} \min_{\delta} \ g(\delta,x) &= \max \{ \min_{\delta_1} g(\delta_1,x_1), \ldots, \min_{\delta_M} g(\delta_M,x_M) \} \\
    &= \min_{\delta_1,\ldots,\delta_M} \max \{ g(\delta_1,x_1), \ldots, g(\delta_M,x_M) \}.
\end{align*}
Combining \eqref{eq:maxmin1} with \eqref{eq:opt_equivalence}, we arrive at
\begin{equation}
\label{eq:ASL_MI_reformulation}
    \begin{aligned}
    \min \quad & \beta  \\
    \text{s.t.} \quad & \theta = (\text{vec}(Q_{yy}), \text{vec}(Q), q) \in \Theta, \quad \lambda_{jk}\geq 0, \quad \alpha_{jk}, \beta \in \mathbb{R} \\
    & \inner{\theta}{\phi(\hat{s},\hat{x})} + \alpha_{jk} + \langle \lambda_{jk}, \hat{c}-\hat{B}z_j \rangle - \langle q, \phi_2(\hat{w},z_j) \rangle + \langle h_k, \hat{y} \rangle + d_z(\hat{z}, z_j) \leq \beta \\
    & \begin{bmatrix}
    Q_{yy} & Q \phi_1(\hat{w},z_j) + h_k + \hat{A}^\top \lambda_{jk} \\
    * & 4\alpha_{jk}
    \end{bmatrix} \succcurlyeq 0,
    \end{aligned}
\end{equation}
where the constraints are $\forall (j, k) \in [M] \times [2u]$, $\mathbb{Z}(\hat{w}) \coloneqq \{z_{1}, \ldots, z_{M_i}\}$, $M \coloneqq |\mathbb{Z}(\hat{w})|$, and if $k \leq u$ ($k > u$),  $h_k$ is the vector of zeros except for the $k$'th $k$'th ($(k-u)$'th) element, which is equal to 1 (-1). Finally, plugging \eqref{eq:ASL_MI_reformulation} into problem \eqref{eq:reg_loss_minimization}, we arrive at
\begin{equation*}
    \begin{aligned}
    \min \ & \kappa\mathcal{R}(\theta) + \frac{1}{N}\sum_{i=1}^N \beta_i  \\
    \text{s.t.} \ & \theta = (\text{vec}(Q_{yy}), \text{vec}(Q), q) \in \Theta, \quad \lambda_{ijk}\geq 0, \quad \alpha_{ijk}, \beta_i \in \mathbb{R} \\
    & \inner{\theta}{\phi(\hat{s}_i,\hat{x}_i)} + \alpha_{ijk} + \langle \lambda_{ijk}, \hat{c}_i-\hat{B}_iz_{ij} \rangle - \langle q, \phi_2(\hat{w}_i,z_{ij}) \rangle + \langle h_k, \hat{y}_i \rangle + d_z(\hat{z}_i, z_{ij}) \leq \beta_i \\
    & \begin{bmatrix}
    Q_{yy} & Q \phi_1(\hat{w}_i, z_{ij}) + h_k + \hat{A}_i^\top \lambda_{ijk} \\
    * & 4\alpha_{ijk}
    \end{bmatrix} \succcurlyeq 0,
    \end{aligned}
\end{equation*}
where the constraints are $\forall (i, j, k) \in [N] \times [M_i] \times [2u]$.

\subsection{Proof of Proposition \ref{prop:samd_rate}}

Let $\theta^\star \in \argmin_{\theta \in \Theta} f(\theta)$. By the definition of an $\varepsilon$-subgradient, we have
\begin{equation}
    \label{eq:b1}
    \begin{aligned}
        \mathbb{E}\left[f(\theta_t) - f(\theta^\star)\right] 
        &\leq \mathbb{E}\left[\inner{g_{\varepsilon_t}(\theta_t)}{\theta_t - \theta^\star} + \varepsilon_t\right] \\
        &= \mathbb{E}\left[\inner{\mathbb{E}\left[\tilde{g}_{\varepsilon_t}(\theta_t) \ | \ \theta_t\right]}{\theta_t - \theta^\star} + \varepsilon_t\right] \\
        &= \mathbb{E}\left[\inner{\tilde{g}_{\varepsilon_t}(\theta_t)}{\theta_t - \theta^\star} + \varepsilon_t\right],
    \end{aligned}
\end{equation}
where the first equality follows from the definition of a stochastic approximate subgradient and the second equality follows from the law of total expectation. Next, we use the standard inequality for the analysis of mirror-descent algorithms \cite[Proposition 5.1]{juditsky2011first_i}
\begin{equation}
    \label{eq:main_inequality}
    \inner{\tilde{g}_{\varepsilon_t}(\theta_t)}{\theta_t - \theta^\star} \leq \frac{1}{\eta_t} ( \mathcal{B}_\omega(\theta^\star,\theta_t) - \mathcal{B}_\omega(\theta^\star,\theta_{t+1}) ) + \frac{\eta_t}{2}\norm{\tilde{g}_{\varepsilon_t}(\theta_t)}_*^2.
\end{equation}
Summing \eqref{eq:main_inequality} from $t=1$ to $T$ and taking its expectation, we arrive at
\begin{equation}
    \label{eq:b2}
    \begin{aligned}
        \mathbb{E}\left[\sum_{t=1}^T \inner{\tilde{g}_{\varepsilon_t}(\theta_t)}{\theta_t - \theta^\star}\right] 
        &\leq \mathbb{E}\left[\sum_{t=1}^T \frac{1}{\eta_t} ( \mathcal{B}_\omega(\theta^\star,\theta_t) - \mathcal{B}_\omega(\theta^\star,\theta_{t+1}) ) + \frac{1}{2}\sum_{t=1}^T \eta_t\norm{\tilde{g}_{\varepsilon_t}}_*^2\right] \\
        &\leq \frac{R^2}{\eta_1} + R^2\sum_{t=2}^T \left( \frac{1}{\eta_t} - \frac{1}{\eta_{t-1}} \right) + \frac{G^2}{2}\sum_{t=1}^T \eta_t \\
        &= \frac{R^2}{\eta_T} + \frac{G^2}{2}\sum_{t=1}^T \eta_t \leq \left( \frac{R^2}{c} + cG^2\right)\sqrt{T},
    \end{aligned}
\end{equation}
where the second line follows from rearranging the sum and the boundedness assumptions, the third line follows from telescoping the sum, and the fourth line follows from the definition of $\eta_t$ and the inequality $\sum_{t=1}^T 1/\sqrt{t} \leq 2\sqrt{T}$. Finally, using Jensen's inequality, we arrive at
\begin{align*}
    \mathbb{E}\left[f\left(\frac{1}{T}\sum_{t=1}^T \theta_t\right) - f(\theta^\star) \right] 
    &\leq \frac{1}{T} \mathbb{E}\left[\sum_{t=1}^T (f(\theta_t) - f(\theta^\star)) \right] \\
    &\leq \frac{1}{T}
    \mathbb{E}\left[\sum_{t=1}^T \inner{\tilde{g}_{\varepsilon_t}(\theta_t)}{\theta_t - \theta^\star} +  \sum_{t=1}^T\varepsilon_t \right] \\
    &\leq \left( \frac{R^2}{c} + cG^2\right)\frac{1}{\sqrt{T}} + \frac{1}{T}\sum_{t=1}^T\varepsilon_t,
\end{align*}
where the second inequality follows from \eqref{eq:b1} and the third inequality follows from \eqref{eq:b2}. Next, we prove an improved convergence bound when $f = h + \mathcal{R}$, and $\mathcal{R}$ is relative $\alpha$-strongly convex. The proof is based on \citep{lacoste2012simpler}. In this case, similar to \eqref{eq:b1}, we have
\begin{equation*}
    \mathbb{E}\left[f(\theta_t) - f(\theta^\star)\right] 
    \leq \mathbb{E}\left[\inner{\tilde{g}_{\varepsilon_t}(\theta_t)}{\theta_t - \theta^\star} + \varepsilon_t- \alpha \mathcal{B}_w(\theta^\star,\theta_t)\right].
\end{equation*}
Invoking \eqref{eq:main_inequality}, multiplying both sides by $t$, and summing the resulting inequality from $t=1$ to $T$, we get
\begin{equation*}
    \begin{aligned}
        \mathbb{E}\left[\sum_{t=1}^T t(f(\theta_t) - f(\theta^\star))\right] 
        &\leq \mathbb{E}\left[\sum_{t=1}^T \left( \frac{t}{\eta_t} ( \mathcal{B}_\omega(\theta^\star,\theta_t) - \mathcal{B}_\omega(\theta^\star,\theta_{t+1}) ) + \frac{t\eta_t}{2} \norm{\tilde{g}_{\varepsilon_t}}_*^2 + t\varepsilon_t- \alpha t\mathcal{B}_w(\theta^\star,\theta_t) \right)\right] \\
        &\leq \left( \frac{1}{\eta_1} - \alpha \right)\mathcal{B}_\omega(\theta^\star,\theta_1) + \sum_{t=2}^T \left( \frac{t}{\eta_t} -\alpha t - \frac{t}{\eta_{t-1}} \right)\mathcal{B}_\omega(\theta^\star,\theta_t) + \frac{G^2}{2}\sum_{t=1}^T t\eta_t + \sum_{t=1}^T t\varepsilon_t.
    \end{aligned}
\end{equation*}
Next, by setting $\eta_t = 2/\alpha(t+1)$, it is easy to show that the sums between parenthesis become nonpositive, and we arrive at
\begin{equation}
    \label{eq:b3}
    \begin{aligned}
        \mathbb{E}\left[\sum_{t=1}^T t(f(\theta_t) - f(\theta^\star))\right] 
        &\leq \frac{G^2}{\alpha}\sum_{t=1}^T \frac{t}{(t+1)} + \sum_{t=1}^T t\varepsilon_t \leq \frac{G^2T}{\alpha} + \sum_{t=1}^T t\varepsilon_t.
    \end{aligned}
\end{equation}
Finally, using Jensen's inequality, we have that
\begin{equation*}
    \begin{aligned}
        \mathbb{E}\left[f\left(\frac{2}{T(T+1)}\sum_{t=1}^T t\theta_t \right) - f(\theta^\star)\right] 
        &\leq
        \frac{2}{T(T+1)}\mathbb{E}\left[\sum_{t=1}^T t(f(\theta_t) - f(\theta^\star))\right]  \\
        &\leq \frac{2G^2}{\alpha(T+1)} + \frac{2}{T(T+1)}\sum_{t=1}^T t\varepsilon_t,
    \end{aligned}
\end{equation*}
where the second inequality follows from \eqref{eq:b3}.

\subsection{Proof of Lemma \ref{lemma:e-subgradient}}

This proof is based on \cite[Example 3.3.1]{bertsekas2015convex}. Recall the definition of $\ell_\theta$:
$$
\ell_\theta(\hat{s},\hat{x}) = \max_{x \in \mathbb{X}(\hat{s})} \left\{ \inner{\theta}{\phi(\hat{s},\hat{x}) - \phi(\hat{s},x)} + d(\hat{x},x) \right\}.
$$
Then, by the definition of $x_\varepsilon$ we have
\begin{align*}
    \ell_\theta(\hat{s},\hat{x}) - \ell_\nu(\hat{s},\hat{x})
    &= \max_{x \in \mathbb{X}(\hat{s})} \left\{ \inner{\theta}{\phi(\hat{s},\hat{x}) - \phi(\hat{s},x)} + d(\hat{x},x) \right\} - \max_{x \in \mathbb{X}(\hat{s})} \left\{ \inner{\nu}{\phi(\hat{s},\hat{x}) - \phi(\hat{s},x)} + d(\hat{x},x) \right\} \\
    &\leq \inner{\theta}{\phi(\hat{s},\hat{x}) - \phi(\hat{s},x_\varepsilon)} + d(\hat{x},x_\varepsilon)  + \varepsilon - \max_{x \in \mathbb{X}(\hat{s})} \left\{ \inner{\nu}{\phi(\hat{s},\hat{x}) - \phi(\hat{s},x)} + d(\hat{x},x) \right\}.
\end{align*}
Next, by the optimality of the $\max$, we arrive at
\begin{align*}
    \ell_\theta(\hat{s},\hat{x}) - \ell_\nu(\hat{s},\hat{x})
    &\leq \inner{\theta}{\phi(\hat{s},\hat{x}) - \phi(\hat{s},x_\varepsilon)} + d(\hat{x},x_\varepsilon)  + \varepsilon - \inner{\nu}{\phi(\hat{s},\hat{x}) - \phi(\hat{s},x_\varepsilon)} + d(\hat{x},x_\varepsilon) \\
    &= \inner{\phi(\hat{s},\hat{x}) - \phi(\hat{s},x_\varepsilon)}{\theta - \nu} + \varepsilon.
\end{align*}
This shows that $\phi(\hat{s},\hat{x}) - \phi(\hat{s},x_\varepsilon)$ is an $\varepsilon$-subgradient of $\ell_\theta$ w.r.t. $\theta$.


\section{Further Numerical Results}
\label{app:further_numerical}

\subsection{Breast cancer Wisconsin prognostic dataset}
\label{sec:cancer}

In this numerical experiment, we present a way to model a real-world decision problem using our IO approaches. Namely, we use the Breast Cancer Wisconsin Prognostic (BCWP) dataset from the UCI Machine Learning Repository \citep{Dua:2019}. This dataset consists of real-world data from different breast cancer patients. For each patient, there are 30 features extracted from images of cells taken from breast lumps, plus the tumor size and the number of involved lymph nodes, for a total of 32 numerical features. Moreover, each case is classified as ``recurrent'' if the disease has recurred after surgery, and ``non-recurrent'' if the disease did not recur by the time of the patient's last check-up. For recurrent patients, there is also information on the time to recur (TTR), that is, how many months it took for the disease to recur. For non-recurrent patients, we have information on the disease-free survival time (DFST), that is, how many months it is known that the patient was disease-free after the surgery. The dataset consists of 198 distinct cases, 47 of which have recurred. Thus, given the 32 features, the goal is to predict if the disease will recur or not, as well as the TTR (if recurrent) or the DFST (if non-recurrent).

We interpret this problem as an IO problem: given a signal vector $\hat{w} \in \mathbb{R}^{32}$ (numerical vector with the 32 features), an expert agent (e.g., a doctor) returns a decision $(y, z) \in \mathbb{R} \times \{0,1\}$, where $y$ is the TTR/DFST and $z=1$ if the expert predicts the disease to return (recurrent patient), and $z=0$ otherwise (non-recurrent patient). In other words, given the data on some patients, the expert predicts if the disease will recur or not. If recurrent, the expert also predicts the TTR; if not recurrent, the expert predicts the DFST, which can be interpreted as a lower bound on how long the patient is expected to be disease-free. Thus, given the dataset of signal-decision data, we use IO to learn a cost function that when minimized, mimics the behavior of the expert. Since the decision of the expert comprises of continuous and discrete components (i.e., $\mathbb{X}(\hat{s}) = \left\{ (y,z) \in \mathbb{R} \times \{0,1\} : y \geq 0 \right\}$), we tackle this scenario using the mixed-integer approach proposed in Section \ref{sec:mixed}. Here we note that since the BCWP dataset comes from a real-world scenario, the choice of hypothesis function we use for the IO approach (in particular, the feature functions $\phi_1$ and $\phi_2$) is a modeling choice. We use the hypothesis function \eqref{eq:MI_hypothesis}, with $\phi_1(w,z) = \phi_2(w,z) = (w, z, zw, 1)$. Thus, we can use the reformulation \eqref{eq:MI_reformulation} to solve the problem, with $\hat{A} = -I$, $\hat{B} = \hat{c} = 0$, $\Theta = \mathbb{R}^n$, $d_z(\hat{z}_i, z_i) = |\hat{z}_i - z_i|$, and $\mathcal{R}(\theta) = \frac{1}{2}\|\theta\|^2_2$. We call this approach \textbf{ASL-yz}. Similar to the experiments in Section \ref{sec:mixed}, we also test an approach similar to \textbf{ASL-yz}, but with $h_k = 0 \ \forall k \in [2u]$ (see discussion in the paragraph after Corollary \ref{coro:MI_reformulation_linear}). We call this approach \textbf{ASL-z}. Finally, as a benchmark, we test a regression+classification approach, where we treat the problem as a separate regression (prediction of the TTR/DFST) and classification (recurrent/non-recurrent patient). For the regression task, we tested scikit-learn's epsilon-support vector regression, ordinary least squares linear regression, kernel ridge regression, multi-layer perceptron regressor, regression based on k-nearest neighbors, gaussian process regression, and a decision tree regressor. For the classification task, we tested scikit-learn's support vector classifier, logistic regression classifier, classifier implementing the k-nearest neighbors vote, gaussian process classification based on Laplace approximation, decision tree classifier, random forest classifier, multi-layer perceptron classifier, and an AdaBoost classifier. For this experiment, using scikit-learn's default tuning parameters, the methods with the best out-of-sample performance for the regression and classification tasks were the kernel ridge regression and support vector classifier, respectively.

\begin{table}
\centering
\begin{tabular}{c | c c c} 
\textbf{Approach} & ASL-z & ASL-yz & regression+classification \\
\hline
\textbf{$| y - y_\text{true} |$} & 51.17 & 27.33 & 27.44 \\
\hline
\textbf{$| z - z_\text{true} |$} & 20\% & 21\% & 21.25\%
\end{tabular}
\caption{Out-of-sample average error in months of the TTR/DFST prediction and out-of-sample average percentage error in the prognostic of recurrent vs non-recurrent patients.}
\label{table:out_breast_cancer_results}
\end{table}

To evaluate our approach, we generate 20 training and test datasets, where each training dataset is generated by randomly sampling 90\% of the original dataset, while the remaining 10\% is used as the test dataset, similar to a $20$-fold cross-validation procedure. Table \ref{table:out_breast_cancer_results} shows the out-of-sample results for approaches to this problem. In particular, it shows the average error for the continuous part of the decision variable, i.e., the average error in months of the TTR/DFST prediction, and the percentage error of the discrete part of the decision variable, i.e., the average error in the prognostic of recurrent vs non-recurrent patients. From these results, we can see that the ASL-yz approach has slightly better performance than the regression+classification approach and has a much smaller regression error than the ASL-y approach. Compared to other works that used the BCWP dataset to predict the cancer recurrence or recurrence time, our IO approach differs from them in two major ways: first, our approach requires no pre-processing of the data. For instance, \cite{zafiropoulos2006support} splits the data into different classes according to the TTR/DFST and learns one predictor per class, and \cite{street1995inductive} pre-processes the dataset using a feature selection procedure. In contrast, our OI approach requires no pre-processing of the dataset, although we believe its results may be improved after an appropriate pre-processing of the dataset. Second, our IO approach can predict the TTR/DFST and recurrence/non-recurrence \textit{simultaneously}. In machine learning terms, our IO approach does classification and regression at the same time. This differs from the approaches in the literature, which focus on either predicting the TTR/DFST or predicting recurrence/non-recurrence. For instance, \cite{street1995inductive} was able to achieve an average decision error of 13.9 months using leave-one-out testing and a feature selection procedure, and \cite{lee2001ssvm} was able to achieve a 16.53\% prediction error for positive vs negative cases, where a case is considered positive if a recurrence occurred before 24 months, and negative otherwise.

\subsection{In-sample results}
\label{app:in_sample_results}

In this section, we show the in-sample results for the numerical experiments of Section \ref{sec:numerical}. Namely, Figure \ref{fig:consistent_in} shows the in-sample counterpart of Figure \ref{fig:consistent_out}, Figure \ref{fig:inconsistent_in} shows the in-sample counterpart of Figure \ref{fig:inconsistent_out}, Figure \ref{fig:MILP_in} shows the in-sample counterpart of Figure \ref{fig:MILP_out}, Figure \ref{fig:FOM_all_in} shows the in-sample counterpart of Figure \ref{fig:FOM_all_out}, and Table \ref{table:in_breast_cancer_results} shows the in-sample counterpart of Table \ref{table:out_breast_cancer_results}. Notice that since we always learn the IO cost vector using the training data, the difference between the true cost vector and the one learned using IO is independent of in- or out-of-sample results, thus, this plot is not shown in this section. Moreover, for the cases when the training dataset is noisy, we do not show the plot that compares the difference between the cost of the expert decisions with the cost of the decisions using $\theta_{\text{IO}}$.

\begin{figure}
\centering
\captionsetup[subfigure]{width=0.96\linewidth}%
    \begin{subfigure}[t]{0.38\linewidth}
        \includegraphics[width = \linewidth]{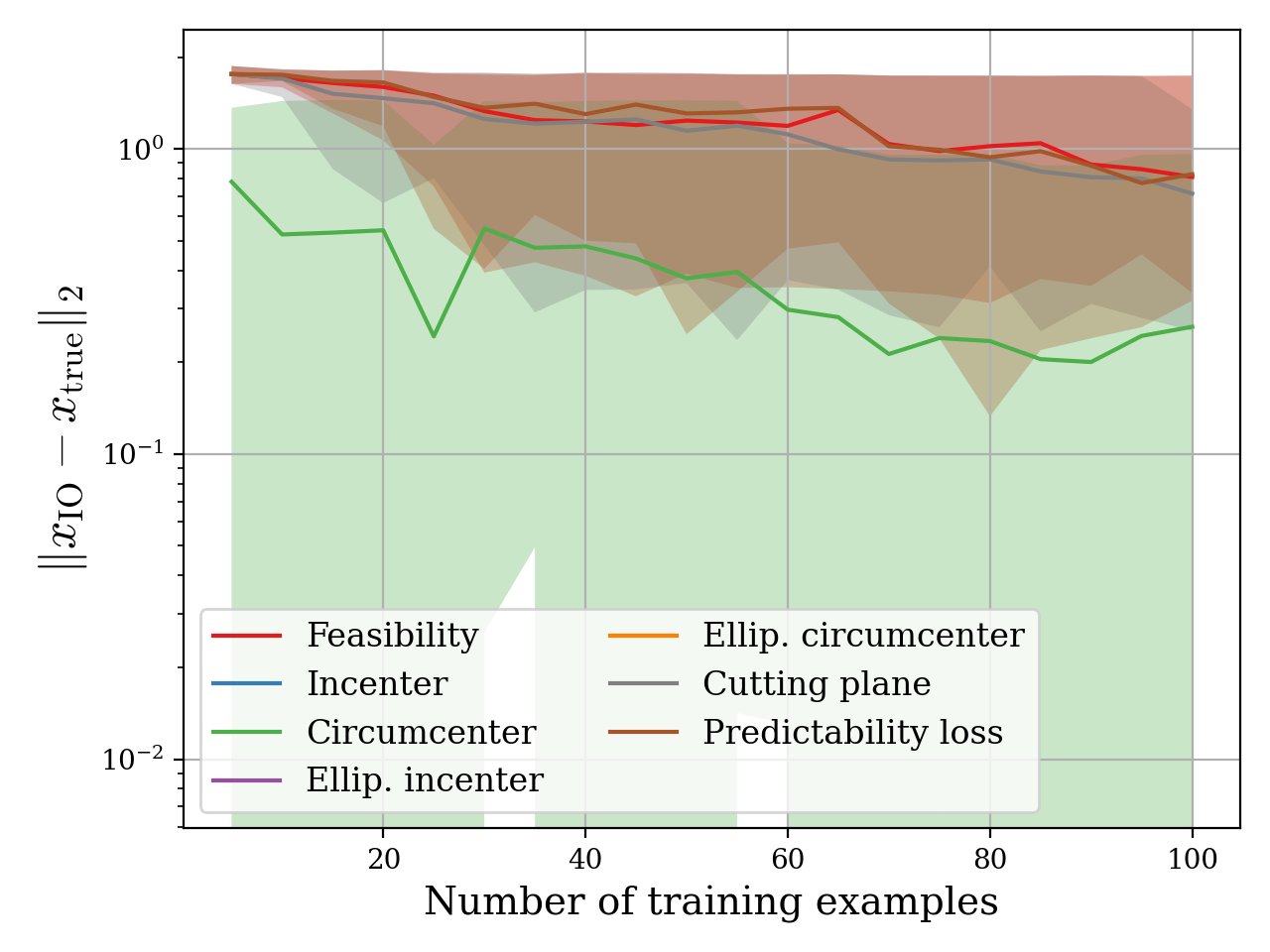}
        \caption{Average error between the decision generated by $\theta_{\text{true}}$ and $\theta_{\text{IO}}$.}
        \label{fig:consistent_x_diff_in}
    \end{subfigure}
    \begin{subfigure}[t]{0.38\linewidth}
        \includegraphics[width = \linewidth]{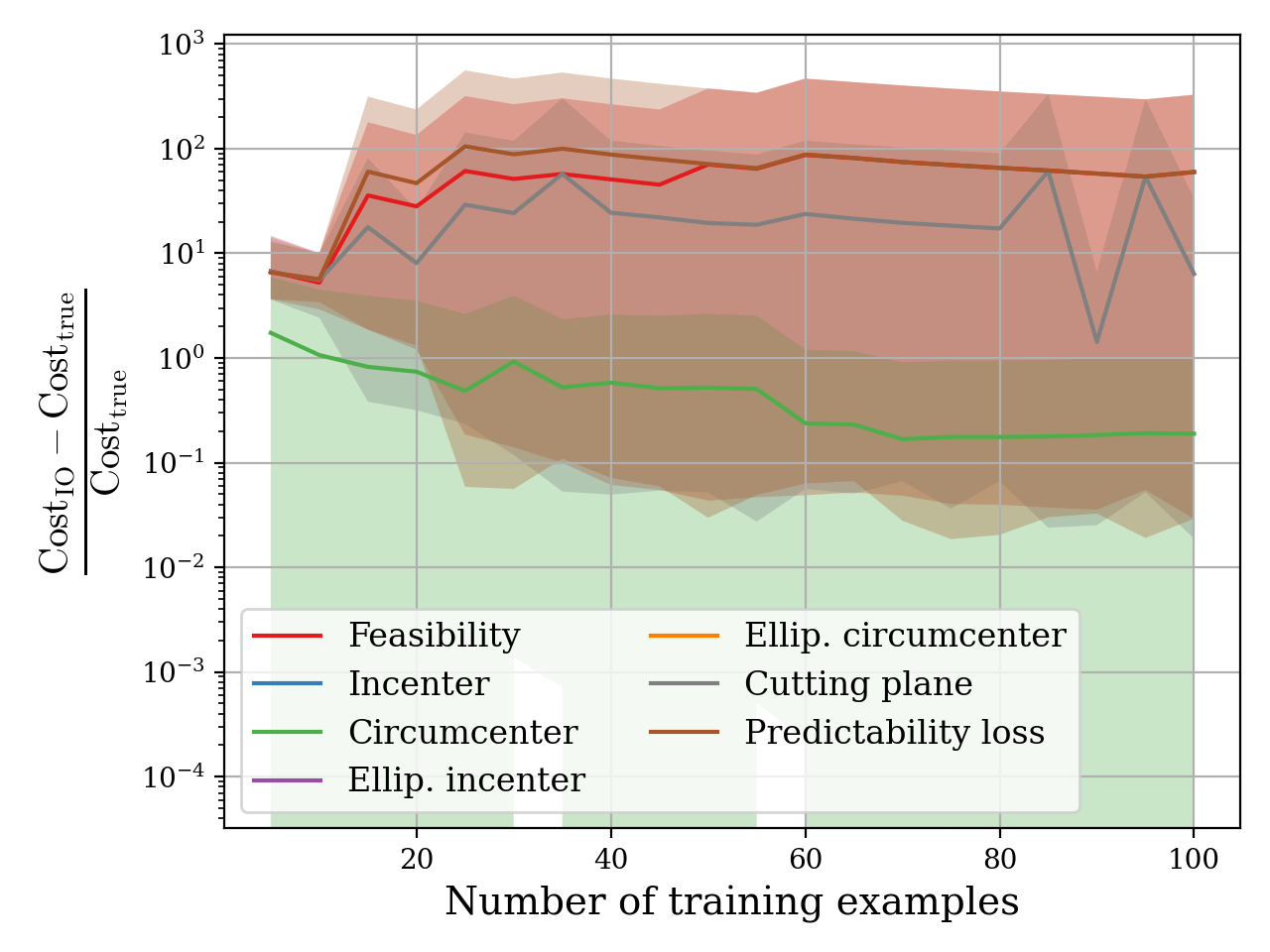}
        \caption{Relative difference between the cost of the decisions generated using $\theta_{\text{true}}$ and $\theta_{\text{IO}}$.}
        \label{fig:consistent_obj_diff_in}
    \end{subfigure}
\caption{In-sample results for consistent data scenario. The results for the incenter approach do not appear in the plots because they are zero for all number of training examples tested.}
\label{fig:consistent_in}
\end{figure}

\begin{figure}
\centering
\includegraphics[width = 0.38\linewidth]{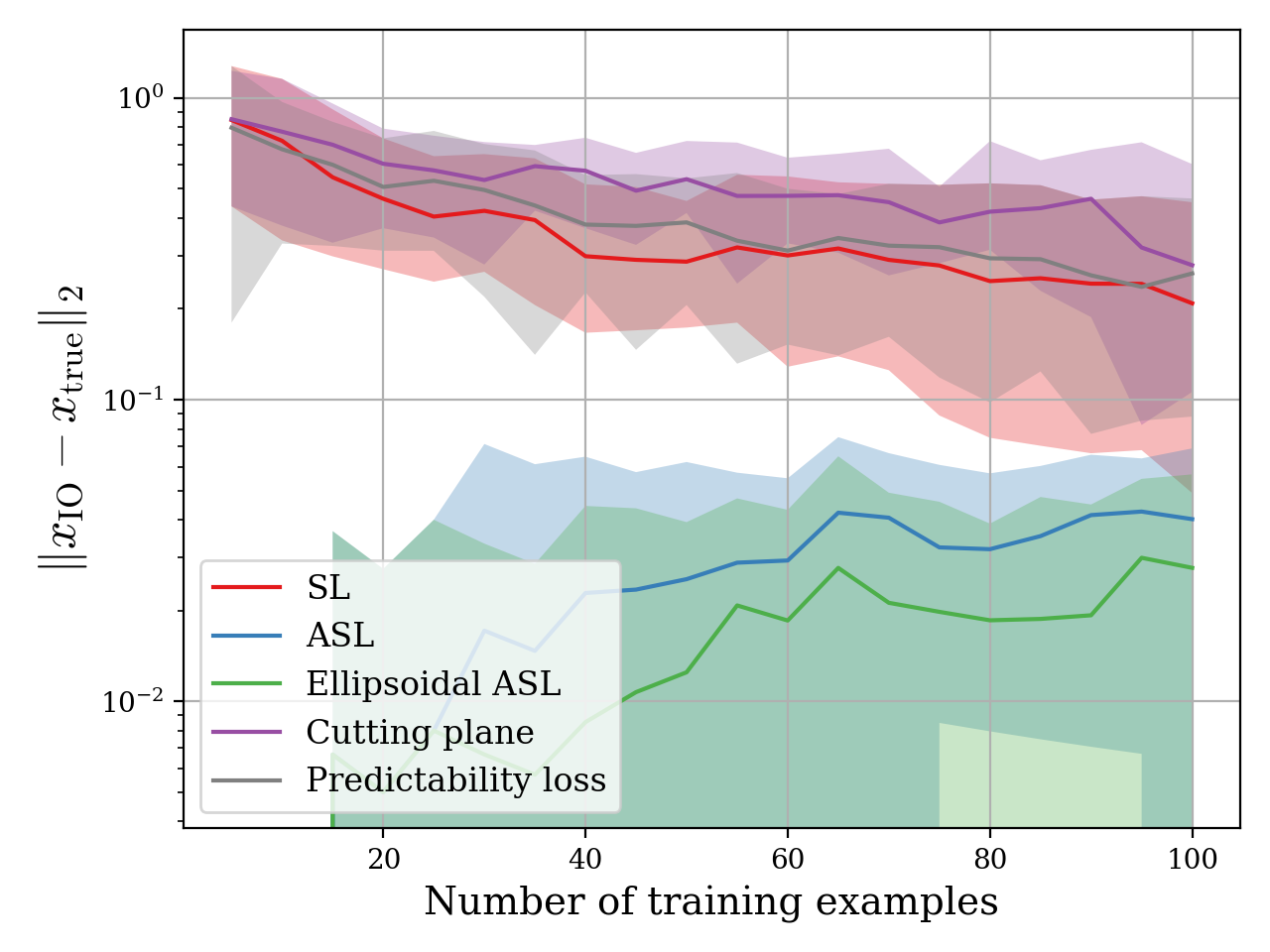}
\caption{In-sample results for inconsistent data scenario. Average error between the decision generated by $\theta_{\text{true}}$ and $\theta_{\text{IO}}$.}
\label{fig:inconsistent_in}
\end{figure}

\begin{figure}
\centering
\captionsetup[subfigure]{width=0.96\linewidth}%
    \begin{subfigure}[t]{0.38\linewidth}
        \includegraphics[width = \linewidth]{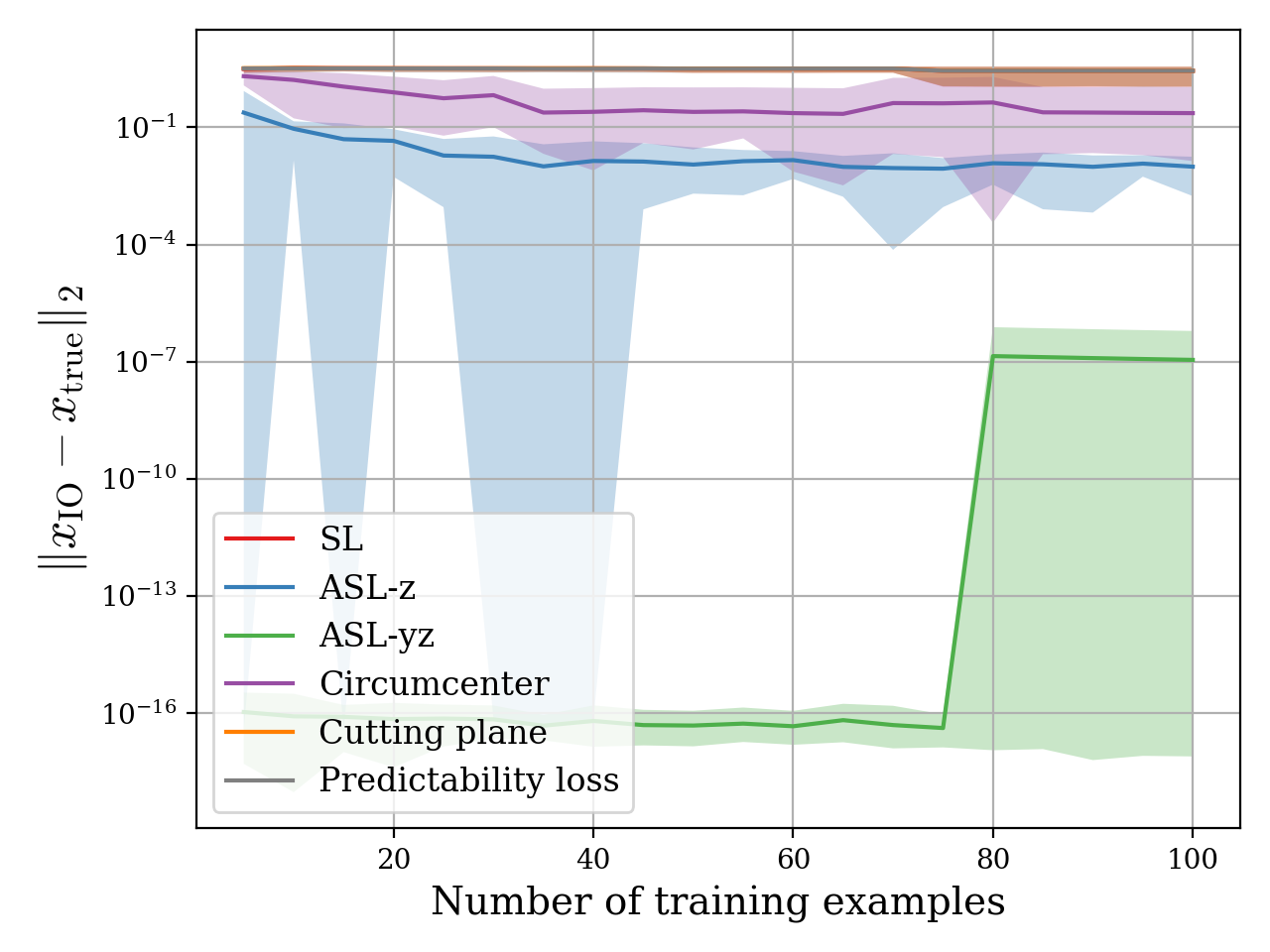}
        \caption{Average error between the decision generated by $\theta_{\text{true}}$ and $\theta_{\text{IO}}$.}
        \label{fig:MILP_x_diff_in}
    \end{subfigure}
    \begin{subfigure}[t]{0.38\linewidth}
        \includegraphics[width = \linewidth]{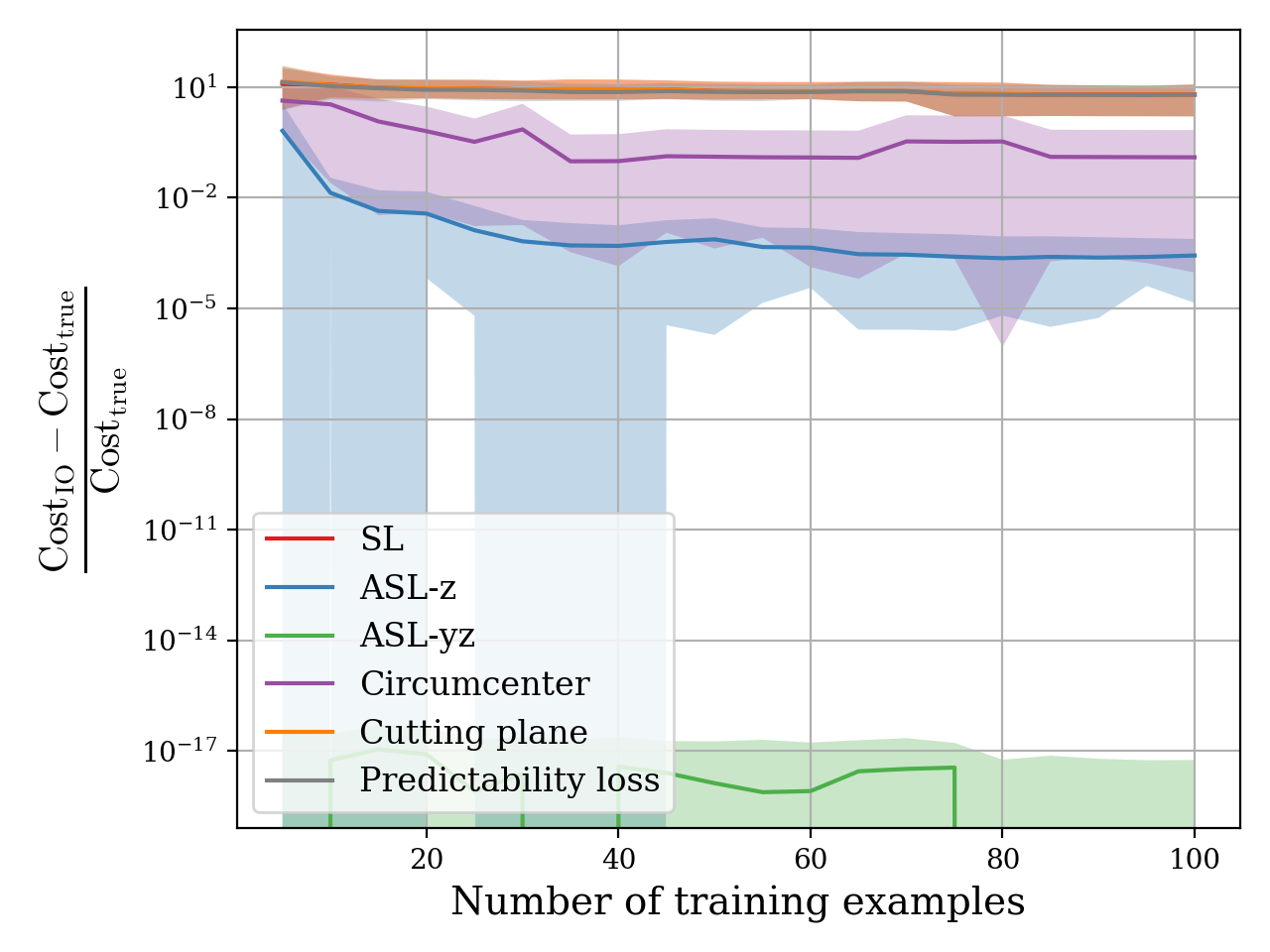}
        \caption{Relative difference between the cost of the decisions generated using $\theta_{\text{true}}$ and $\theta_{\text{IO}}$.}
        \label{fig:MILP_obj_diff_in}
    \end{subfigure}
\caption{In-sample results for the mixed-integer feasible set scenario.}
\label{fig:MILP_in}
\end{figure}

\begin{table}[h]
\centering
\begin{tabular}{c | c c c} 
\textbf{Approach} & ASL-z & ASL-yz & regression+classification \\
\hline
\textbf{$| y - y_\text{true} |$} & 34.27 & 26.06 & 25.03 \\
\hline
\textbf{$| z - z_\text{true} |$} & 24.16\% & 24.16\% & 23.48\%
\end{tabular}
\caption{In-sample average error in months of the TTR/DFST prediction and out-of-sample average percentage error in the prognostic of recurrent vs non-recurrent patients.}
\label{table:in_breast_cancer_results}
\end{table}

\subsection{IO results for SAMD}

Figure \ref{fig:FOM_all_out} shows the performance of the algorithms tested in Section \ref{sec:approximate_subgradients} in terms of the IO performance metrics. The discussion on the interpretation of the results of Figure \ref{fig:FOM_all_out} mirrors the one related to Figure \ref{fig:consistent_out} from the previous section, with the difference that now the x-axis of the figures refers to running time instead of the number of training examples. In these plots, we can see that the improvements in convergence speed of mirror descent updates, stochastic subgradients, and approximate subgradients are also reflected in the performance of the resulting solution for the IO problem.

\begin{figure}
\centering
\captionsetup[subfigure]{width=0.96\linewidth}%
    \begin{subfigure}[t]{0.32\linewidth}
        \includegraphics[width = \linewidth]{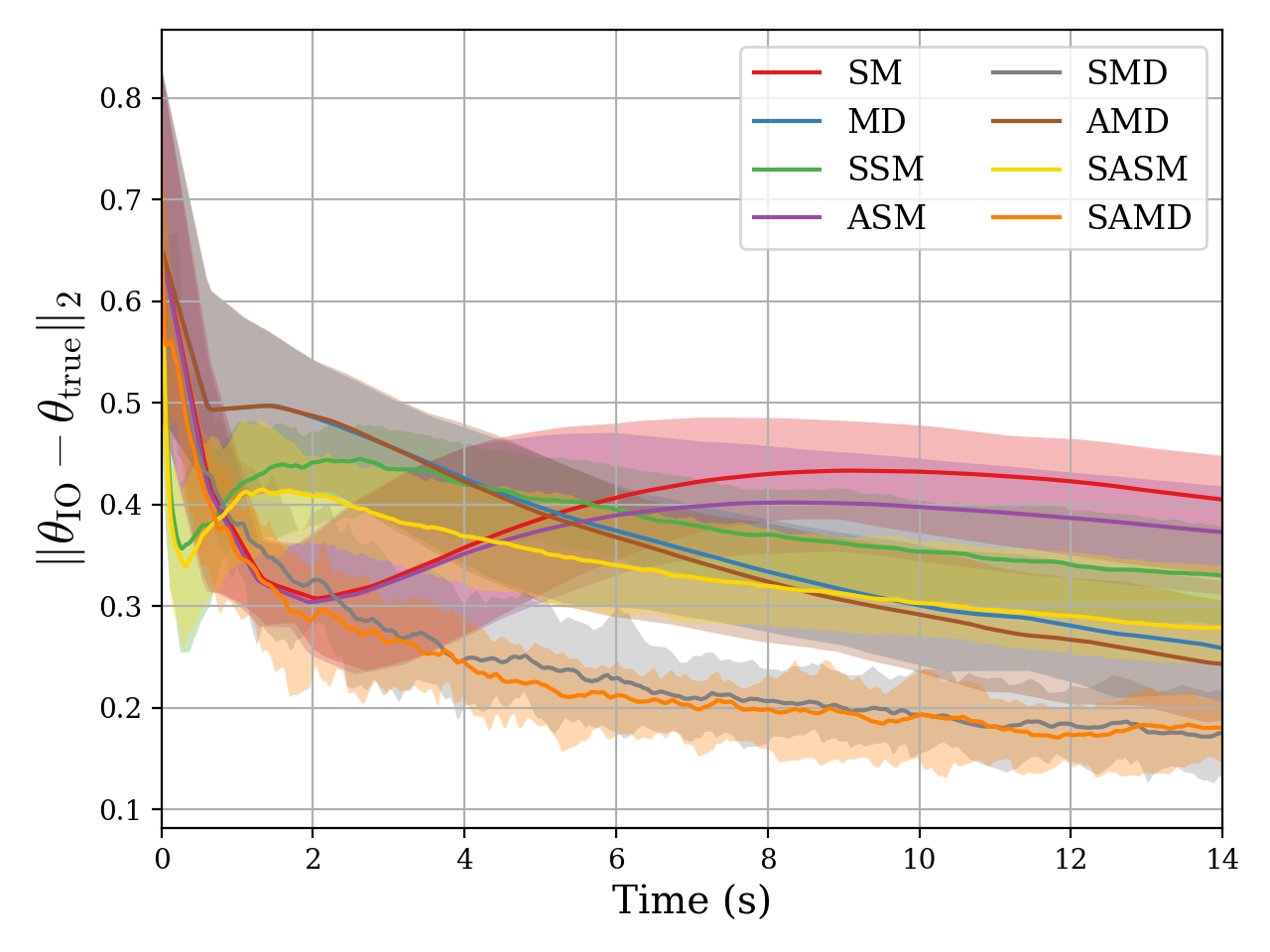}
        \caption{Difference between the true cost vector ($\theta_{\text{true}}$) and the one learned using IO ($\theta_{\text{IO}}$).}
        \label{fig:FOM_all_theta_diff}
    \end{subfigure}
    \begin{subfigure}[t]{0.32\linewidth}
        \includegraphics[width = \linewidth]{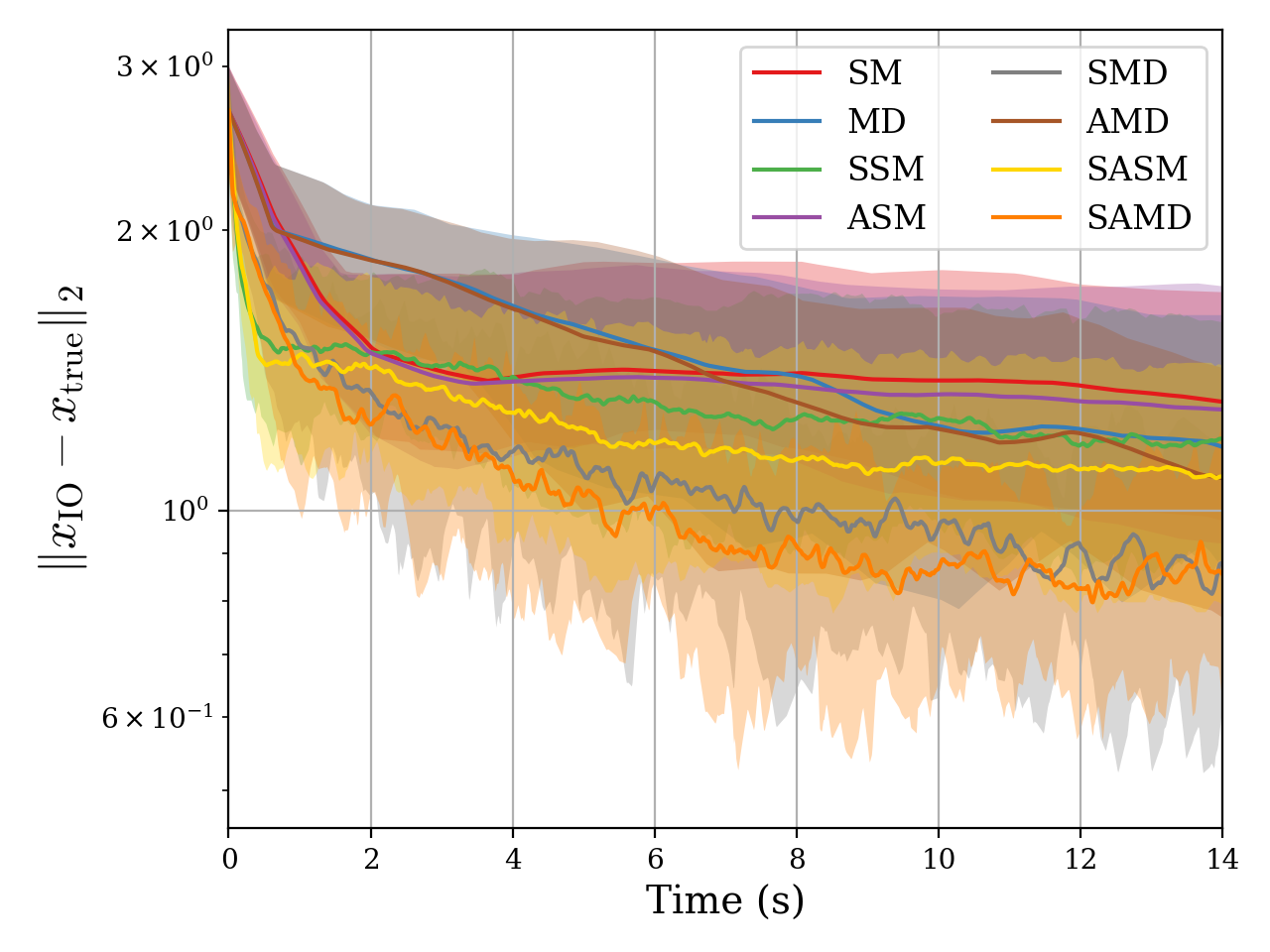}
        \caption{Average error between the decision generated by $\theta_{\text{true}}$ and $\theta_{\text{IO}}$.}
        \label{fig:FOM_all_x_diff_out}
    \end{subfigure}
    \begin{subfigure}[t]{0.32\linewidth}
        \includegraphics[width = \linewidth]{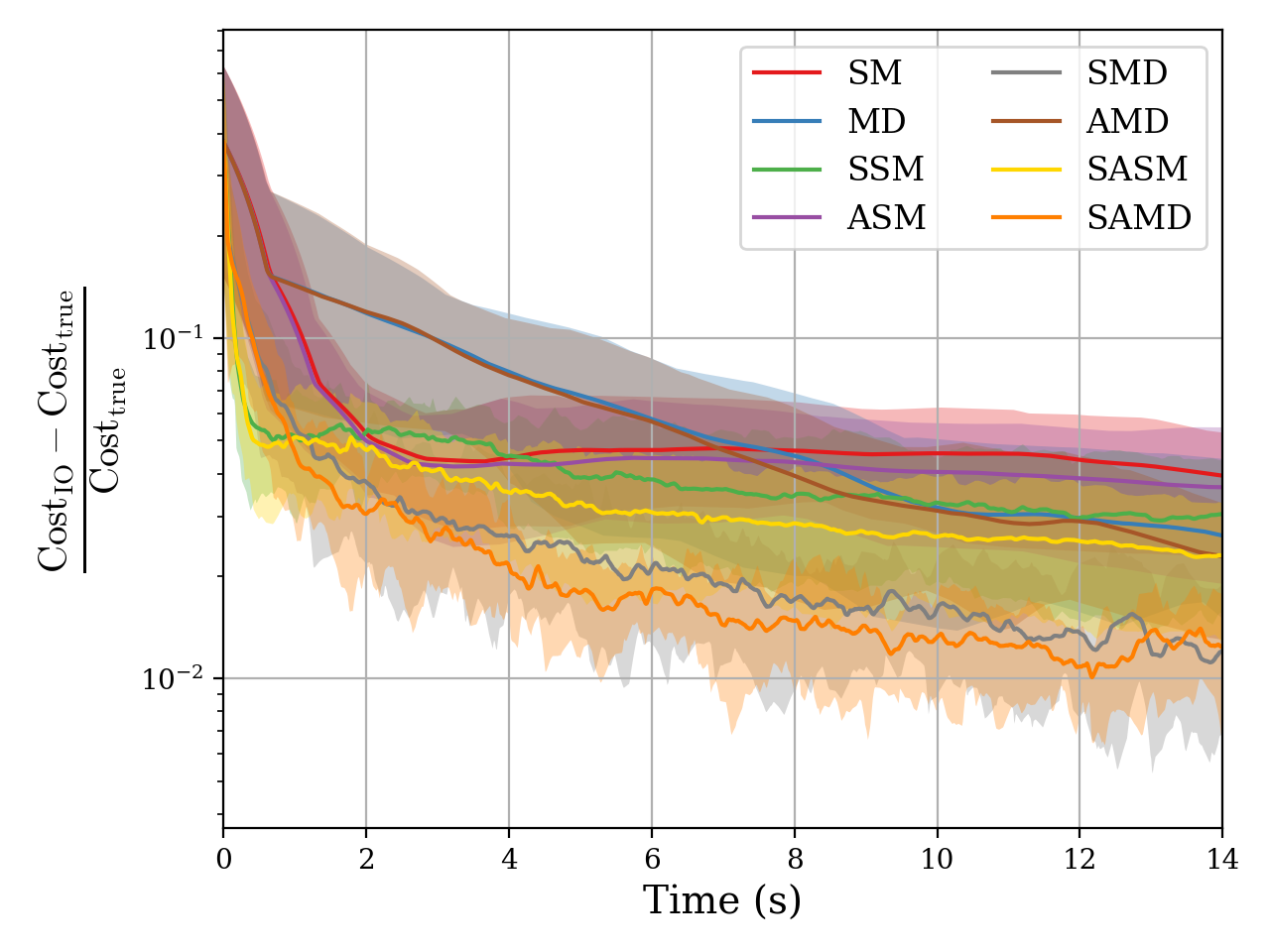}
        \caption{Relative difference between the cost of the decisions generated using $\theta_{\text{true}}$ and $\theta_{\text{IO}}$.}
        \label{fig:FOM_all_obj_diff_out}
    \end{subfigure}
\caption{Out-of-sample results using first-order algorithms.}
\label{fig:FOM_all_out}
\end{figure}

\begin{figure}
\centering
\captionsetup[subfigure]{width=0.96\linewidth}%
    \begin{subfigure}[t]{0.32\linewidth}
        \includegraphics[width = \linewidth]{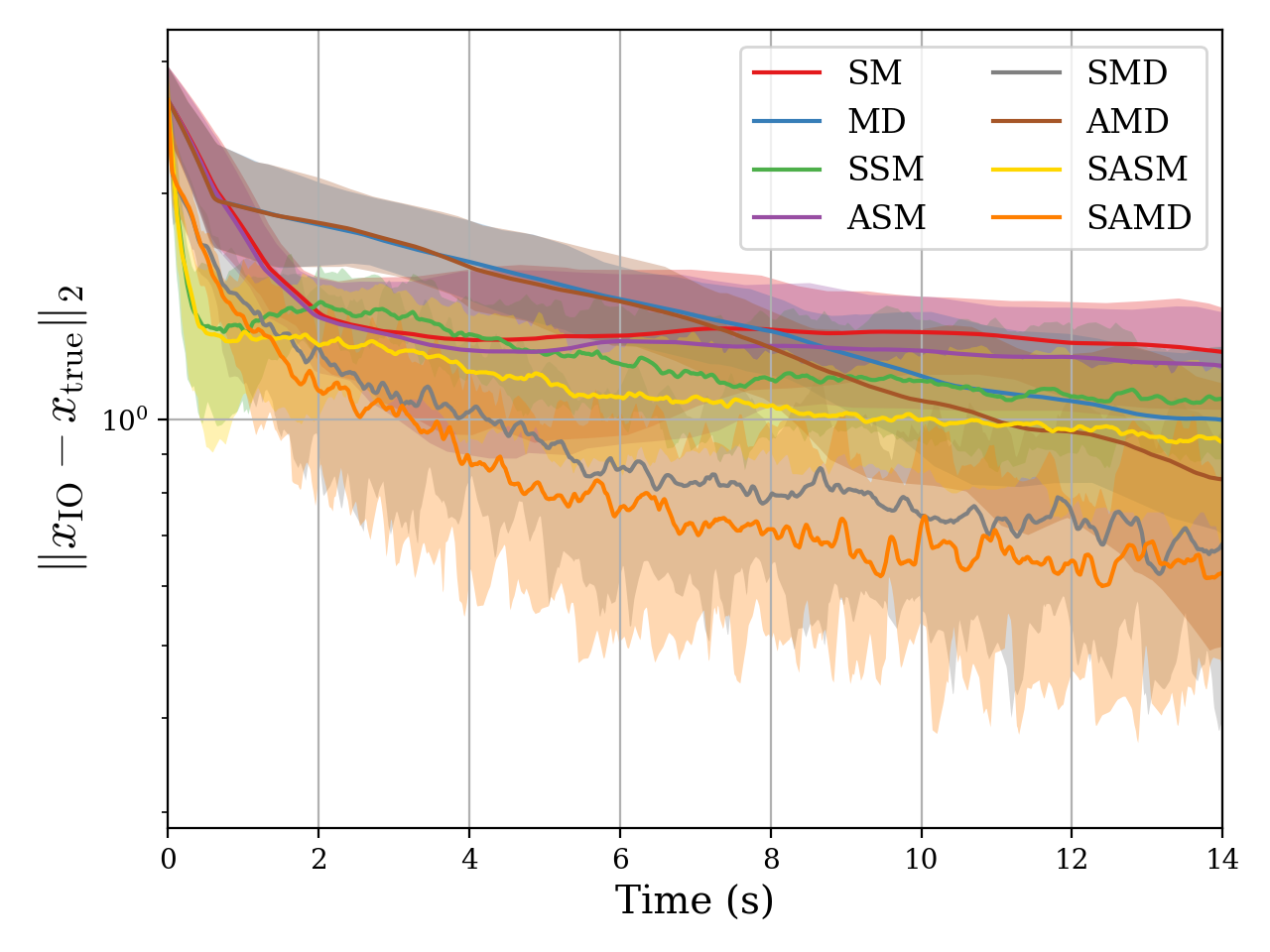}
        \caption{Average error between the decision generated by $\theta_{\text{true}}$ and $\theta_{\text{IO}}$.}
        \label{fig:FOM_all_x_diff_in}
    \end{subfigure}
    \begin{subfigure}[t]{0.32\linewidth}
        \includegraphics[width = \linewidth]{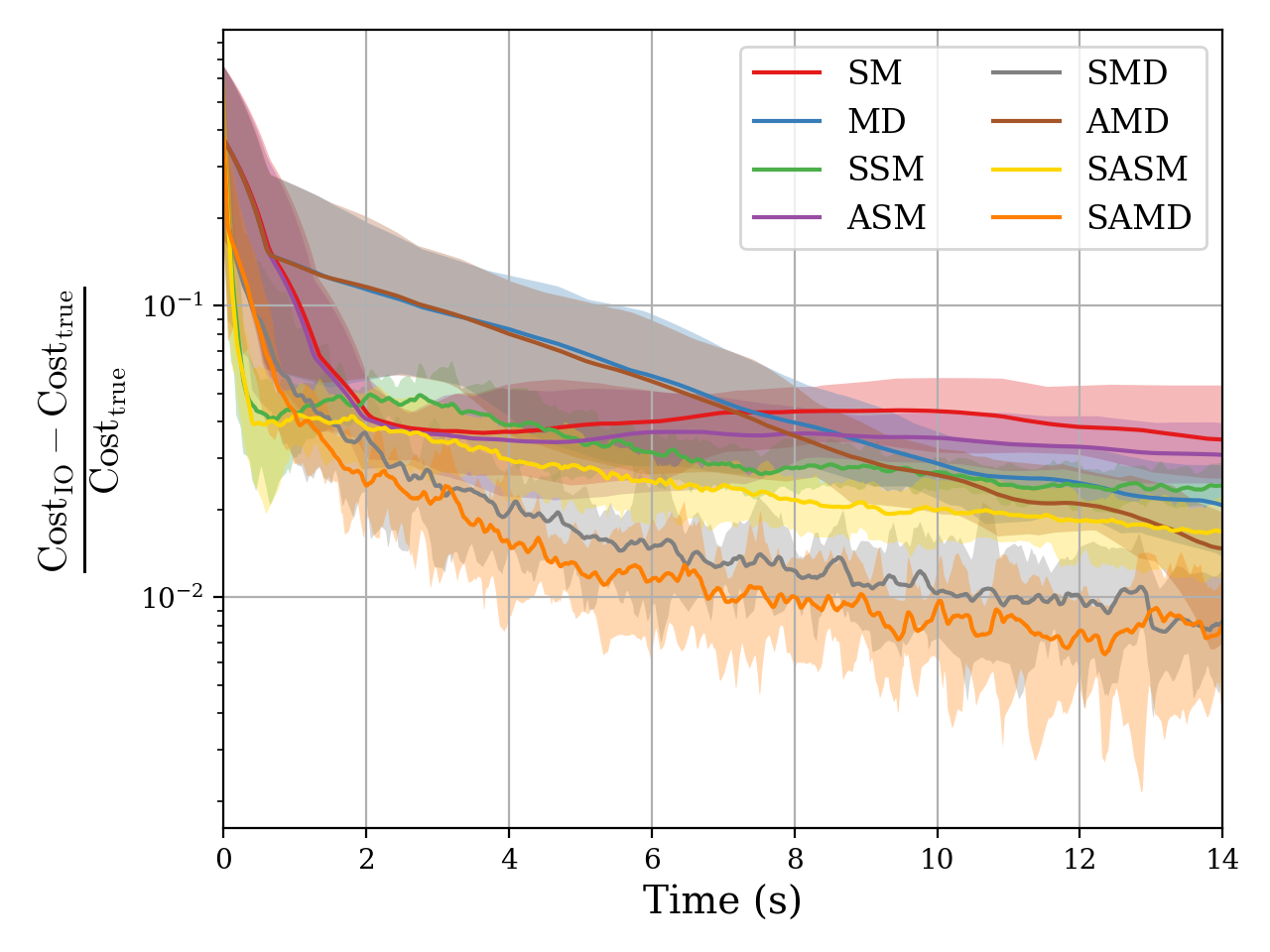}
        \caption{Relative difference between the cost of the decisions generated using $\theta_{\text{true}}$ and $\theta_{\text{IO}}$.}
        \label{fig:FOM_all_obj_diff_in}
    \end{subfigure}
\caption{In-sample results using iterative first-order algorithms.}
\label{fig:FOM_all_in}
\end{figure}

\bibliographystyle{plain}
\bibliography{references}

\end{document}

%% file: style.tex
\makeatletter
\def\@settitle{\begin{center}%
		\baselineskip14\p@\relax
		\normalfont\LARGE\scshape\bfseries
		\@title
	\end{center}%
}

\def\@setauthors{%
  \begingroup
  \def\thanks{\protect\thanks@warning}%
  \trivlist
  \centering\small \@topsep30\p@\relax
  \advance\@topsep by -\baselineskip
  \item\relax
  \author@andify\authors
  \def\\{\protect\linebreak}%
  \authors%
  \ifx\@empty\contribs
  \else
    ,\penalty-3 \space \@setcontribs
    \@closetoccontribs
  \fi
  \endtrivlist
  \endgroup
}

\makeatother

\makeatletter

\def\subsection{\@startsection{subsection}{2}%
	\z@{.5\linespacing\@plus.7\linespacing}{.5\linespacing}%
	{\normalfont\large\bfseries}}

\def\subsubsection{\@startsection{subsubsection}{3}%
	\z@{.5\linespacing\@plus.7\linespacing}{.5\linespacing}%
	{\normalfont\itshape}}

\usepackage[usenames, dvipsnames]{xcolor}
\definecolor{darkblue}{rgb}{0.0, 0.0, 0.45}

\usepackage[colorlinks	= true,
raiselinks	= true,
linkcolor	= darkblue, 
citecolor	= Mahogany,
urlcolor	= ForestGreen,
pdfauthor	= {Arman Sharifi Kolarijani},
pdftitle	= {},
pdfkeywords	= {},
pdfsubject	= {},
plainpages	= false]{hyperref}

\usepackage{dsfont,amssymb,amsmath, graphicx, multirow} 
\usepackage{amsfonts,dsfont,mathtools, mathrsfs,amsthm} 
\usepackage[amssymb, thickqspace]{SIunits}
\usepackage{algorithm,algorithmicx,fancyhdr}
\usepackage{makecell}

\allowdisplaybreaks
\date{\today}

%% file: commands.tex
\DeclareMathOperator*{\argmax}{argmax}
\DeclareMathOperator*{\argmin}{argmin}

\newcommand{\norm}[1]{\|#1\|}
\newcommand{\inner}[2]{\langle #1,#2 \rangle}

\newtheorem{theorem}{Theorem}[section]
\newtheorem{definition}[theorem]{Definition}
\newtheorem{lemma}[theorem]{Lemma}
\newtheorem{assumption}[theorem]{Assumption}
\newtheorem{remark}[theorem]{Remark}

\newtheorem{corollary}[theorem]{Corollary}

\newtheorem{example}[theorem]{Example}
\newtheorem{proposition}[theorem]{Proposition}

%% file: main.bbl
\begin{thebibliography}{10}

\bibitem{ahmed2005inverse}
Shabbir Ahmed and Yongpei Guan.
\newblock The inverse optimal value problem.
\newblock {\em Mathematical programming}, 102:91--110, 2005.

\bibitem{ahuja2001inverse}
Ravindra~K Ahuja and James~B Orlin.
\newblock Inverse optimization.
\newblock {\em Operations research}, 49(5):771--783, 2001.

\bibitem{akhtar2021learning}
Syed~Adnan Akhtar, Arman~Sharifi Kolarijani, and Peyman Mohajerin~Esfahani.
\newblock Learning for control: An inverse optimization approach.
\newblock {\em IEEE Control Systems Letters}, 2021.

\bibitem{allen2014linear}
Zeyuan Allen-Zhu and Lorenzo Orecchia.
\newblock Linear coupling: An ultimate unification of gradient and mirror
  descent.
\newblock {\em arXiv preprint arXiv:1407.1537}, 2014.

\bibitem{aswani2018inverse}
Anil Aswani, Zuo-Jun Shen, and Auyon Siddiq.
\newblock Inverse optimization with noisy data.
\newblock {\em Operations Research}, 66(3):870--892, 2018.

\bibitem{barmann2017emulating}
Andreas B{\"a}rmann, Sebastian Pokutta, and Oskar Schneider.
\newblock Emulating the expert: Inverse optimization through online learning.
\newblock In {\em International Conference on Machine Learning}, pages
  400--410. PMLR, 2017.

\bibitem{bertsekas2015convex}
Dimitri Bertsekas.
\newblock {\em Convex optimization algorithms}.
\newblock Athena Scientific, 2015.

\bibitem{bertsekas2008nonlinear}
Dimitri~P. Bertsekas.
\newblock {\em Nonlinear programming}.
\newblock Athena Scientific, 2008.

\bibitem{bertsimas2012inverse}
Dimitris Bertsimas, Vishal Gupta, and Ioannis~Ch Paschalidis.
\newblock Inverse optimization: A new perspective on the black-litterman model.
\newblock {\em Operations research}, 60(6):1389--1403, 2012.

\bibitem{bertsimas2015data}
Dimitris Bertsimas, Vishal Gupta, and Ioannis~Ch Paschalidis.
\newblock Data-driven estimation in equilibrium using inverse optimization.
\newblock {\em Mathematical Programming}, 153(2):595--633, 2015.

\bibitem{besbes2023contextual}
Omar Besbes, Yuri Fonseca, and Ilan Lobel.
\newblock Contextual inverse optimization: Offline and online learning.
\newblock {\em Operations Research}, 2023.

\bibitem{bodur2022inverse}
Merve Bodur, Timothy~CY Chan, and Ian~Yihang Zhu.
\newblock Inverse mixed integer optimization: Polyhedral insights and trust
  region methods.
\newblock {\em INFORMS Journal on Computing}, 34(3):1471--1488, 2022.

\bibitem{boyd2004convex}
Stephen Boyd and Lieven Vandenberghe.
\newblock {\em Convex optimization}.
\newblock Cambridge university press, 2004.

\bibitem{bubeck2015convex}
S.~Bubeck.
\newblock {\em Convex Optimization: Algorithms and Complexity}.
\newblock Foundations and Trends in Machine Learning. Now Publishers, 2015.

\bibitem{burton1992instance}
Didier Burton and Ph~L Toint.
\newblock On an instance of the inverse shortest paths problem.
\newblock {\em Mathematical programming}, 53:45--61, 1992.

\bibitem{chan2014generalized}
Timothy~CY Chan, Tim Craig, Taewoo Lee, and Michael~B Sharpe.
\newblock Generalized inverse multiobjective optimization with application to
  cancer therapy.
\newblock {\em Operations Research}, 62(3):680--695, 2014.

\bibitem{chan2022inverse}
Timothy~CY Chan, Maria Eberg, Katharina Forster, Claire Holloway, Luciano
  Ieraci, Yusuf Shalaby, and Nasrin Yousefi.
\newblock An inverse optimization approach to measuring clinical pathway
  concordance.
\newblock {\em Management Science}, 68(3):1882--1903, 2022.

\bibitem{chan2019inverse}
Timothy~CY Chan, Taewoo Lee, and Daria Terekhov.
\newblock Inverse optimization: Closed-form solutions, geometry, and goodness
  of fit.
\newblock {\em Management Science}, 65(3):1115--1135, 2019.

\bibitem{chan2021inverse}
Timothy~CY Chan, Rafid Mahmood, and Ian~Yihang Zhu.
\newblock Inverse optimization: Theory and applications.
\newblock {\em arXiv preprint arXiv:2109.03920}, 2021.

\bibitem{chen2020online}
Violet~Xinying Chen and Fatma K{\i}l{\i}n{\c{c}}-Karzan.
\newblock Online convex optimization perspective for learning from dynamically
  revealed preferences.
\newblock {\em arXiv preprint arXiv:2008.10460}, 2020.

\bibitem{chow2012inverse}
Joseph~YJ Chow and Will~W Recker.
\newblock Inverse optimization with endogenous arrival time constraints to
  calibrate the household activity pattern problem.
\newblock {\em Transportation Research Part B: Methodological}, 46(3):463--479,
  2012.

\bibitem{conforti2014integer}
M.~Conforti, G.~Cornu{\'e}jols, and G.~Zambelli.
\newblock {\em Integer Programming}.
\newblock Graduate Texts in Mathematics. Springer International Publishing,
  2014.

\bibitem{Dua:2019}
Dheeru Dua and Casey Graff.
\newblock {UCI} machine learning repository, 2017.

\bibitem{el2019generalization}
Othman El~Balghiti, Adam~N Elmachtoub, Paul Grigas, and Ambuj Tewari.
\newblock Generalization bounds in the predict-then-optimize framework.
\newblock {\em Advances in neural information processing systems}, 2019.

\bibitem{elmachtoub2022smart}
Adam~N. Elmachtoub and Paul Grigas.
\newblock Smart “predict, then optimize”.
\newblock {\em Management Science}, 68(1):9--26, 2022.

\bibitem{farago2003inverse}
Andr{\'a}s Farag{\'o}, {\'A}ron Szentesi, and Bal{\'a}zs Szviatovszki.
\newblock Inverse optimization in high-speed networks.
\newblock {\em Discrete Applied Mathematics}, 129(1):83--98, 2003.

\bibitem{ghobadi2018robust}
Kimia Ghobadi, Taewoo Lee, Houra Mahmoudzadeh, and Daria Terekhov.
\newblock Robust inverse optimization.
\newblock {\em Operations Research Letters}, 46(3):339--344, 2018.

\bibitem{ghobadi2021inferring}
Kimia Ghobadi and Houra Mahmoudzadeh.
\newblock Inferring linear feasible regions using inverse optimization.
\newblock {\em European Journal of Operational Research}, 290(3):829--843,
  2021.

\bibitem{heuberger2004inverse}
Clemens Heuberger.
\newblock Inverse combinatorial optimization: A survey on problems, methods,
  and results.
\newblock {\em Journal of combinatorial optimization}, 8:329--361, 2004.

\bibitem{hoffman2010integral}
Alan~J Hoffman and Joseph~B Kruskal.
\newblock Integral boundary points of convex polyhedra.
\newblock {\em 50 Years of Integer Programming 1958-2008: From the Early Years
  to the State-of-the-Art}, pages 49--76, 2010.

\bibitem{iyengar2005inverse}
Garud Iyengar and Wanmo Kang.
\newblock Inverse conic programming with applications.
\newblock {\em Operations Research Letters}, 33(3):319--330, 2005.

\bibitem{joachims2009cutting}
Thorsten Joachims, Thomas Finley, and Chun-Nam~John Yu.
\newblock Cutting-plane training of structural {SVM}s.
\newblock {\em Machine learning}, 77(1):27--59, 2009.

\bibitem{juditsky2011first_i}
Anatoli Juditsky and Arkadi Nemirovski.
\newblock First order methods for nonsmooth convex large-scale optimization, i:
  general purpose methods.
\newblock {\em Optimization for Machine Learning}, 30(9):121--148, 2011.

\bibitem{keshavarz2011imputing}
Arezou Keshavarz, Yang Wang, and Stephen Boyd.
\newblock Imputing a convex objective function.
\newblock In {\em 2011 IEEE International Symposium on Intelligent Control},
  pages 613--619, 2011.

\bibitem{lacoste2012simpler}
Simon Lacoste-Julien, Mark Schmidt, and Francis Bach.
\newblock A simpler approach to obtaining an o (1/t) convergence rate for the
  projected stochastic subgradient method.
\newblock {\em arXiv preprint arXiv:1212.2002}, 2012.

\bibitem{lee2001ssvm}
Yuh-Jye Lee and Olvi~L Mangasarian.
\newblock Ssvm: A smooth support vector machine for classification.
\newblock {\em Computational optimization and Applications}, 20(1):5--22, 2001.

\bibitem{lu2018relatively}
Haihao Lu, Robert~M Freund, and Yurii Nesterov.
\newblock Relatively smooth convex optimization by first-order methods, and
  applications.
\newblock {\em SIAM Journal on Optimization}, 28(1):333--354, 2018.

\bibitem{esfahani2018data}
Peyman Mohajerin~Esfahani, Soroosh Shafieezadeh-Abadeh, Grani~A Hanasusanto,
  and Daniel Kuhn.
\newblock Data-driven inverse optimization with imperfect information.
\newblock {\em Mathematical Programming}, 167(1):191--234, 2018.

\bibitem{more1993generalizations}
Jorge~J Mor{\'e}.
\newblock Generalizations of the trust region problem.
\newblock {\em Optimization methods and Software}, 2(3-4):189--209, 1993.

\bibitem{nemirovsky1996interior}
Arkadi Nemirovski.
\newblock Lecture notes: Interior point polynomial methods in convex
  programming, 1996.
\newblock URL: \url{https://www2.isye.gatech.edu/~nemirovs/Lect_IPM.pdf}. Last
  visited on 2022/06/23.

\bibitem{nemirovski2004prox}
Arkadi Nemirovski.
\newblock Prox-method with rate of convergence {O}(1/t) for variational
  inequalities with lipschitz continuous monotone operators and smooth
  convex-concave saddle point problems.
\newblock {\em SIAM Journal on Optimization}, 15(1):229--251, 2004.

\bibitem{nesterov2007dual}
Yurii Nesterov.
\newblock Dual extrapolation and its applications to solving variational
  inequalities and related problems.
\newblock {\em Mathematical Programming}, 109(2-3):319--344, 2007.

\bibitem{nowozin2011structured}
Sebastian Nowozin, Christoph~H Lampert, et~al.
\newblock Structured learning and prediction in computer vision.
\newblock {\em Foundations and Trends{\textregistered} in Computer Graphics and
  Vision}, 6(3--4):185--365, 2011.

\bibitem{orabona2019modern}
Francesco Orabona.
\newblock A modern introduction to online learning.
\newblock {\em arXiv preprint arXiv:1912.13213}, 2019.

\bibitem{pong2014generalized}
Ting~Kei Pong and Henry Wolkowicz.
\newblock The generalized trust region subproblem.
\newblock {\em Computational optimization and applications}, 58(2):273--322,
  2014.

\bibitem{ratliff2006maximum}
Nathan~D. Ratliff, J.~Andrew Bagnell, and Martin~A. Zinkevich.
\newblock Maximum margin planning.
\newblock In {\em Proceedings of the 23rd international conference on Machine
  learning}, pages 729--736, 2006.

\bibitem{ratliff2007approximate}
Nathan~D Ratliff, J~Andrew Bagnell, and Martin~A Zinkevich.
\newblock (approximate) subgradient methods for structured prediction.
\newblock In {\em Artificial Intelligence and Statistics}, pages 380--387.
  PMLR, 2007.

\bibitem{saez2017short}
Javier Saez-Gallego and Juan~M Morales.
\newblock Short-term forecasting of price-responsive loads using inverse
  optimization.
\newblock {\em IEEE Transactions on Smart Grid}, 9(5):4805--4814, 2017.

\bibitem{sahni1974computationally}
Sartaj Sahni.
\newblock Computationally related problems.
\newblock {\em SIAM Journal on computing}, 3(4):262--279, 1974.

\bibitem{schaefer2009inverse}
Andrew~J Schaefer.
\newblock Inverse integer programming.
\newblock {\em Optimization Letters}, 3:483--489, 2009.

\bibitem{shor1985minimization}
Naum~Zuselevich Shor.
\newblock {\em Minimization methods for non-differentiable functions}.
\newblock Springer Series in Computational Mathematics. Springer, 1985.

\bibitem{street1995inductive}
W~Nick Street, Olvi~L Mangasarian, and William~H Wolberg.
\newblock An inductive learning approach to prognostic prediction.
\newblock In {\em Machine Learning Proceedings 1995}, pages 522--530. Elsevier,
  1995.

\bibitem{taskar2005learning}
Ben Taskar, Vassil Chatalbashev, Daphne Koller, and Carlos Guestrin.
\newblock Learning structured prediction models: A large margin approach.
\newblock In {\em Proceedings of the 22nd international conference on Machine
  learning}, pages 896--903, 2005.

\bibitem{taskar2006structured}
Ben Taskar, Simon Lacoste-Julien, Michael~I. Jordan, Kristin~P. Bennett, and
  Emilio Parrado-Hern{\'a}ndez.
\newblock Structured prediction, dual extragradient and bregman projections.
\newblock {\em Journal of Machine Learning Research}, 7(7), 2006.

\bibitem{tacskesen2022semi}
Bahar Ta{\c{s}}kesen, Soroosh Shafieezadeh-Abadeh, and Daniel Kuhn.
\newblock Semi-discrete optimal transport: Hardness, regularization and
  numerical solution.
\newblock {\em Mathematical Programming}, pages 1--74, 2022.

\bibitem{teo2010bundle}
Choon~Hui Teo, SVN Vishwanathan, Alex Smola, and Quoc~V Le.
\newblock Bundle methods for regularized risk minimization.
\newblock {\em Journal of Machine Learning Research}, 11(1), 2010.

\bibitem{tibshirani1996regression}
Robert Tibshirani.
\newblock Regression shrinkage and selection via the lasso.
\newblock {\em Journal of the Royal Statistical Society: Series B
  (Methodological)}, 58(1):267--288, 1996.

\bibitem{tsochantaridis2005large}
Ioannis Tsochantaridis, Thorsten Joachims, Thomas Hofmann, Yasemin Altun, and
  Yoram Singer.
\newblock Large margin methods for structured and interdependent output
  variables.
\newblock {\em Journal of machine learning research}, 6(9), 2005.

\bibitem{wang2022generalized}
Alex~L Wang and Fatma K{\i}l{\i}n{\c{c}}-Karzan.
\newblock The generalized trust region subproblem: solution complexity and
  convex hull results.
\newblock {\em Mathematical Programming}, 191(2):445--486, 2022.

\bibitem{wang2009cutting}
Lizhi Wang.
\newblock Cutting plane algorithms for the inverse mixed integer linear
  programming problem.
\newblock {\em Operations research letters}, 37(2):114--116, 2009.

\bibitem{zafiropoulos2006support}
Elias Zafiropoulos, Ilias Maglogiannis, and Ioannis Anagnostopoulos.
\newblock A support vector machine approach to breast cancer diagnosis and
  prognosis.
\newblock In {\em IFIP international conference on artificial intelligence
  applications and innovations}, pages 500--507. Springer, 2006.

\bibitem{zattoniscroccaro2023invopt}
Pedro Zattoni~Scroccaro.
\newblock Inv{O}pt: An open-source {P}ython package to solve {I}nverse
  {O}ptimization problems.
\newblock \url{https://github.com/pedroszattoni/invopt}, 2023.

\end{thebibliography}
